\documentclass{amsart}

\usepackage{times,amsthm}
\usepackage{amsmath}
\usepackage{amssymb}
\usepackage{amsfonts}
\usepackage{mathrsfs}
\usepackage{latexsym}
\usepackage[all]{xy}
\SelectTips {cm}{}
\usepackage{graphicx}
\usepackage{comment}
\usepackage{appendix}
\usepackage{amscd}
\usepackage{a4wide}
\usepackage{mathrsfs}
\usepackage{pdfsync}
\usepackage{color}
\definecolor{Chocolat}{rgb}{0.36, 0.2, 0.09}
\definecolor{BleuTresFonce}{rgb}{0.215, 0.215, 0.36}

\usepackage[colorlinks,final,backref=page,hyperindex]{hyperref}
\hypersetup{citecolor=BleuTresFonce, linkcolor=Chocolat}

\theoremstyle{plain}
\newtheorem{prop}{Proposition}
\newtheorem{theo}[prop]{Theorem}
\newtheorem{lemm}[prop]{Lemma}

\theoremstyle{definition}
\newtheorem{ex}{\sc Exercise}
\newtheorem*{defi}{\sc Definition}
\newtheorem*{rema}{\sc Remark}
\newtheorem*{exam}{\sc Example}
\newtheorem*{exams}{\sc Examples}

\newtheorem*{ques}{\sc Questions}
\newtheorem*{appl}{\sc Application}

\DeclareFontFamily{U}{shuffle}{} 
\DeclareFontShape{U}{shuffle}{m}{n}{<5-8>shuffle7  <8->shuffle10}{}
\DeclareSymbolFont{Shuffle}{U}{shuffle}{m}{n}
\DeclareMathSymbol\shuffle{\mathbin}{Shuffle}{"001} 
\DeclareMathSymbol\cshuffle{\mathbin}{Shuffle}{"002}

\def\I{\mathrm{ I }}
\def\DDD{D}

\def\Pac{\Po^{\ac}}

\newcommand{\CC}{\mathbb{C}}

\def\CCC{{\mathcal{C}}}
\def\TTT{{\mathcal{T}}}
\newcommand{\Ai}{A_\infty}

\newcommand{\Tw}{\mathrm{Tw}}

\newcommand{\qi}{\xrightarrow{\sim}}
\newcommand{\g}{\mathfrak{g}}

\newcommand{\KK}{\mathbb{K}}
\newcommand{\ZZ}{\mathbb{Z}}

\newcommand{\End}{\mathrm{End}}

\newcommand{\B}{\mathrm{B}}

\newcommand{\NN}{\mathbb{N}}

\newcommand{\Sy}{\mathbb{S}}
\newcommand{\CA}{\mathcal{A}}

\newcommand{\Po}{\mathcal{P}}

\newcommand{\ac}{\scriptstyle \text{\rm !`}}

\newcommand{\Id}{\mathrm{Id}}
\newcommand{\Qo}{\mathcal{Q}}

\newcommand{\Hom}{\mathrm{Hom}}

\newcommand{\sgn}{{sgn}}

\newcommand{\epi}{\twoheadrightarrow}
\newcommand{\mono}{\rightarrowtail}

\newcommand{\Y}{\vcenter{\xymatrix@M=0pt@R=6pt@C=6pt{
\ar@{-}[dr] &  &\ar@{-}[dl]  \\
 &\ar@{-}[d] &  \\  & &}}}
\newcommand{\YY}{\vcenter{\xymatrix@M=0pt@R=6pt@C=6pt{
\ar@{-}[dr] &  &\ar@2{-}[dl]  \\
 &\ar@2{-}[d] &  \\  & &}}}
\newcommand{\YYY}{\vcenter{\xymatrix@M=0pt@R=6pt@C=6pt{
\ar@{-}[dr] &  &\ar@3{-}[dl]  \\
 &\ar@3{-}[d] &  \\  & &}}}
\newcommand{\cop}{\vcenter{\xymatrix@M=0pt@R=6pt@C=6pt{
 & \ar@{-}[d] & \\
 &\ar@{-}[dr] \ar@{-}[dl] &  \\  & &}}}

\newcommand{\copL}{\xymatrix@M=0pt@R=6pt@C=6pt{
 & \ar@{-}[d] & \\
 &\ar@{-}[dr] \ar@{-}[dl] &  \\  & &\\  & &\\  & &}}
\newcommand{\YL}{\vcenter{\xymatrix@M=0pt@R=6pt@C=6pt{
\ar@{-}[dr] &  &\ar@{-}[dl]  \\
 &\ar@{-}[d] &   \\  & &\\  & &\\  & &}}}
\newcommand{\YYL}{\vcenter{\xymatrix@M=0pt@R=6pt@C=6pt{
\ar@{-}[dr] &  &\ar@2{-}[dl]  \\
 &\ar@2{-}[d] &  \\  & &\\  & &\\  & &}}}
\newcommand{\LYY}{\vcenter{\xymatrix@M=0pt@R=6pt@C=6pt{
\ar@{-}[dr] &  &\ar@2{-}[dl]  \\
 &\ar@2{-}[d] &  \\  & &\\  & &\\  & &}}}

\newcommand{\XX}{\vcenter{\xymatrix@M=0pt@R=6pt@C=6pt{\ar@{-}[ddrr]&&\ar@{-}[ddll] \\ && \\ &&   }}}

\newcommand{\Ta}{\vcenter{\xymatrix@M=0pt@R=6pt@C=6pt{ \ar@{-}[dddrrr] && \ar@{-}[dl] &&  \\
&&& \ar@{-}[dl]  &  \\ &&&&  \ar@{-}[dl]  \\&&&  \ar@{-}[d] &
\\&&&& }}}
\newcommand{\Tb}{\vcenter{\xymatrix@M=0pt@R=6pt@C=6pt{  & \ar@{-}[dr]&&\ar@{-}[dl] \\
\ar@{-}[dr]&&\ar@{-}[dl]& \\&\ar@{-}[dr]&&\ar@{-}[dl]
\\&&\ar@{-}[d]& \\&&& }}}
\newcommand{\Tc}{\vcenter{\xymatrix@M=0pt@R=6pt@C=6pt{   \ar@{-}[dr]&&\ar@{-}[dl]& \\
&\ar@{-}[dr]&& \ar@{-}[dl] \\\ar@{-}[dr]&&\ar@{-}[dl]&
\\&\ar@{-}[d]&& \\&&& }}}
\newcommand{\Td}{\vcenter{\xymatrix@M=0pt@R=6pt@C=6pt{ && \ar@{-}[dr]&&\ar@{-}[dddlll] \\
 &\ar@{-}[dr]&&& \\ \ar@{-}[dr]&&&& \\& \ar@{-}[d]&&& \\&&&&  }}}
\newcommand{\Te}{\vcenter{\xymatrix@R=3pt@C=3pt{\ar@{-}[drdr] &&\ar@{-}[dl]  *=0{}
\ar@{-}[dr]&& \ar@{-}[ddll] \\ &&& *=0{}& \\&& *=0{} \ar@{-}[d]&&
\\&&&& }}}
\newcommand{\TaC}{\vcenter{\xymatrix@M=0pt@R=6pt@C=6pt{ \ar@{-}[ddddddrrrrrr] && \ar@{-}[dl] && && \\
&&& \ar@{-}[dl]  &&&  \\ &&&&  \ar@{-}[dl]&&  \\&&& &&&
\\&&&&\ar@{-}[dl]&& \\&&&&&\ar@{-}[dl]&\\&&&&&& }}}
\newcommand{\TreeL}{\vcenter{\xymatrix@M=0pt@R=5pt@C=5pt{ \ar@{-}[dr] &
&\ar@{-}[dl] & &  \\
& \ar@{-}[dr] & &\ar@{-}[dl]  & \\
& &\ar@{-}[d] & & \\
& & \\ & & }}}
\newcommand{\TreeR}{\vcenter{\xymatrix@M=0pt@R=5pt@C=5pt{
 & &\ar@{-}[dr] & & \ar@{-}[dl]  \\
& \ar@{-}[dr] & &\ar@{-}[dl]  & \\
& &\ar@{-}[d] & & \\
& & \\ & & }}}

\long\def\symbolfootnote[#1]#2{\begingroup%
\def\thefootnote{\fnsymbol{footnote}}\footnote[#1]{#2}\endgroup}

\title{Algebra\ +\ Homotopy\ =\ Operad}

\author{Bruno Vallette}

\address{
Laboratoire J.A. Dieudonn\'e \\
Universit\'e de Nice Sophia-Antipolis \\
Parc Valrose \\ 06108 Nice
Cedex 02 \\ France
}
\email{brunov@unice.fr}

\begin{document}

\dedicatory{``If I could only understand the beautiful consequence following from the concise proposition $d^2=0$.'' \\  \qquad\qquad\qquad\qquad\qquad\qquad\qquad\qquad\qquad\qquad\qquad\qquad\qquad\qquad\qquad\qquad\qquad\qquad \emph{Henri Cartan}}

\maketitle

\begin{abstract}
This survey provides an elementary introduction  to operads and to their applications in homotopical algebra. 
The aim is to explain how the notion of an operad was prompted by the necessity to have an algebraic object which encodes higher homotopies.  We try to show how universal this theory is by giving many applications in Algebra, Geometry, Topology, and Mathematical Physics. 
(This text is  accessible to any student knowing what  tensor products, chain complexes, and categories are.)
\end{abstract}

\tableofcontents

\section*{Introduction}

Galois explained to us that operations acting on the solutions of algebraic equations are mathematical objects as well. The notion of an \emph{operad} was created in order to have a well defined mathematical object which encodes ``operations''. Its name is a {porte-manteau} word, coming from the contraction of the words ``operations'' and ``monad'', because an operad can be defined as a monad encoding operations. 
%(We will not mention nor use that word anymore through the text). 
The introduction of this notion was prompted in the 60's, by the necessity of working with higher operations made up of  higher homotopies appearing in algebraic topology. 

\medskip

\emph{Algebra} is the study of algebraic structures with respect to \emph{isomorphisms}. Given two isomorphic vector spaces and one algebra structure on one of them, one can always define, by means of transfer, an algebra structure on the other space such that these two algebra structures become isomorphic. 

\emph{Homotopical algebra} is the study of algebraic structures with respect to \emph{quasi-isomorphisms}, i.e. morphisms of chain complexes which induce isomorphisms in homology only. The main difference with the previous situation is that quasi-isomorphisms of chain complexes are not invertible in general. However, given two quasi-isomorphic chain complexes and one algebra structure on one of them, one can always define, on the other chain complex, an algebra structure. But, in this case, the relations are satisfied only \emph{up to homotopy}. In the end, these two algebra structures become \emph{quasi-isomorphic} in some way. This last result is called the \emph{Homotopy Transfer Theorem}, or \emph{HTT} for short. 

\medskip

In the first cases like associative algebras, commutative algebras, and Lie algebras, this kind of result can be proved by hand. However, in general, the transferred algebra structure yields a coherent system of higher homotopies, whose complexity increases quickly with the number of  the initial  operations.  Using operad theory, for instance Koszul duality theory, one can prove the HTT with explicit formulae. In this way, one recovers, and thus unifies, many classical objects appearing in Algebra, Geometry, Topology and, Mathematical Physics  like spectral sequences, Massey products, or Feynman diagrams, for instance. \\

This text does not pretend to be exhaustive, nor to serve as a faithful reference to the existing literature. Its only aim is to give a gentle introduction  to the ideas of this field of mathematics. We would like to share here one  point of view  on the subject, from ``student to student''.  It includes many figures and exercises to help the learning reader in its journey. To ease the reading, we skipped many technical details, which can be found in the  book \cite{LodayVallette10}.   \\

\textsc{Convention.} 
In this text, a chain complex $(V, d)$ is usually a graded module $V:=\lbrace V_n \rbrace_{n\in \ZZ}$ equipped with a degree $-1$ map $d$ (homological convention), which squares to zero. For the simplicity of the presentation, we always work over a field $\KK$ of characteristic $0$, even if some exposed results admit generalizations beyond that case. 

\section{When Algebra meets Homotopy}\label{Sec:Alg+Homo}

In this section, we treat the mother example of higher algebras:  $A_\infty$-algebras. We show how this notion appears naturally when one tries to transfer the structure of an associative algebra through  homotopy data. We provide an elementary but extensive study of its homotopy properties (Homotopy Transfer Theorem, $A_\infty$-morphism, Massey products and homotopy category). 

\subsection{Homotopy data and algebraic data}
Let us first consider  the following \emph{homotopy data} of chain complexes: 
\begin{eqnarray*}
&\xymatrix{     *{ \quad \ \  \quad (A, d_A)\ } \ar@(dl,ul)[]^{h}\ \ar@<0.5ex>[r]^{p} & *{\
(H,d_H)\quad \ \  \ \quad }  \ar@<0.5ex>[l]^{i}}&\\
& \Id_A-i p =d_A  h+ h  d_A \ ,
\end{eqnarray*}
where $i$ and $p$ are morphisms of chain complexes and where $h$ is a degree $+1$ map. It is called a \emph{homotopy retract}, when the map $i$ is a quasi-isomorphism, i.e. when it realizes an isomorphism in homology. If moreover $p i=\Id_H$, then it is called a \emph{deformation retract}. 

\begin{ex}
Since we work over a field, show that any quasi-isomorphism $i$ extends to a homotopy retract. 
(Such a statement holds over $\ZZ$ when all the $\ZZ$-modules are free.)

\texttt{Hint}. Use the same kind of decomposition of chain complexes with their homology groups as in Section~\ref{subsec:Massey}.
\end{ex}

Independently, let us consider the following \emph{algebraic data} of a associative algebra structure on $A$:
$$\nu : A^{\otimes 2} \to A, \quad \textrm{such that} \quad  \nu(\nu(a,b),c)=\nu(a,\nu(b,c)), \ \forall a,b,c \in A \ . $$ 
By simplicity, we choose to depict these composites by the following graphically composition schemes: 

\begin{center}
\includegraphics[scale=0.2]{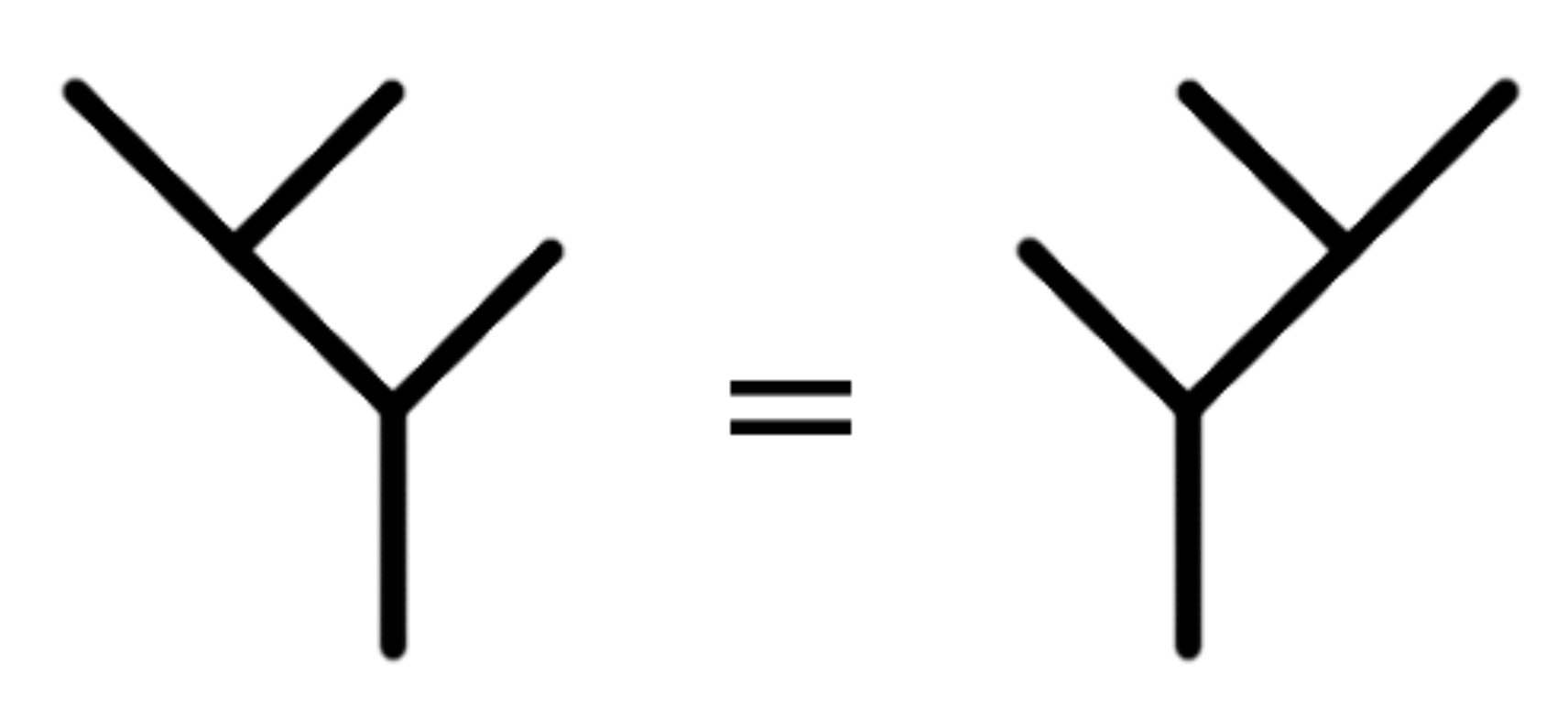} 
\end{center}

where we forget about the input variables since they are generic.
Actually, we consider a differential graded associative algebra structure on $(A,d_A)$, dga  algebra for short, on $A$. This means that the differential $d_A$ of $A$ is a derivation for the product $\nu$:

%\begin{figure}[!h]
%\centering
\begin{center}
\includegraphics[scale=0.2]{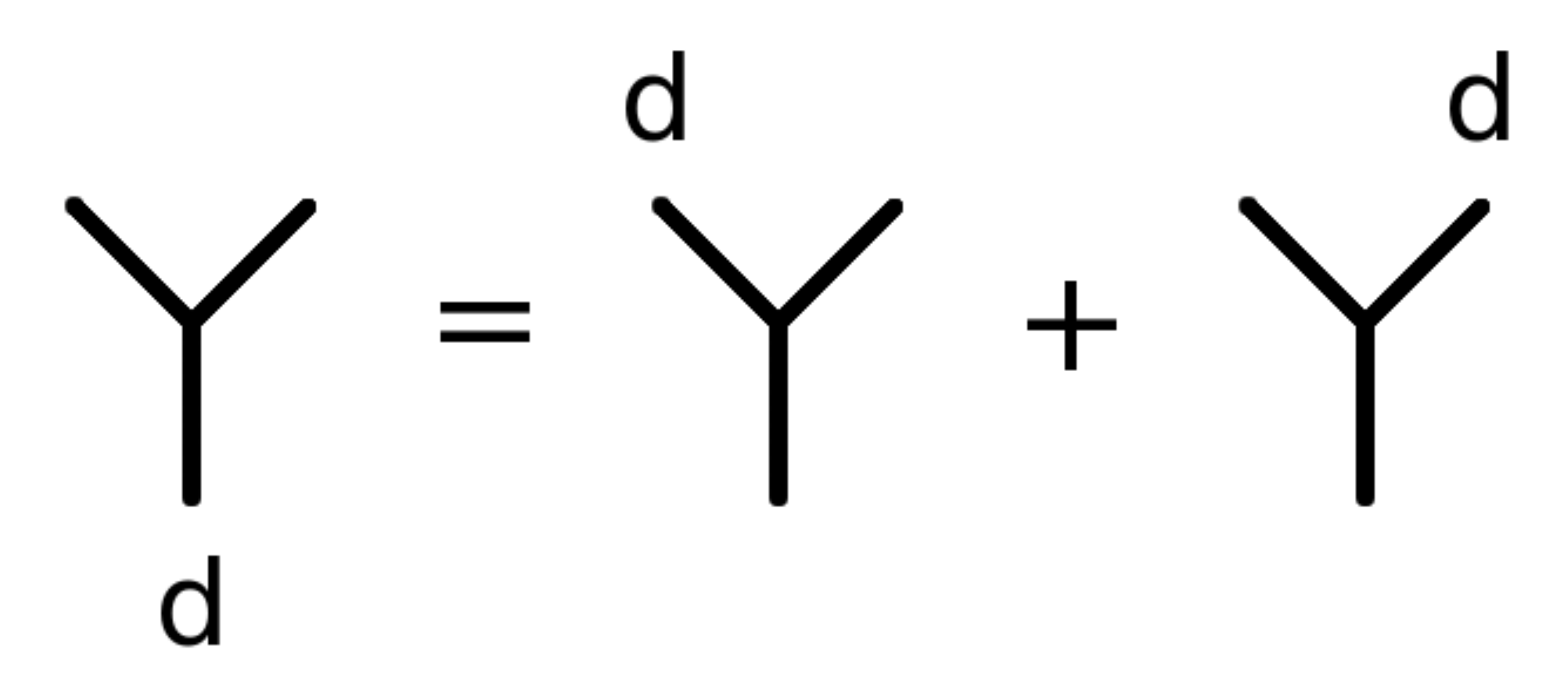} 
\end{center}
%\caption{Derivation relation}
%\label{Fig:Der}
%\end{figure}

(In this section, we do not require that the associative algebra has a unit.)

\subsection{Transfer of algebraic structure}\label{subsec:TransferDGA}
One can now try to study how these data mix. Namely,  is it  possible to transfer the associative algebra structure from $A$ to $H$ through the homotopy data in some way~? The first answer is given by the following formula for a binary product 
$$\mu_2 := p \circ \nu \circ i^{\otimes 2}\  : \  H^{\otimes 2} \to H$$ on $H$. 
Under our graphical representation, it gives

%\begin{figure}[!h]
%\centering
\begin{center}
\includegraphics[scale=0.2]{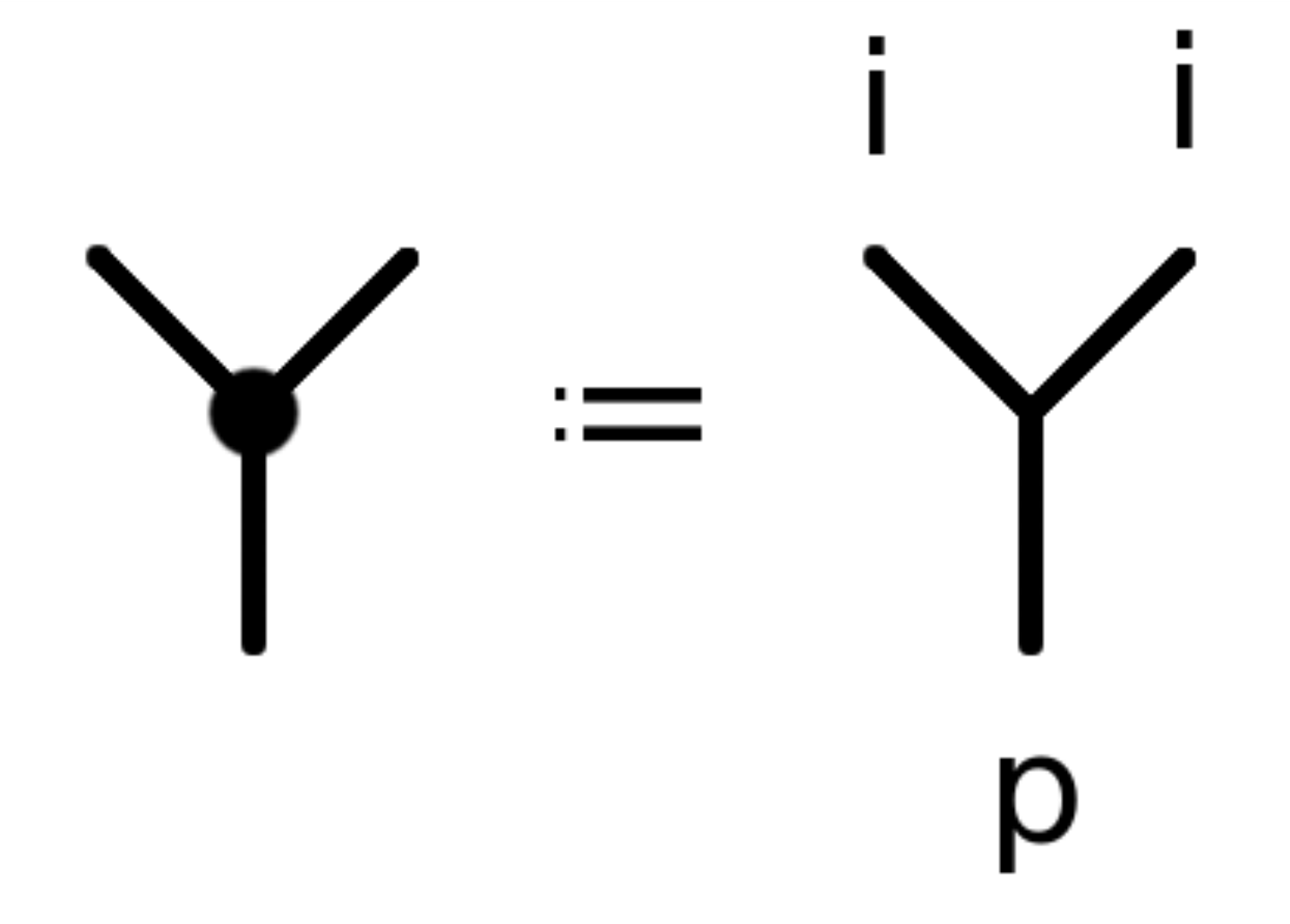} 
\end{center}
%\caption{Definition of $\mu_2$}
%\label{Fig:mu2}
%\end{figure}

Now, let us check whether the product $\mu_2$ is associative: 

%\begin{figure}[h]
%\centering
\begin{center}
\includegraphics[scale=0.2]{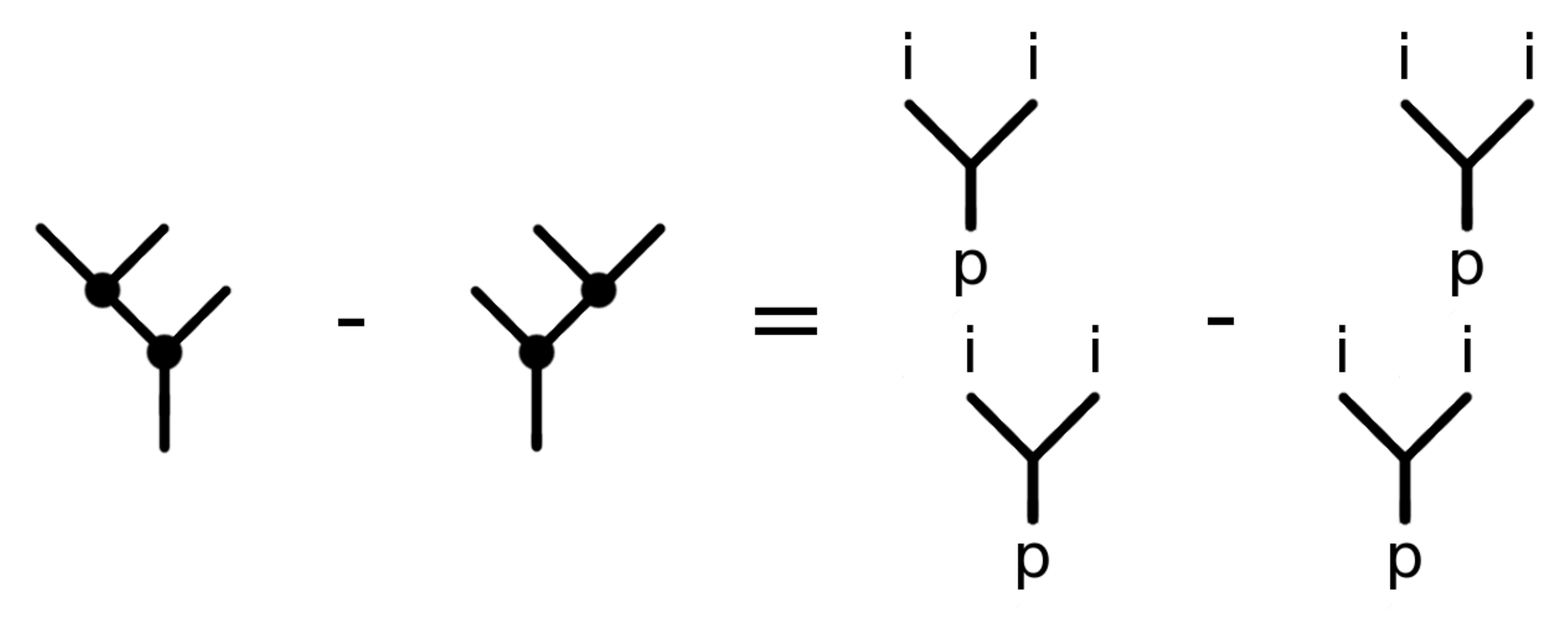} 
\end{center}
%\caption{Associativity of  $\mu_2$ ?}
%\label{Fig:AssocMu2}
%\end{figure}

If the map $i$ were an isomorphism, with inverse $p$, the answer would be \emph{yes}. But, this does not hold in general; so the answer is \emph{no}. Since the composite $i p$ is equal to the identity up to the homotopy $h$, the transferred product $\mu_2$ seems to be associative only ``up to the homotopy $h$''. Let us make this statement precise.

The associativity relation is equivalent to the  vanishing of the associator 

%\begin{figure}[h]
%\centering
\begin{center}
\includegraphics[scale=0.2]{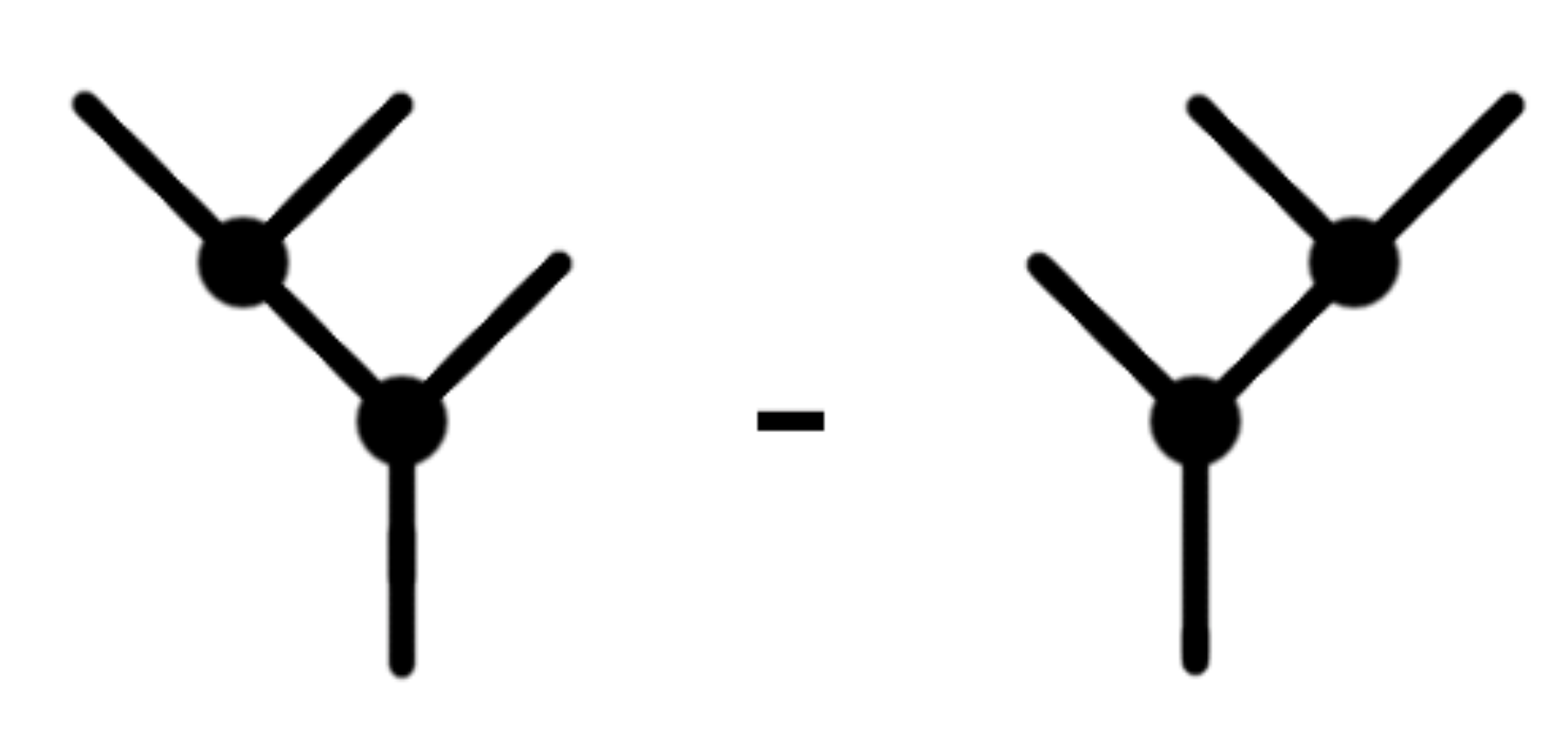} 
\end{center}
%\caption{Associator of  $\mu_2$}
%\label{Fig:AssociatorMu2}
%\end{figure}

 in $\Hom(H^{\otimes 3}, H)$. This mapping space becomes a chain complex when equipped with the usual differential map $\partial (f) :=  d_H f - (-1)^{|f|} d_{H^{\otimes 3}} f $. We introduce the element  $\mu_3$

%\begin{figure}[h]
%\centering
\begin{center}
\includegraphics[scale=0.2]{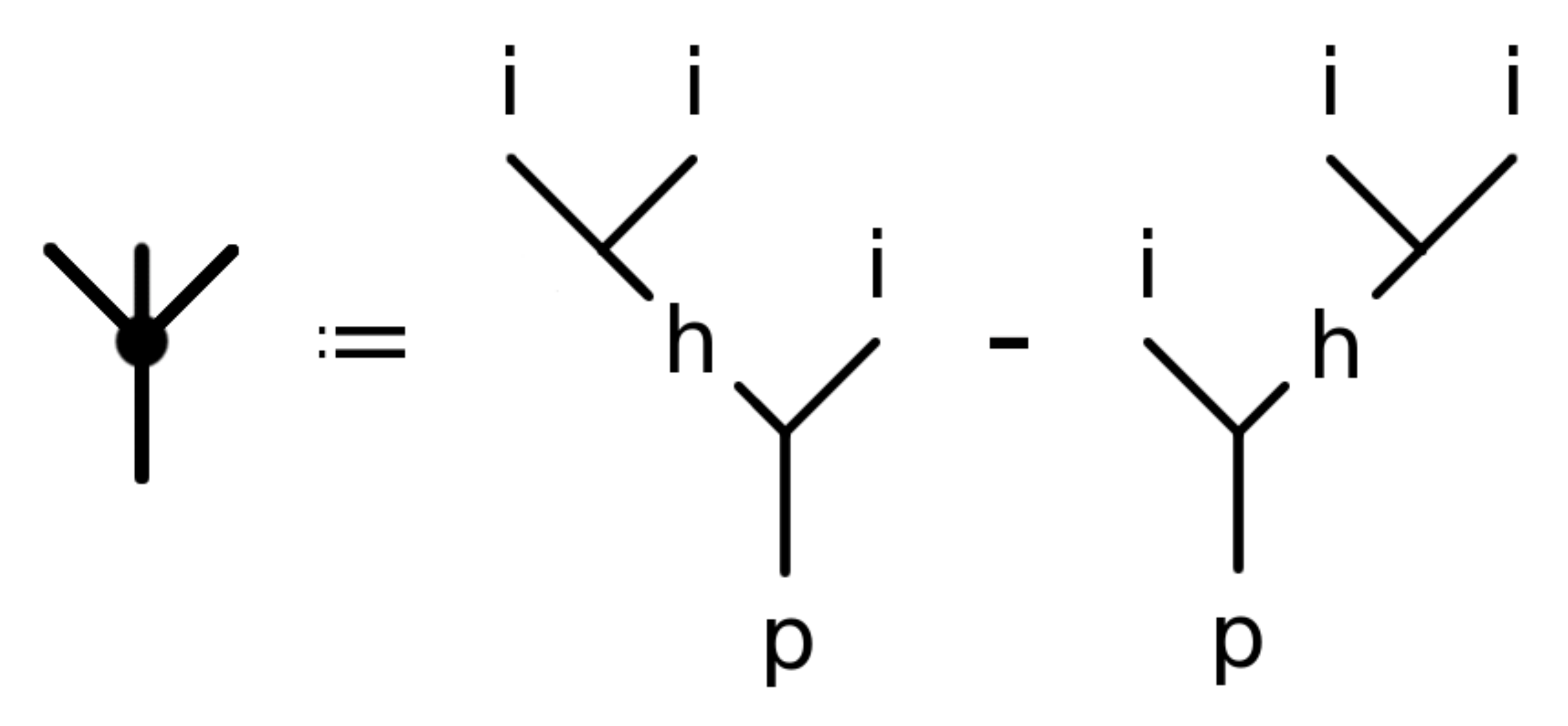} 
\end{center}
%\caption{Definition of  $\mu_3$}
%\label{Fig:Mu3}
%\end{figure}

It has degree $+1$, since it is the composite of maps of degree $0$ (the maps $i$, $p$, and $\nu$) and one map of degree $1$ (the map $h$).
\begin{ex}
We leave it to the reader to check that 

%\begin{figure}[h]
%\centering
\begin{center}
\includegraphics[scale=0.2]{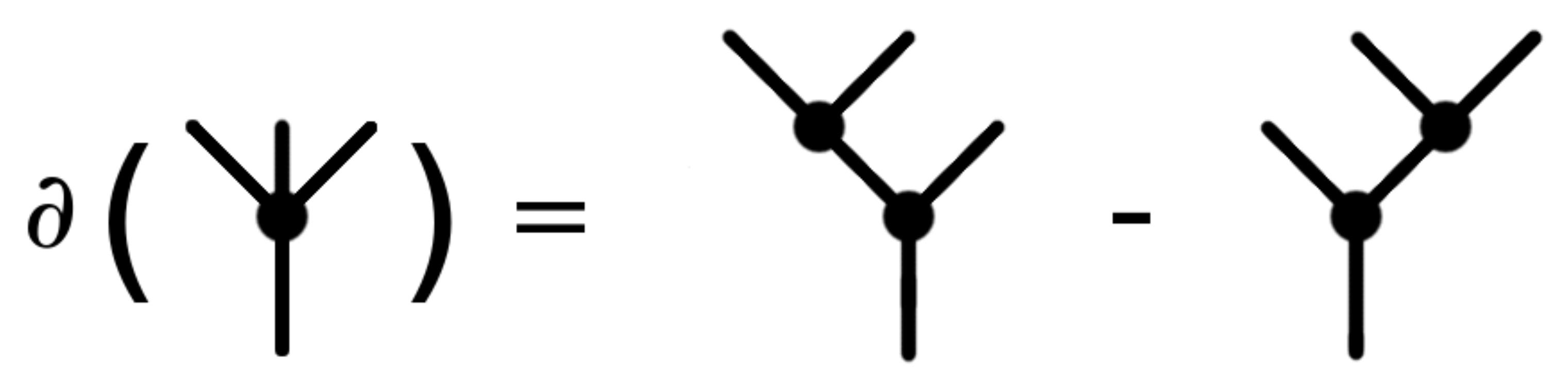} 
\end{center}
%\caption{Derivative of  $\mu_3$}
%\label{Fig:DerMu3}
%\end{figure}
\end{ex}

This means that, in the chain complex $(\Hom(H^{\otimes 3}, H), \partial)$, the associator of $\mu_2$ is not zero in general, but vanishes up to the homotopy $\mu_3$. Under this property, one says that  ``the product $\mu_2$ is associative up to homotopy''. \\ 

The next step consists in checking whether the two operations $\mu_2$ and $\mu_3$ satisfy some relation. The answer is again \emph{yes}, they satisfy one new relation but \emph{only up to yet another homotopy}, which is a linear map $\mu_4$ of degree $+2$ in  $\Hom(H^{\otimes 4}, H)$. And so on and so forth. 

Let us cut short and give a general answer: generalizing the previous two formulae, we consider the following family of maps 
$$ \mu_n=\vcenter{\hbox{\includegraphics[scale=0.2]{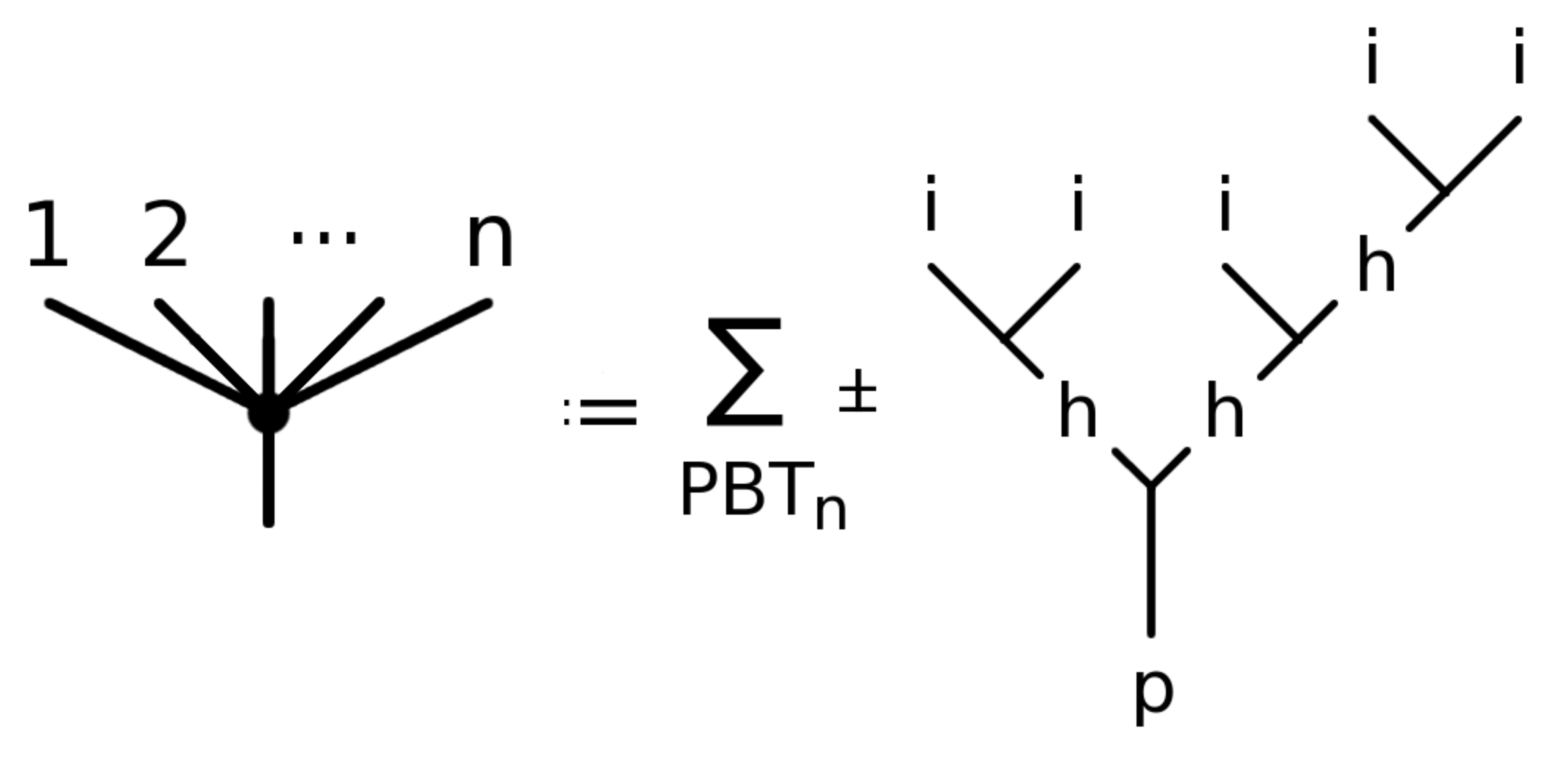}}} $$
in  $\Hom(H^{\otimes n}, H)$, for any $n \ge 2$, where the notation $PBT_n$ stands for the set of planar binary rooted trees with $n$ leaves. This defines linear maps of degree $n-2$, which satisfy the following relations 

\begin{center}
\boxed{\includegraphics[scale=0.2]{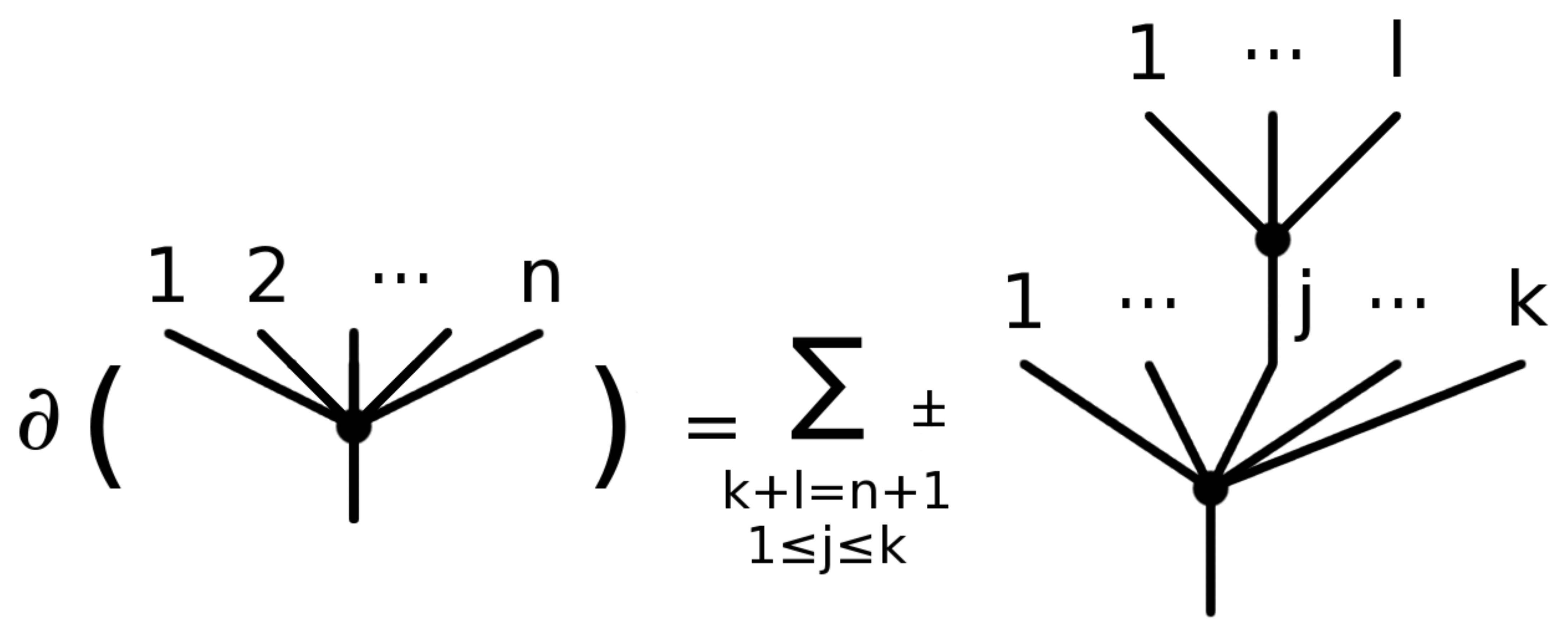} }
\end{center}

\subsection{$A_\infty$-algebra}\label{subsec:AinftyAlg}
The aforementioned example triggers the following general definition.

\begin{defi}[$A_\infty$-algebra]
An \emph{associative algebra up to homotopy}, also called \emph{homotopy as\-so\-cia\-ti\-ve algebra} or  \emph{$A_\infty$-algebra} for short, is a chain complex $(A, d)$ endowed with a family of operations 
$\mu_n : A^{\otimes n} \to A$ of degree $n-2$, for any $n\ge 2$,  satisfying the aforementioned relations.
\end{defi}

This notion was discovered by Jim Stasheff in his Ph.D thesis \cite{Stasheff63} on loop spaces. 

\begin{exam}
An $A_\infty$-algebra, whose higher operations $\mu_n =0$ vanish for $n\ge 3$, is a dga algebra. If we define a morphism of $A_\infty$-algebras to be a chain map commuting with the respective operations, then 
 the category of dga algebras is a full sub-category of the category of 
 $A_\infty$-algebras:
$$\mathsf{ dga \  alg} \ \subset\ A_\infty\textsf{-}\mathsf{alg}\ .$$
\end{exam}

\begin{theo}[Homotopy Transfer Theorem for dga algebras \cite{Kadeishvili80}]\label{HTT1}
The operations $\lbrace \mu_n \rbrace_{n\ge 2}$  defined on $H$ from the associative product $\nu$ on $A$  by the formulae of Section~$\ref{subsec:TransferDGA}$  form an $A_\infty$-algebra structure. 
\end{theo}

\begin{proof}
We leave the proof to the reader as an easy and pedagogical exercise.
\end{proof}

\subsection{Application: Massey products}\label{subsec:Massey}

The first kind of examples is given by considering 
$$(H, d_H):=(H_\bullet(A, d_A), 0)$$ 
together with the following deformation retract.  First the degree $n$ space  $A_n$ is the direct sum $A_n \cong  Z_n \oplus B_{n-1}$ of the $n$-cycles $Z_n:=\lbrace x\in A_n; d_A(x)=0\rbrace$ with the $(n-1)$-boundaries $B_{n-1}:=d_A(A_{n})$. Then, one lifts the homology groups $H_n:=Z_n/B_n$, to finally get $A_n\cong B_n \oplus H_n \oplus B_{n-1}$. Such a decomposition induces the required  maps:
\begin{eqnarray*}
&& i : H_n \mono  B_n \oplus H_n \oplus B_{n-1}, \quad p : B_n \oplus H_n \oplus B_{n-1} \epi H_n \quad \textrm{and} 
\\ 
&& h : A_{n-1} \cong B_{n-1} \oplus H_{n-1} \oplus B_{n-2} \epi B_{n-1} \mono B_{n} \oplus H_{n} \oplus B_{n-1} \cong A_{n}\ .
\end{eqnarray*}

It is an easy exercise to prove that the product $\nu$ of a dga algebra induces an associative product $\bar \nu$ on the underlying homology groups. Moreover, Theorem~\ref{HTT1} allows us to endow $H(A)$ with a much richer structure extending $\bar \nu=\mu_2$.

\begin{prop}
The underlying homology groups $H(A)$ of a dga algebra $A$ carry an $\Ai$-algebra structure, with trivial differential, and extending the induced binary product. 
\end{prop}

\begin{defi}[$\Ai$-Massey products]
We call \emph{$\Ai$-Massey products} the  operations $\lbrace \mu_n : H(A)^{\otimes n} \to H(A) \rbrace_{n \ge 2}$ of the transferred $\Ai$-algebra structure on $H(A)$.
\end{defi}

In the light of the formulae of Section~\ref{subsec:TransferDGA}, we would like to say that  the induced binary product on homology is associative ``for a bad reason'': not that the first homotopy $\mu_3$ vanishes, but just because the differential map $d_H$ is trivial ! In this case, even if $\mu_2$ is associative, the higher homotopies $\lbrace \mu_n \rbrace_{n \ge 3}$  are not trivial, in general, as the example below shows. The $\Ai$-Massey products actually encode faithfully the homotopy type of the dga algebra $A$, see Section~\ref{subsec:HoTheoAiAlg}. \\

The classical \emph{Massey products}, as defined by William S. Massey in \cite{Massey58}, are only partially defined and live on some coset of the homology groups. For example,  the \emph{triple Massey product} $\langle \bar x, \bar y, \bar z\rangle$ is defined for homology classes  $\bar x,\bar  y,\bar  z$ such that $ \bar{\nu}(\bar x,  \bar  y) = 0 = \bar{\nu} (\bar  y,  \bar  z)$ as follows. Let us denote by $x,y,z$ cycles representing the homology classes. Because of the hypothesis, there exist chains $a$ and $b$ such that $(-1)^{|x|} \nu(x, y) = d_A(a) , (-1)^{|y|}\nu(y, z) = d_A(b)$. Then the chain 
$$(-1)^{|x|} \nu(x, b) + (-1)^{|a|} \nu(a, z)$$
 is a cycle and so defines an element $\langle \bar x, \bar y, \bar z\rangle$ in  
 $H(A)/\lbrace \bar{\nu}(\bar x , H(A)) + \bar{\nu}(H(A), \bar z)\rbrace$.
 
 \begin{ex}
 Show that $\mu_3(x,y,z)$, as defined by the formula of Section~\ref{subsec:TransferDGA}, provides a representative for the partial triple Massey product $\langle \bar x, \bar y, \bar z\rangle$. 
 \end{ex}
 
 The formulae defining the Massey products are similar to  the formulae defining the $\Ai$-Massey products. The  difference lies in the fact  that, in the latter case,  all the choices are made once and for all in the deformation retract. (The various liftings are then given by the homotopy $h$.)
 
 \begin{prop}
 The $\Ai$-Massey products provide representatives for the partial Massey products. 
 \end{prop}
 
The original Massey products were defined in the context of the singular cohomology  since, for any topological space  $X$, the cup product $\cup$ endows the singular cochain complex $C^\bullet_\textrm{sing}(X)$ with a dga algebra structure, cf. \cite[Section~$3.2$]{Hatcher02}.

\begin{appl}[Borromean rings] 
The Borromean rings are three interlaced rings, which represent the coat of arms of the Borromeo family in Italy, see Figure~\ref{fig:Borro}. Removing any one of the three rings, the remaining two rings are no longer interlaced. 

We consider the complement in the $3$-sphere of the Booromean rings. Each ``ring'', i.e.\ solid torus, determines a $1$-cocycle: take a loop from the base-point with linking number one with the circle. Since any two of the three circles are disjoint, the cup product of the associated cohomology classes is $0$. The nontriviality of the triple Massey product of these three cocycles detects the nontrivial entanglement of the three circles, cf. \cite {Massey58}. 
\begin{figure}[h]
\centering
{\scalebox{1.8}{\includegraphics{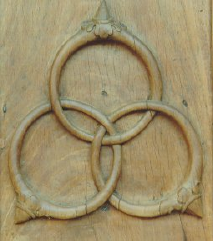}}}
\caption{Coat of arms of the Borromeo family}
\label{fig:Borro}
\end{figure}
\end{appl}

\subsection{Homotopy Transfer Theorem}
We have seen that starting from a dga algebra  structure on one side of a homotopy data, one gets an $A_\infty$-algebra structure, with full higher operations,  on the other side. What will happen if one starts from an $A_\infty$-algebra structure $\lbrace \nu_n : A^{\otimes n} \to A\rbrace_{n \ge 2} $ ? 

In this case, one can extend the aforementioned formulae, based on planar binary trees, by considering the full set $PT_n$ of planar rooted trees with $n$ leaves, this time:
$$ \boxed{\mu_n=\vcenter{\hbox{\includegraphics[scale=0.2]{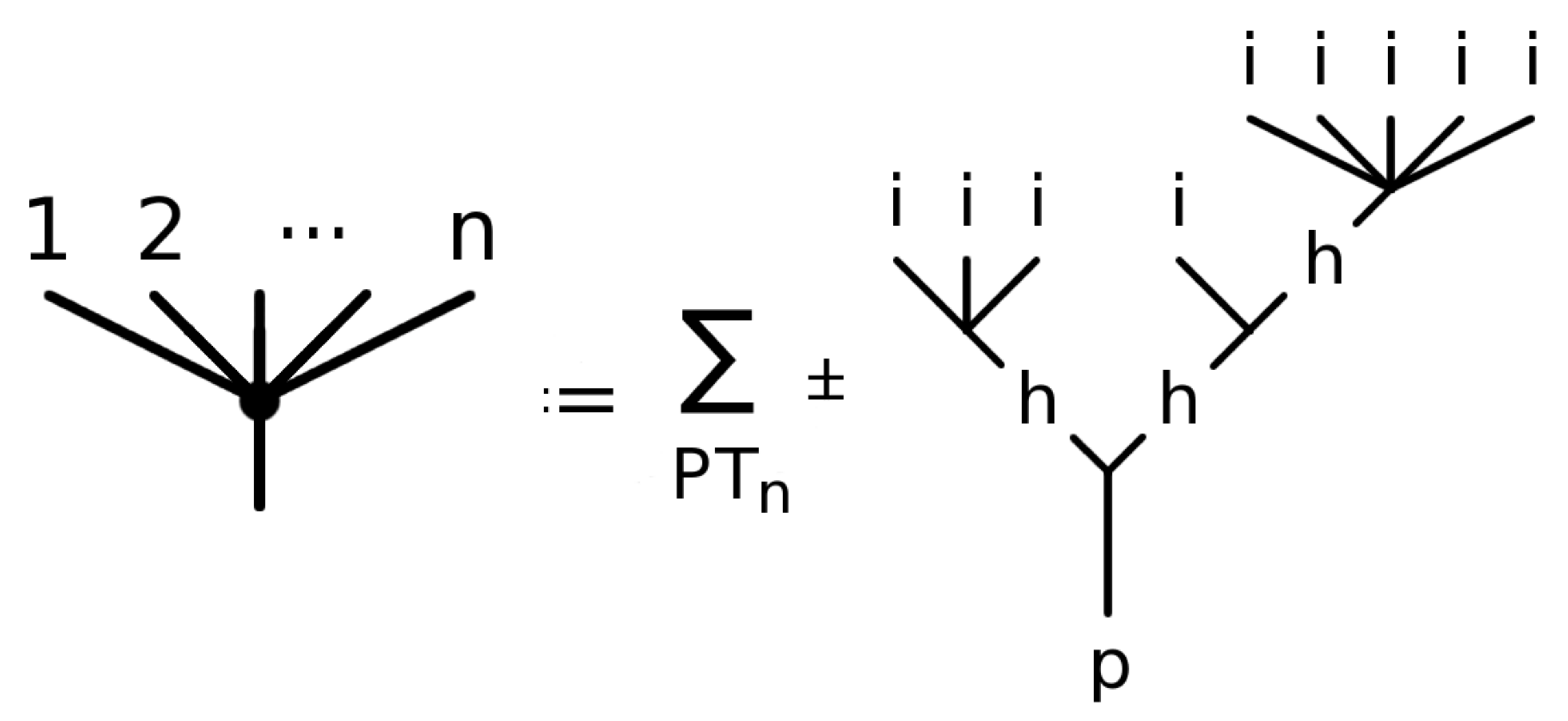}}} }$$
\begin{theo}[Homotopy Transfer Theorem for $A_\infty$-algebras \cite{Kadeishvili82}]\label{HTT2}
The operations $\lbrace \mu_n \rbrace_{n\ge 2}$  defined on $H$ from the $A_\infty$-algebra structure $\lbrace \nu_n \rbrace_{n\ge 2}$ on $A$  by the aforementioned formulae  form an $A_\infty$-algebra structure. 
\end{theo}

\begin{proof}
Once again, we leave this straightforward computation to the learning reader. 
\end{proof}

So, with the definition of $\Ai$-algebras, we have reached a big enough ``category'', including that of dga algebras, and stable under transfer through homotopy data. It remains to define a good  notion of morphism.

\subsection{$A_\infty$-morphism}\label{subsec:AiMorph}

\begin{defi}[$A_\infty$-morphism]
An \emph{$\Ai$-morphism} between two $\Ai$-algebras $(A, d_A, \lbrace \mu_n  \rbrace_{n\ge 2})$ and $(B,$ $d_B,$ $\lbrace \nu_n  \rbrace_{n\ge 2})$ is a collection of linear maps 
$$ \lbrace f_n \ : \ A^{\otimes n } \to B \rbrace_{n\ge 1} $$
of degree $n-1$ satisfying 
\begin{center}
\boxed{\includegraphics[scale=0.27]{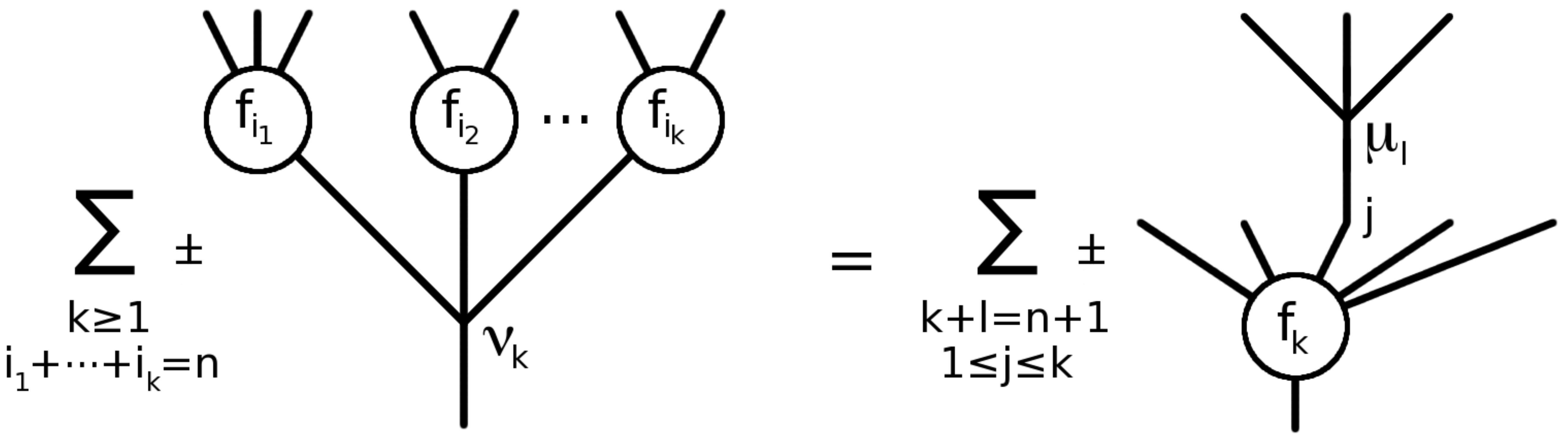}  }
\end{center}
\end{defi}
where $\mu_1=d_A$ and $\nu_1=d_B$.
The data of an $A_\infty$-morphism is denoted $f : A \rightsquigarrow B$. 
The composite of 
two 
$A_\infty$-morphisms 
$f : A \rightsquigarrow B $ and $g : B \rightsquigarrow C $
is defined by 
$$ (g_\circ f)_n:=\vcenter{\hbox{\includegraphics[scale=0.27]{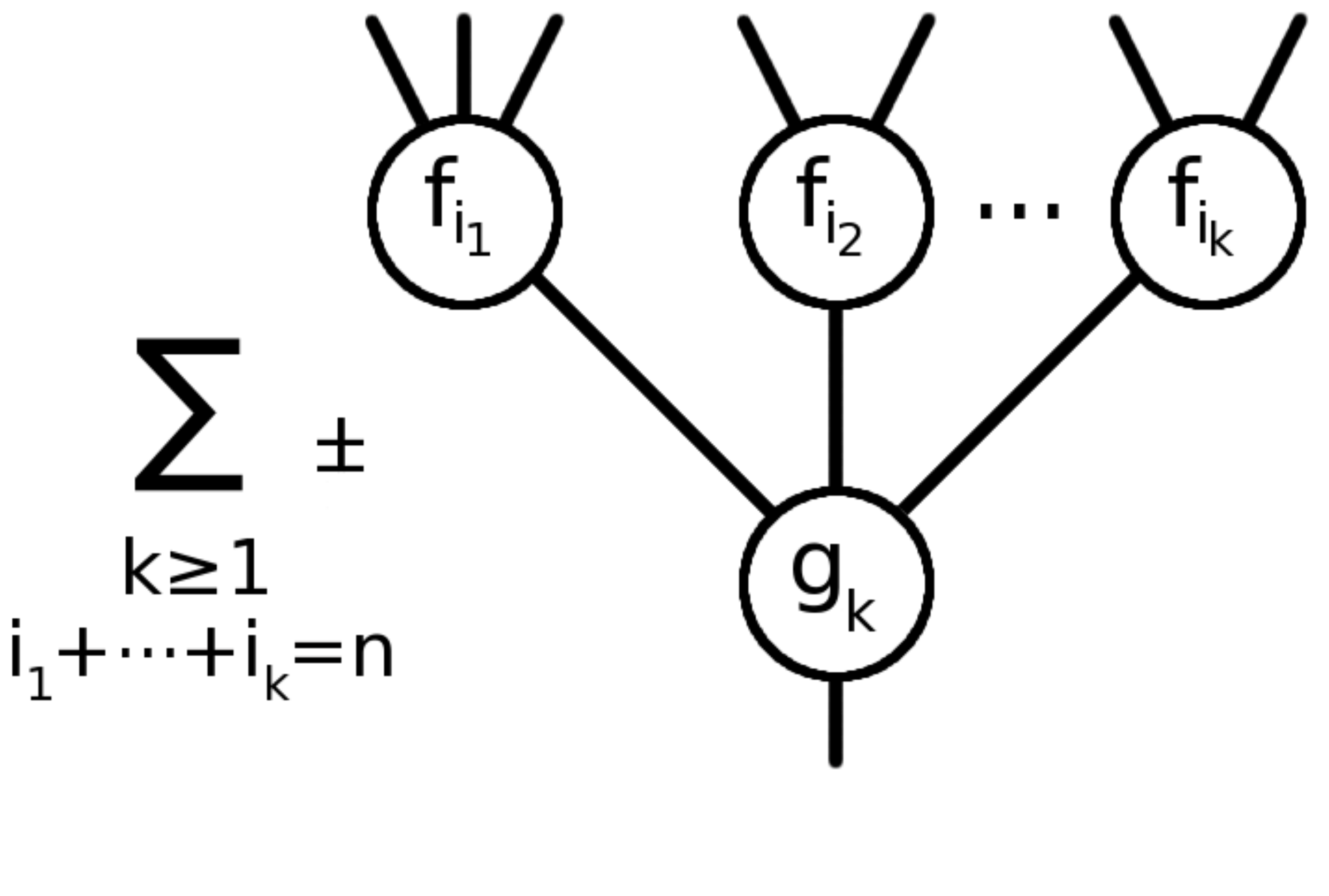}}} $$

\begin{ex}
Show that $\Ai$-algebras with $\Ai$-morphisms form a category. We denote this category by $\infty\textsf{-}A_\infty\textsf{-alg}$, where the first symbol $\infty$ stands for the $A_\infty$-morphisms.
\end{ex}

By definition, the first component $f_1 : A \to B$ of an $\Ai$-morphism is a morphism of chain complexes. 

\begin{defi}[$A_\infty$-isomorphism and $A_\infty$-quasi-isomorphism]
When the map $f_1$ is an isomorphism (resp. a quasi-isomorphism), the $\Ai$-morphism  $f$ is called an \emph{$\Ai$-isomorphism} (resp. an \emph{$\Ai$-quasi-iso\-mor\-phi\-sm}). 
\end{defi}

\begin{ex}
Show that $\Ai$-isomorphisms are the invertible morphisms of the category $\infty\textsf{-}A_\infty\textsf{-alg}$. (This proof is similar to the proof that power series $a_1 x + a_2 x^2 + \cdots$ with first term  invertible  are invertible). 
\end{ex}

\begin{theo}[Homotopy Transfer Theorem with $A_\infty$-quasi-isomorphism \cite{KontsevichSoibelman01}]\label{HTTAiQI}
Starting from a homotopy retract, the maps $\tilde{\imath}_1:=i : H \to A$ and 
$$ \tilde{\imath}_n:=\vcenter{\hbox{\includegraphics[scale=0.2]{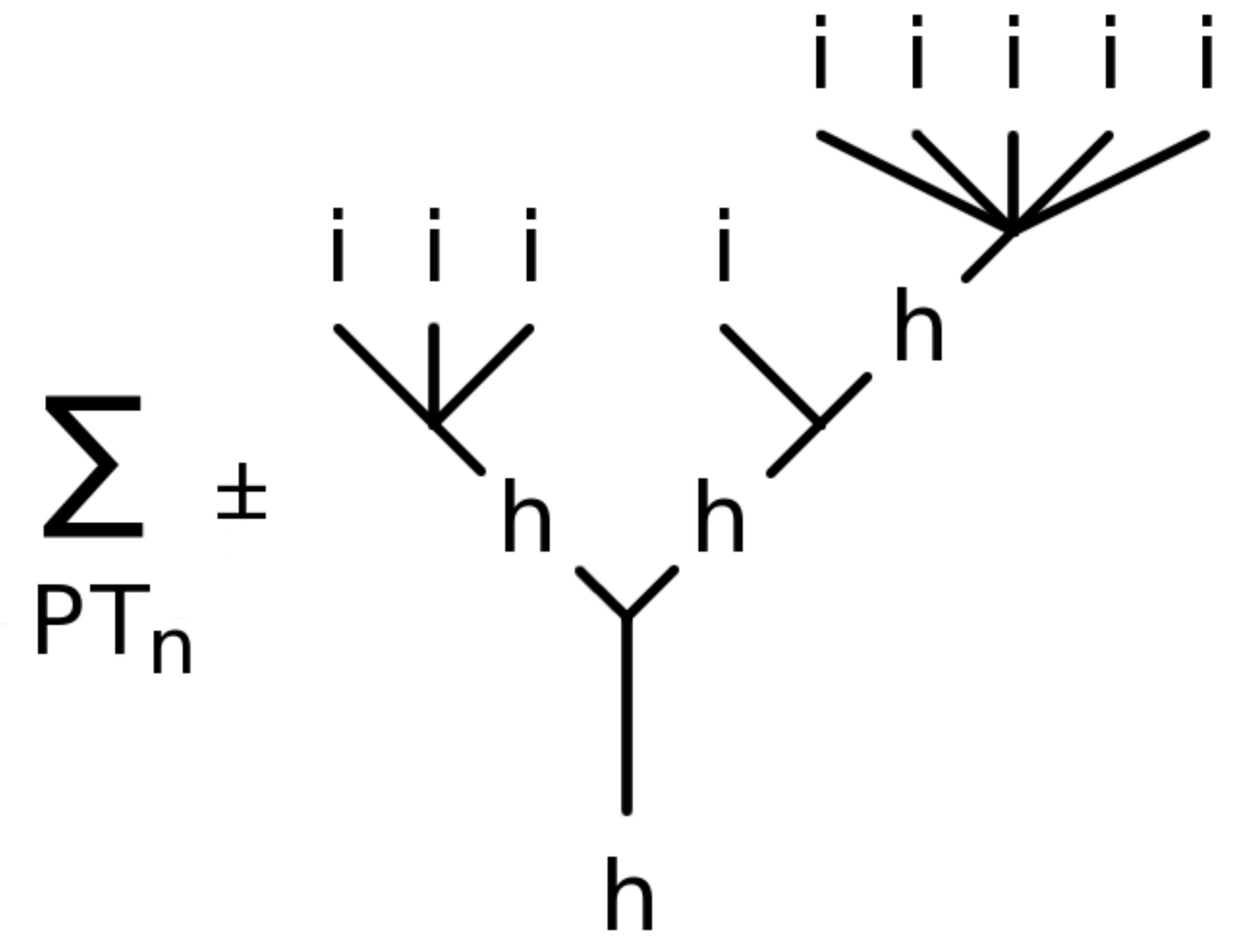}}} \ : \ H^{\otimes n} \to A\ , $$ for $n\ge 2$, 
define an $\Ai$-quasi-isomorphism $\tilde{\imath} : H \stackrel{\sim}{\rightsquigarrow} A$ between the transferred $A_\infty$-algebra structure $\lbrace \mu_n \rbrace_{n\ge 2}$ on $H$ and the initial 
$A_\infty$-algebra structure $\lbrace \nu_n \rbrace_{n\ge 2}$ on $A$.
\end{theo}

\begin{proof}
As usual, the proof is left to the reader as a good exercise. 
\end{proof}

One can actually proceed one step further and extend the whole homotopy data from the category of chain complexes to the category of $\Ai$-algebras. There exists a notion of an   \emph{$\Ai$-homotopy} between $\Ai$-morphisms, whose equivalence relation is denoted $\sim_h$. Starting from a deformation retract, the complete result states that the quasi-isomorphism $p$ extends to an $\Ai$-quasi-isomorphism $\tilde{p}$ and that the homotopy $h$ extends to an {$\Ai$-homotopy} between $\Id_A$ and the composite $\tilde{p} \, \tilde{\imath}$, cf. \cite{Markl04}.

\subsection{Homotopy theory of $\Ai$-algebras}\label{subsec:HoTheoAiAlg}

\begin{theo}[Fundamental theorem of $\Ai$-quasi-isomorphisms \cite{LefevreHasegawa03}]
Any $\Ai$-quasi-isomorphism 
$A\stackrel{\sim}{\rightsquigarrow}B$ admits a homotopy inverse 
$\Ai$-quasi-isomorphism $B  \stackrel{\sim}{\rightsquigarrow} A$.
\end{theo}

This is the main property of $\Ai$-quasi-isomorphisms. Notice that given a 
 quasi-isomorphism of dga algebras, there does not always exist a quasi-isomorphism in the opposite direction as dga algebras. 
These two notions of quasi-isomorphisms are related by the following property:
 there exists a zig-zag of quasi-isomorphisms of dga algebras 
if and only if there exists a direct $\Ai$-quasi-isomorphism
 $$\exists \ A \stackrel{\sim}{\leftarrow} \bullet  \stackrel{\sim}{\rightarrow} 
\bullet  \stackrel{\sim}{\leftarrow} \bullet \cdots \bullet  \stackrel{\sim}{\rightarrow}  B 
\quad \Longleftrightarrow \quad \exists \ A \ \stackrel{\sim}{\rightsquigarrow} B \ . $$  

\begin{defi}[Formality]
A dga algebra $(A, d, \nu)$ is called \emph{formal} if there exists a zig-zag of quasi-isomorphisms of dga algebras
$$(A, d, \nu) \stackrel{\sim}{\leftarrow} \bullet  \stackrel{\sim}{\rightarrow} 
\bullet  \stackrel{\sim}{\leftarrow} \bullet \cdots \bullet  \stackrel{\sim}{\rightarrow}  (H(A), 0, \bar \nu)  \ .$$
\end{defi}

\begin{prop}$Ê\ $

\begin{enumerate}
\item If a dga algebra is formal, then the higher Massey products, i.e. for $n\ge 3$, vanish. 

\item If the higher $\Ai$-Massey products, i.e. for $n\ge 3$, vanish, then the dga algebra is formal.

\end{enumerate}
\end{prop}

\begin{proof}
The proof of the second point is a corollary of the HTT, Theorem~\ref{HTTAiQI}. When the higher $\Ai$-Massey products vanish, the map  $\tilde{\imath} : H(A) \stackrel{\sim}{\rightsquigarrow} A$ is an $\Ai$-quasi-isomorphism between two dga algebras. We conclude then with the above property. 
\end{proof}

In the latter case, the Massey products vanish in a ``uniform way'', see \cite{DGMS75} for more details.

\begin{defi}[Homotopy category] 
The \emph{homotopy category} of dga algebras $$\textsf{Ho}(\mathsf{dga\ alg}):=\mathsf{dga\ alg}[qi^{-1}]$$  is defined as  the localization of the category of dga algebras with respect to the class of  quasi-isomorphisms. 
\end{defi}
This is the ``smallest'' category containing the category $\mathsf{dga\ alg}$ in which quasi-isomorphisms become invertible. It is actually made up of more maps than the category $\mathsf{dga\ alg}$ since the maps of the homotopy category are (some equivalence classes)
of chains 
of maps either from the morphisms   of dga algebras or from the  \emph{formal inverse} of quasi-isomorphisms:
$$\xymatrix{A \ar[r]   &   \bullet  \ar[r]&  \bullet       \ar@{..>}[r] &  \ar@/^/[l]^\sim   \bullet\ar[r]& \bullet \ \  \cdots \ \ \bullet   \ar[r] &
\bullet    \ar@{..>}[r]&  \ \ar@/^/[l]^\sim \bullet   \ar@{..>}[r]  &  \ar@/^/[l]^\sim \bullet \ar[r]  &  B\ .} $$
 
The following results shows that the homotopy classes of $\Ai$-morphisms correspond to the 
maps in the homotopy category of dga algebras.

\begin{theo}[Homotopy category \cite{Munkholm78, LefevreHasegawa03}]
The following categories are equivalent
$$\boxed{\mathsf{Ho}(\mathsf{dga\ alg})\  \cong \  \infty\textsf{-}A_\infty\mathsf{-alg}/\sim_h    
\ 
 \cong \  \infty\textsf{-}\mathsf{dga\ alg}/\sim_h     \   . }$$
\end{theo}

\begin{proof}
The last equivalence with the category of dga algebras equipped with the $\Ai$-morphisms  is given by the \emph{Rectification property}: any $\Ai$-algebra admits a dga algebra, which is  $\Ai$-quasi-isomorphic to it.
\end{proof}

%\hrule

\begin{ques}$ \ $

\begin{enumerate}

\item Why this particular definition of an $A_\infty$-algebra when starting from an associative  algebra ? 

\item How could we perform the same theory starting from other types of algebras like commutative algebras,  Lie algebras or Lie bialgebras, for instance ? 

\end{enumerate}
\end{ques}

\section{Operads}

In this section, we introduce the notion of an operad trough the paradigm given by the multilinear maps of a vector space. One can also use operads in the context of topological spaces, differential graded modules and simplicial sets, for instance. Operads encode categories of algebras having operations having multiple inputs and one output. When one wants to encode other types of algebraic structures, one has to introduce some generalizations of operads.

\subsection{Unital associative algebras and their representation theory}\label{subsec:uAsRep}
Let $V$ be a vector space. The vector space $\Hom(V,V)$ of endomorphisms of $V$, equipped with the composition $\circ$ of functions and the identity morphism $id_V$,  becomes a unital associative algebra. 
Representation theory is the study  of morphisms of unital associative algebras with target  the endomorphism algebra:

$$
\begin{array}{ccc}
\Phi \ : \ (A,\mu, 1)  &\to& (\Hom(V,V), \circ, id_V) \\
a &\mapsto& \vcenter{\hbox{\includegraphics[scale=0.2]{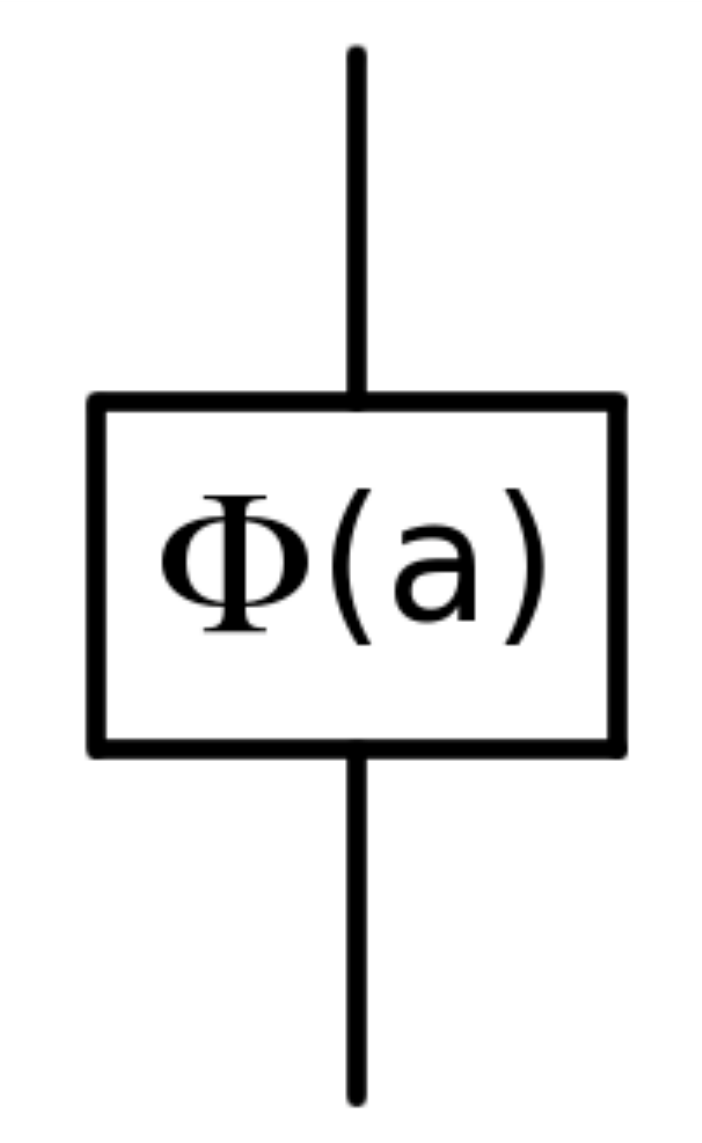}}}
\end{array}
$$

This means that the map $\Phi$ sends the unit $1$ to the identity $id_V$ and any product of two elements of $A$ to the associated composite

$$ \Phi(\mu(b\otimes a))\ =\Phi(b)\circ \Phi(a) = \ \vcenter{\hbox{\includegraphics[scale=0.2]{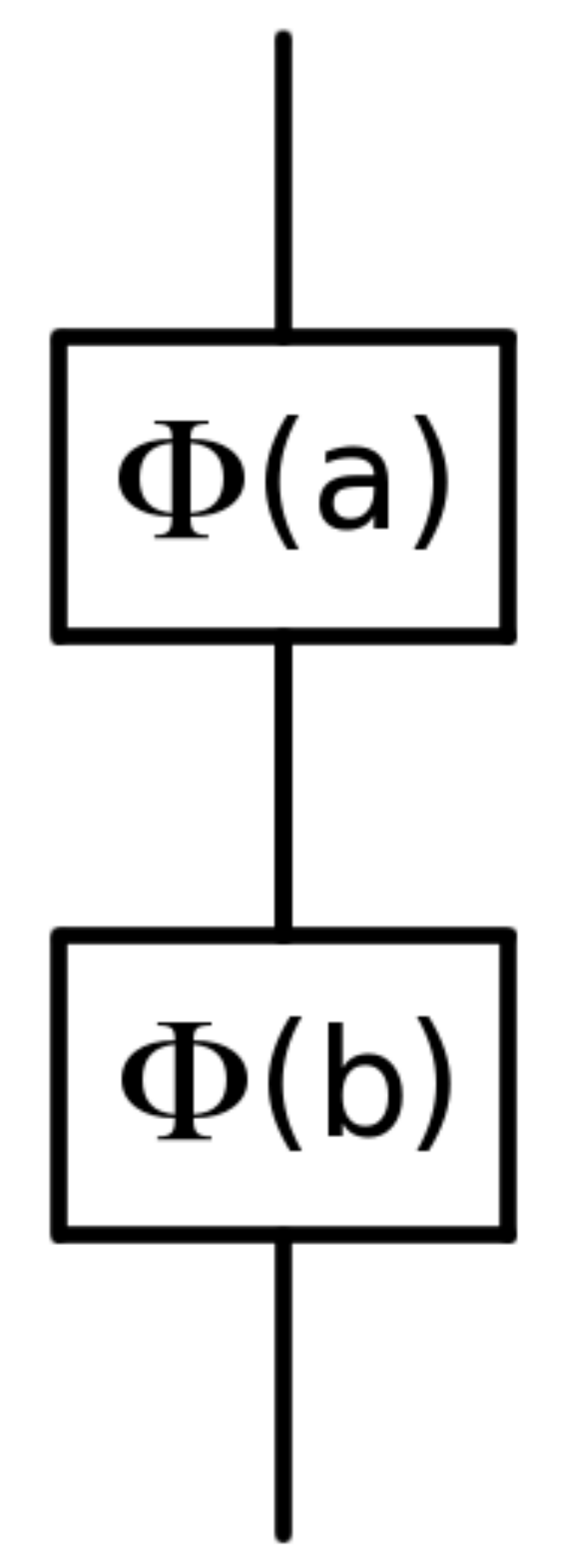}}}  $$

On the one hand, the endomorphism space $\Hom(V,V)$ is the space of all ``concrete'' linear operations acting on $V$. On the other hand, the elements composing the image of the representation $\Phi$ form a particular choice of operations. The type (number, relations, etc.) of these operations is encoded, once and for all $V$ and $\Phi$, by the 
algebra $A$. 

\begin{exam}[Algebra of dual numbers and chain complexes]
Let us begin with an elementary example. Suppose that we want to encode in one associative algebra, the algebraic data of a unary square-zero operator. In this case, we would first consider the free unital associative algebra on one generator $\delta$: it is explicitly given by the tensor module  
$$T(\KK \delta):=\bigoplus_{n\in \NN} (\KK \delta)^{\otimes n}\ ,$$ 
equipped with the concatenation product. This is meant to represent all the possible ``formal'' compositions of the operator with itself. 
Any morphism of unital associative algebras $\Phi : T(\KK \delta) \to \Hom(V,V)$ is completely characterized by the image $d:=\Phi(\delta)$ of this generator. Then we quotient this free algebra by the ideal generated by $\delta^2$, that is the relation we want to model. It gives the algebra of \emph{dual numbers}:
$$D:= T(\KK \delta)/(\delta^2)\ .$$
In this case, a morphism of unital associative algebras $\Phi : D \to \Hom(V,V)$ is characterized by the image $d:=\Phi(\delta)$ of the generator $\delta$, which  has to satisfy the relation $d^2=0$. Finally, a vector space equipped with a square-zero operator is a representation of $D$. 
%, also called a $D$-module \draftnote{no confusion with diff mod}.
 If we work with graded vector spaces and if we place the generator $\delta$ in degree $-1$, then we get nothing but the notion of a chain complex. 
\end{exam}

\begin{exam}[Steenrod operations  and Steenrod  algebra]
Even if the previous example can seem naive at first sight, the idea of abstracting families of unary operations, which act in the same way is not trivial. Consider again the singular cohomology $H_\textrm{sing}^\bullet(X)$ of  topological spaces. Norman E. Steenrod  constructed in \cite{Steenrod47} natural unary operations, the \emph{Steenrod squares} $\lbrace Sq^i\rbrace_{i\ge 1}$ in characteristic $2$, which act on the singular cohomology $H_\textrm{sing}^\bullet(X)$. They are higher operators which come from the higher homotopies for the commutativity of the cup product. 
They are functorial and always satisfy the same relations, the \emph{Adem relations}, for any topological space $X$. 
So, Henri Cartan introduced in \cite{Cartan55bis} an algebra 
$$\mathcal{A}_2:=T(\lbrace Sq^i\rbrace_{i\ge 1})/(R_{{Adem}}) $$ 
naturally called the \emph{Steenrod algebra}.
It encodes all the Steenrod squares at once. Proving a functorial result for all Steenrod operations amounts to proving only one statement on the level of  the Steenrod algebra. 
\end{exam}

\subsection{Operads and multilinear representation theory}
We aim now at developing the same kind of theory but with operations having multiple inputs this time. On the right hand side, we consider the family $$\End_V:=\lbrace \Hom(V^{\otimes n}, V) \rbrace_{n\ge 0}$$ of multilinear maps. 
The goal now is to find what kind of mathematical object $\Po$ should be put on the left hand side: 
$$\Phi \ : \  \text{?`} \ \Po  \  \text{?} \to \End_V\ . $$

Let us extract the properties of $\End_V$ as follows. 

\begin{itemize}
\item[$(i)$] First, it is made of a family of vector spaces indexed by the integers, so will be $\Po:=\lbrace \Po_n \rbrace_{n\in \NN}$. 

\item[$(ii)$] Then, the space $\End_V$ of multilinear maps is equipped with the classical composition of functions. Given a first level of $k$ operations and one operation with $k$ inputs, one can compose them in the following way 

%\begin{figure}[h]
%\centering
\begin{center}
\includegraphics[scale=0.22]{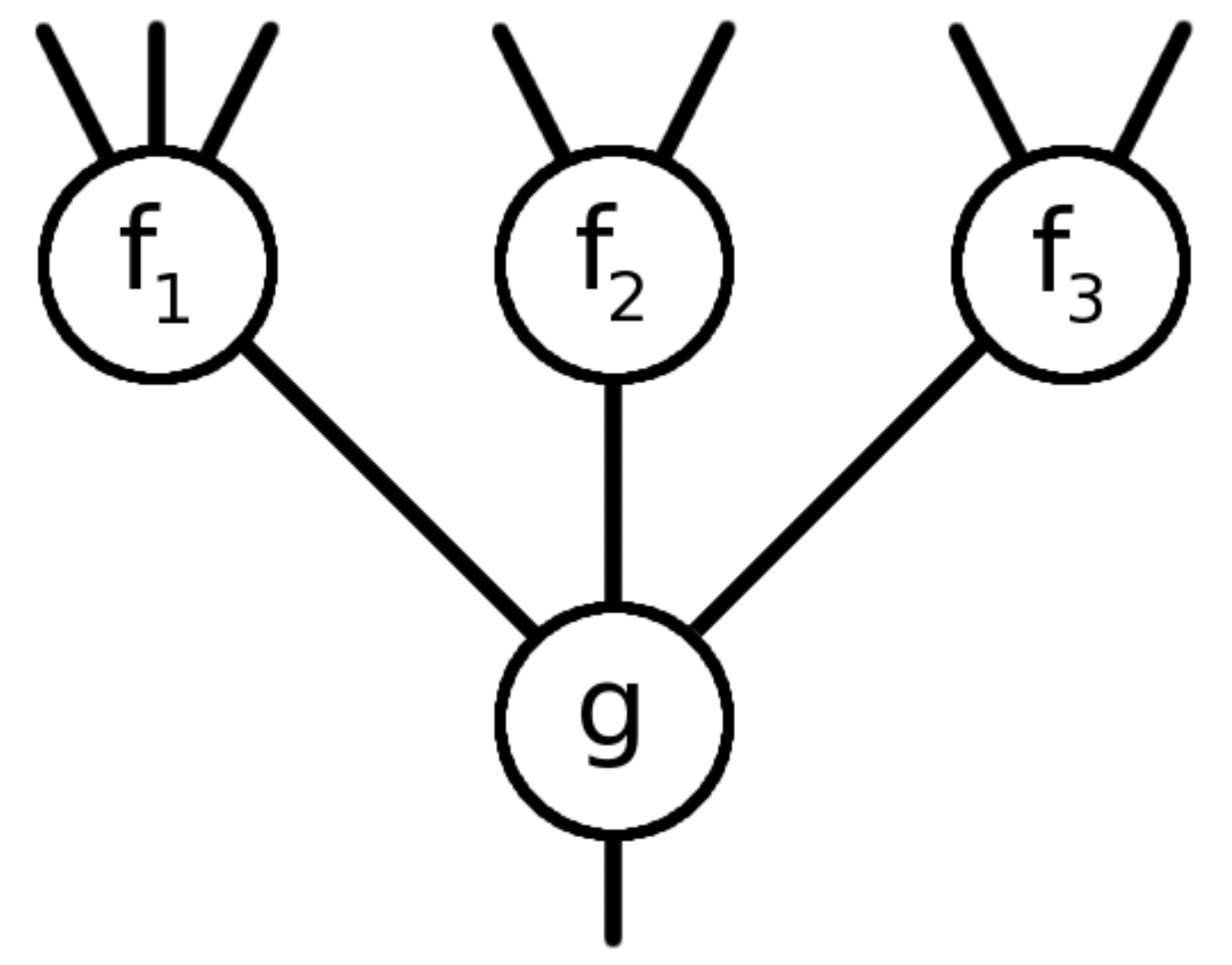} 
\end{center}
%\caption{Definition of  $\mu_n$}
%\label{Fig:CompoEnd}
%\end{figure}

\noindent
Mimicking this composition, we require $\Po$ to be equipped with \emph{composite maps}: 
\begin{eqnarray*}
\gamma_{i_1, \ldots, i_k} \ : \ \Po_k \otimes \Po_{i_1} \otimes \cdots \otimes  \Po{i_k} &\to& \Po_{i_1+\cdots+i_k} \\
\mu \otimes \nu_1 \otimes \cdots \otimes \nu_k & \mapsto & \gamma(\mu; \nu_1, \ldots , \nu_k)
\end{eqnarray*}
that we of course represent similarly. 

\noindent
The composition of functions is associative in the following way: if we consider $3$ levels of functions, composing the two upper levels and then the result with the third one or beginning with the two lower levels and then compositing with the first one, this gives the same function. This associativity condition reads as follows.

%\begin{figure}[h]
%\centering
\begin{center}
\includegraphics[scale=0.2]{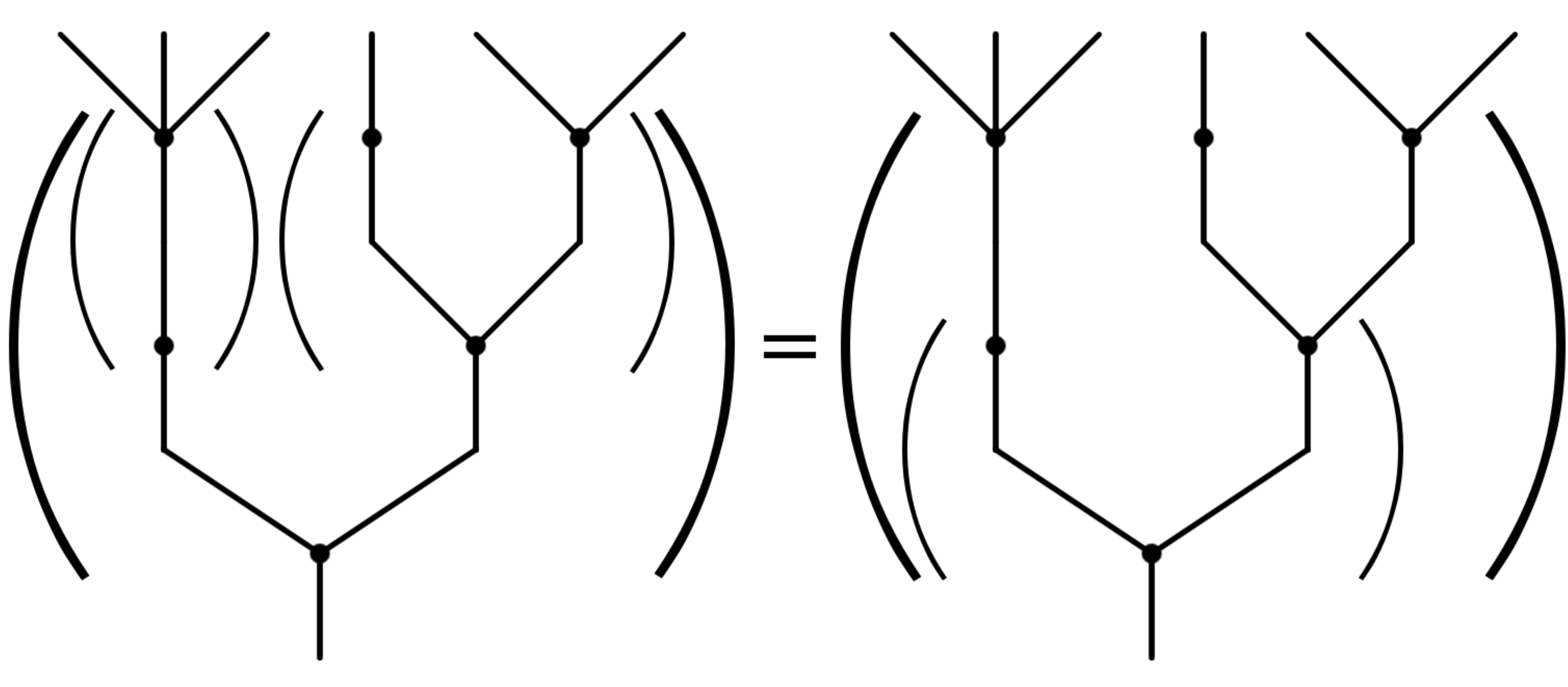} 
\end{center}
%\caption{Definition of  $\mu_n$}
%\label{Fig:AssocCompo}
%\end{figure}

\item[$(iii)$]  Finally, we have the identity morphism $id_V$ in $\Hom(V,V)$. So, we ask for a particular element $\I$ in $\Po(1)$, which behaves as a unit for the composite maps:
$$\gamma_{1, \ldots, 1} : \mu \otimes \I \otimes \cdots \otimes \I \mapsto \mu \quad \text{and} \quad \gamma_k : \I\otimes \mu \mapsto \mu\ . $$
\end{itemize}

\begin{defi}[Nonsymmetric operad]
A \emph{nonsymmetric operad} $\Po$, \emph{ns operad} for short,   is a family $\lbrace \Po_n \rbrace_{n\in \NN}$ of vector spaces with an element $\I\in \Po_1$ and  endowed with associative and unital composite maps 
$$\gamma_{i_1, \ldots, i_k} \ : \ \Po_k \otimes \Po_{i_1} \otimes \cdots \Po_{i_k} \to \Po_{i_1+\cdots+i_k}\ . $$
\end{defi}

\begin{exams}$ \ $

\begin{itemize}
\item[$\diamond$] The \emph{endomorphism operad} $\End_V:=\lbrace \Hom(V^{\otimes n}, V)\rbrace_{n\ge 0}$ is the mother of (nonsymmetric) operads. 

\smallskip

\item[$\diamond$] Let $(A, \mu, 1_A)$ be a unital associative algebra. We consider the family $\mathcal A$ defined by $\CA_1:=A$ and $\CA_n:=0$ for $n\neq 1$ and we set $\I:=1_A$. All the composite maps are equal to zero, except for $\gamma_{1}:=\mu$. So a unital associative algebra induces a ns operad. In the other way round, any ns  operad concentrated in arity $1$ is a unital associative algebra. \end{itemize}
\end{exams}

Since the notion of a nonsymmetric operad is  a generalization of the notion of unital associative algebra, we will be able to extend many results from associative algebras to operads.

\begin{defi}[Morphism of ns operads]
A \emph{morphism} $f: \Po \to \Qo$ between two ns operads $(\Po, \gamma, \I_\Po)$ and $(\Qo, \zeta, \I_\Qo)$ is a family $\lbrace f_n : \Po_n \to \Qo_n \rbrace_{n\ge 0}$ of linear morphisms which preserve the units 
$f_1(\I_\Po)=\I_\Qo$ and which commute with the composite maps
$$ \zeta_{i_1, \ldots, i_k} \circ (f_k\otimes f_{i_1} \otimes \cdots \otimes f_{i_k})= f_{i_1+\cdots+i_k} \circ \gamma_{i_1, \ldots, i_k} \ .$$
\end{defi}

We now define the representations of a nonsymmetric operad $\Po$ in a vector space $V$ by the  morphisms of nonsymmetric operads:
$$\Phi : \Po \to \End_V \ . $$
Such a map associates to any ``formal'' operation $\mu \in \Po_n$ a ``concrete'' $n$-multilinear operation $\Phi(\mu)\in \Hom(V^{\otimes n}, V)$.

%\begin{figure}[h]
%\centering
\begin{center}
\includegraphics[scale=0.2]{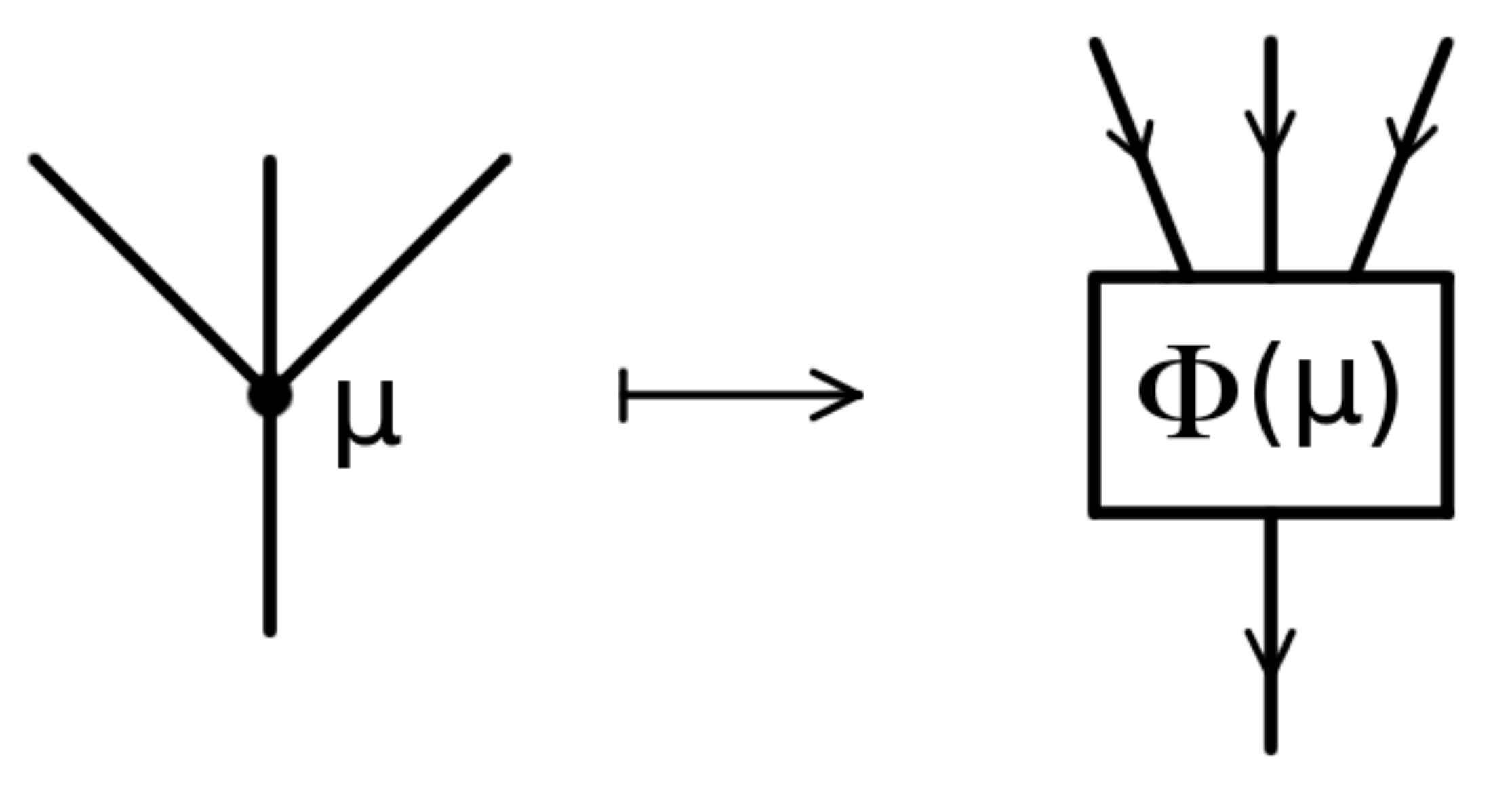} 
\end{center}
%\caption{Definition of  $\mu_n$}
%\label{Fig:AssocCompo}
%\end{figure}

The operations in the image of $\Phi$ satisfy the same relations as that of $\Po$. In the other way round, given a certain category of algebras, like associative algebras, commutative algebras or Lie algebras, for instance, we are able to encode the space of all possible operations into one operad, as follows.

\begin{defi}[$\Po$-algebra]
Let $\Po$ be an operad. A \emph{$\Po$-algebra structure} on a vector space $V$ is a morphism of ns operads 
$\Phi : \Po \to \End_V$.
\end{defi}

\subsection{The examples of $As$ and $uAs$} Let us see how this works on an example. We go back to the example of Section~\ref{Sec:Alg+Homo}: the category of associative algebras, not necessarily unital. Let us now try to figure out the operad, which models this category. Given $n$ elements $a_1, \ldots, a_n$ of an associative algebra $A$, what are the operations, induced by the associative binary product, acting on them ? Keeping the variables in that order, that is without permuting them, there is just one because any bracketing will produce the same result under the associativity relation. Hence, we decide to define 
$$As_n:=\KK \  \vcenter{\hbox{\includegraphics[scale=0.2]{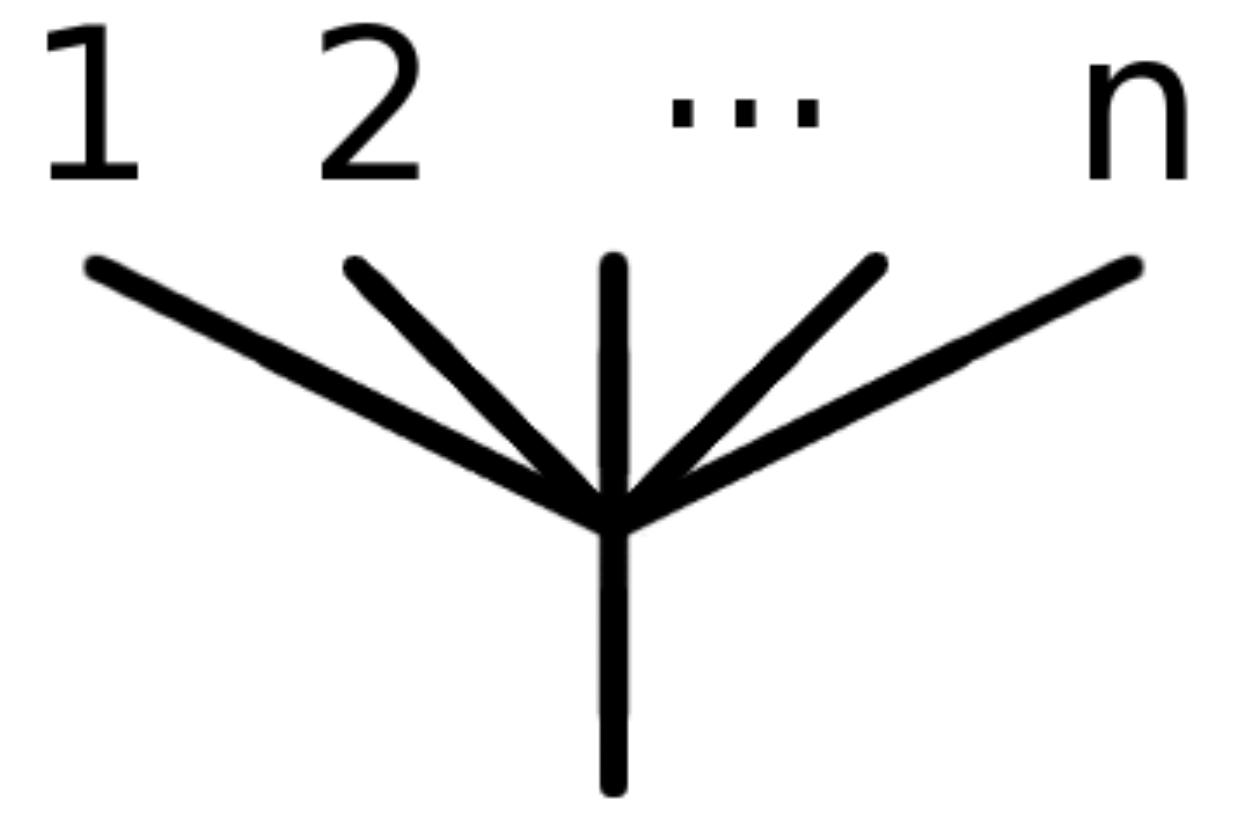}}}$$
the one-dimensional vector space, for any $n\ge 1$. For an algebra over an operad, the set $\Po_0$ is sent to $\Hom(\KK, V)$, since $V^{\otimes 0}=\KK$ by convention. Therefore it stands for the set of particular elements in $V$. In the case of associative algebras,  we have none of them. There is no unit for instance; so we define $As_0:=0$. 

The composite map $\gamma_{i_1, \ldots, i_k} : As_k\otimes As_{i_1} \otimes \cdots \otimes As_{i_k} \to As_{i_1+\cdots+i_k}$ is the multiplication of scalars, i.e. the isomorphism $\KK\otimes \KK\otimes \cdots \otimes \KK \cong \KK$.

\begin{ex}
Show that, with this definition, the data $(As, \gamma, \I)$ forms a nonsymmetric operad.
\end{ex}

The next proposition shows that we have done a good job.

\begin{prop}
The category of $As$-algebras is isomorphic to the category of associative algebras.
\end{prop}

\begin{proof}
Let $\Phi : As \to \End_V$ be a representation of the ns operad $As$. Since $As_0=0$ and $As_1=\KK\,  \I$, they code for nothing. Let us call the image of $\Y$ by $\mu:=\Phi(\Y)$. It defines a binary product on $V$. 
The composite of $\Y \otimes (\Y \otimes I)$ and $\Y \otimes (I \otimes \Y)$ in $As$ give the same result. Therefore the composite of their respective image in $\Hom(V^{\otimes 3}, V)$ are equal, which gives the associativity of $\mu$. 
As usual, we leave the rest of the proof to the reader as a good exercise. 
\end{proof}

If we want to model the category of \emph{unital} associative algebras this time, we define the operad $uAs$ in the same way, except for $uAs_0:=\KK \, \mathfrak{1}$. 

\begin{ex}
Show that $(uAs, \gamma, \I)$ forms a nonsymmetric operad and that $uAs$-algebras are unital associative algebras.
\end{ex}

\textsc{Remark.} At this point of the theory, some readers might be confused by the fact that unital associative algebras appeared twice: on the one hand, the notion of a nonsymmetric operad is a generalization of the notion of unital associative algebra and, on the other hand, there is one operad $uAs$, which encodes the category of unital associative algebras. So, one should be careful. In these two cases, the category of unital associative algebras does not play the same role and is not placed on the same footing.

\subsection{Symmetric operads}\label{subsec:SymmetricOp}
Let us now extend the definition of a nonsymmetric operad in order to take care of the possible symmetries of the operations that we aim to encode. 

First, we notice that the space $\Hom(V^{\otimes n}, V)$ of $n$-multilinear maps carries a natural right action of the symmetric group $\Sy_n$, induced by the permutation of the input elements:
$$f^\sigma(v_1, \ldots, v_n):=f(v_{\sigma^{-1}(1)}, \ldots, v_{\sigma^{-1}(n)})\ .$$
The composition of multilinear functions satisfies natural equivariance properties with respect to this action. 

\begin{ex}
Write down these equivariance properties.
\end{ex}

\begin{defi}[Symmetric operad]
A \emph{symmetric operad} $\Po$, or simply an \emph{operad}, is a family $\lbrace \Po(n) \rbrace_{n\in \NN}$ of right $\Sy_n$-modules  with an element $\I\in \Po(1)$ and  endowed with associative, unital and equivariant composite maps 
$$\gamma_{i_1, \ldots, i_k} \ : \ \Po(k) \otimes \Po(i_1) \otimes \cdots \otimes \Po(i_k) \to \Po(i_1+\cdots+i_k)\ . $$
\end{defi}

We refer to this definition as to the \emph{classical} definition of an operad, since it coincides with the original definition of J. Peter\ May in \cite{May72}.
The definition of a morphism of symmetric operads and the definition of an algebra over a symmetric operad are the equivariant extensions of the nonsymmetric ones. 
So, forgetting the action of the symmetric groups, one defines a functor from symmetric operads to nonsymmetric operads. 

\begin{ex}[Partial definition]
Show that the definition of an operad structure on a family of $\Sy_{n}$-modules $\lbrace \Po(n)\rbrace_{n\geq 0}$ is equivalent to the data of \emph{partial compositions}
$$ \circ_{i}: \Po(m)\otimes \Po(n) \to \Po(m-1+n),\textrm{ for } 1\leq i \leq m,$$
satisfying 
\begin{itemize}
\item[$\diamond$]
the equivariance with respect to the symmetric groups (to be made precise), 

\item[$\diamond$] the axioms:
\begin{displaymath}
\left\{\begin{array}{crcll}
(\textrm{I})&(\lambda \circ_i \mu)\circ_{i-1+j}\nu  &=& \lambda \circ_i (\mu\circ_{j}\nu ), &\mathrm{for }\ 1\leq i\leq l, 1\leq j\leq m,  \\
(\textrm{II})&(\lambda \circ_i \mu)\circ_{k-1+m}\nu  &=& (\lambda \circ_k\nu)\circ_{i}\mu , &\mathrm{for }\  1\leq i < k\leq l, \\
\end{array}\right .
\end{displaymath}
for any $\lambda \in \Po(l), \mu\in \Po(m), \nu\in \Po(n)$, 

\item[$\diamond$] and having  a unit element $\I \in \Po(1)$.
\end{itemize}
We refer to this definition as to the \emph{partial} definition of an operad.
\end{ex}

\begin{exams}$ \ $

\begin{itemize}

\item[$\diamond$]  Let us define the symmetric operad $Com$ (respectively $uCom$), in the same way as the nonsymmetric operad $As$ (respectively $uAs$), by the one-dimensional spaces
$$
Com(n):=\KK \ 
\vcenter{\hbox{\includegraphics[scale=0.2]{FIG14Corolla.pdf}}}$$
but endowed with the trivial representation of the symmetric groups, this time. As usual, we leave it, as a good exercise, to the reader to check that these data form an operad. Prove also that the category of $Com$-algebras (respectively $uCom$-algebras) is isomorphic to the category of commutative algebras (respectively unital commutative algebras).

\smallskip

\item[$\diamond$]  In contrast to the operad $Com$ of commutative algebras, the various known bases of the space of $n$-ary operations of the operad $Lie$ of Lie algebras 
do not  behave easily with respect to both the operadic composition and to the action of the symmetric groups.
 
\noindent
On the one hand, we know that the right module $Lie(n)$ is the induced representation $\mathrm{Ind}^{\Sy_{n}}_{\ZZ/n\ZZ}(\rho)$, where $(\rho)$ is the one-dimensional representation of the cyclic group given by an irreducible $n$th root of unity. On the other hand, the following bases of $Lie(n)$, due to Guy Melan\c con and Christophe Reutenauer in \cite{MelanconReutenauer96}, behaves well with respect to the operadic composition. 

\noindent
We consider the subset $MR_n$ of planar binary trees with $n$ leaves such that, at any given vertex, the left-most  (resp.\ right-most) index is the smallest (resp.\ largest) index among the indices involved by this vertex. 
Here are the elements of $MR_2$ and $MR_3$.
%\begin{center}
$$
\vcenter{\hbox{\includegraphics[scale=0.2]{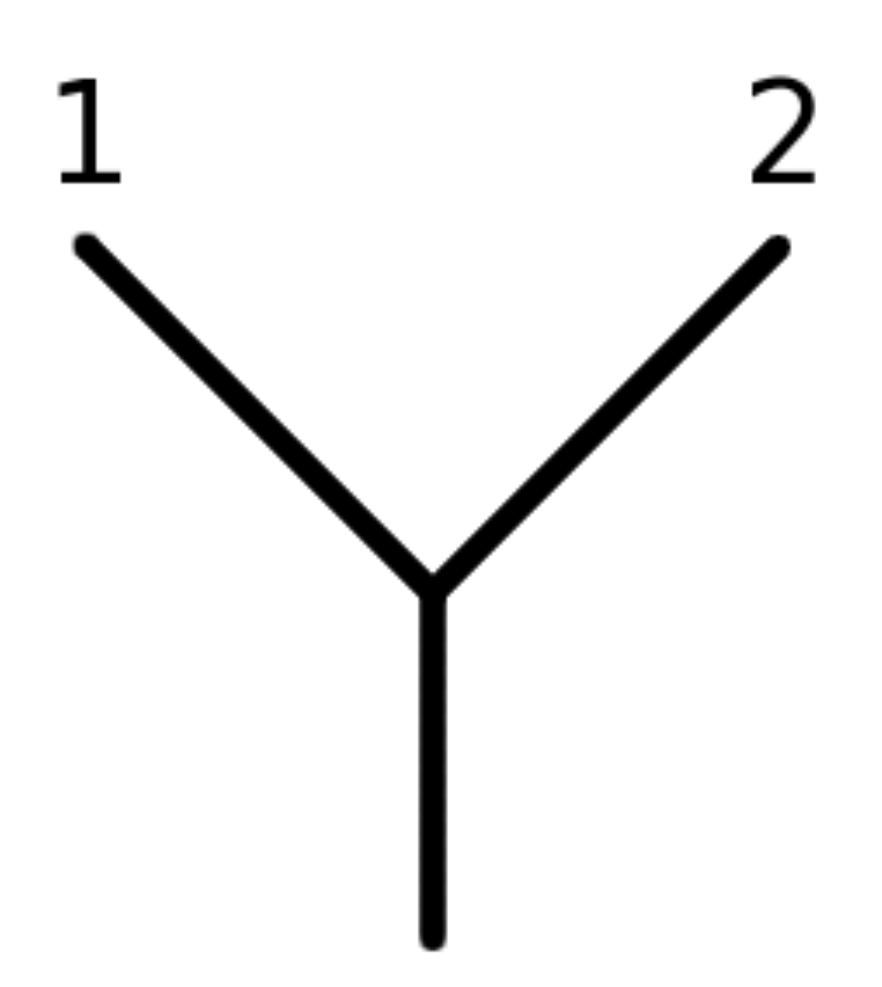}}}
\qquad
\vcenter{\hbox{\includegraphics[scale=0.2]{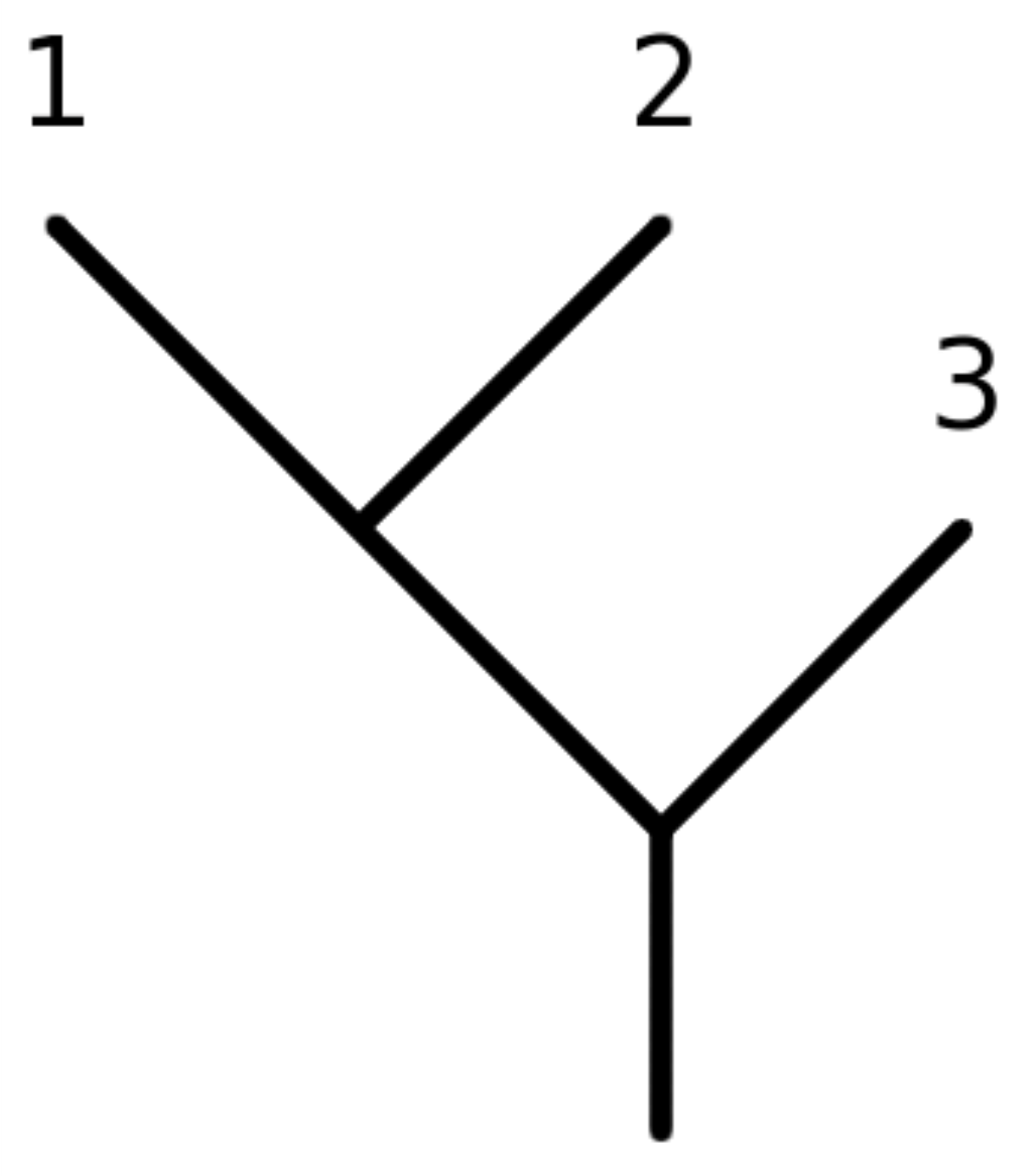}}}
\qquad
\vcenter{\hbox{\includegraphics[scale=0.2]{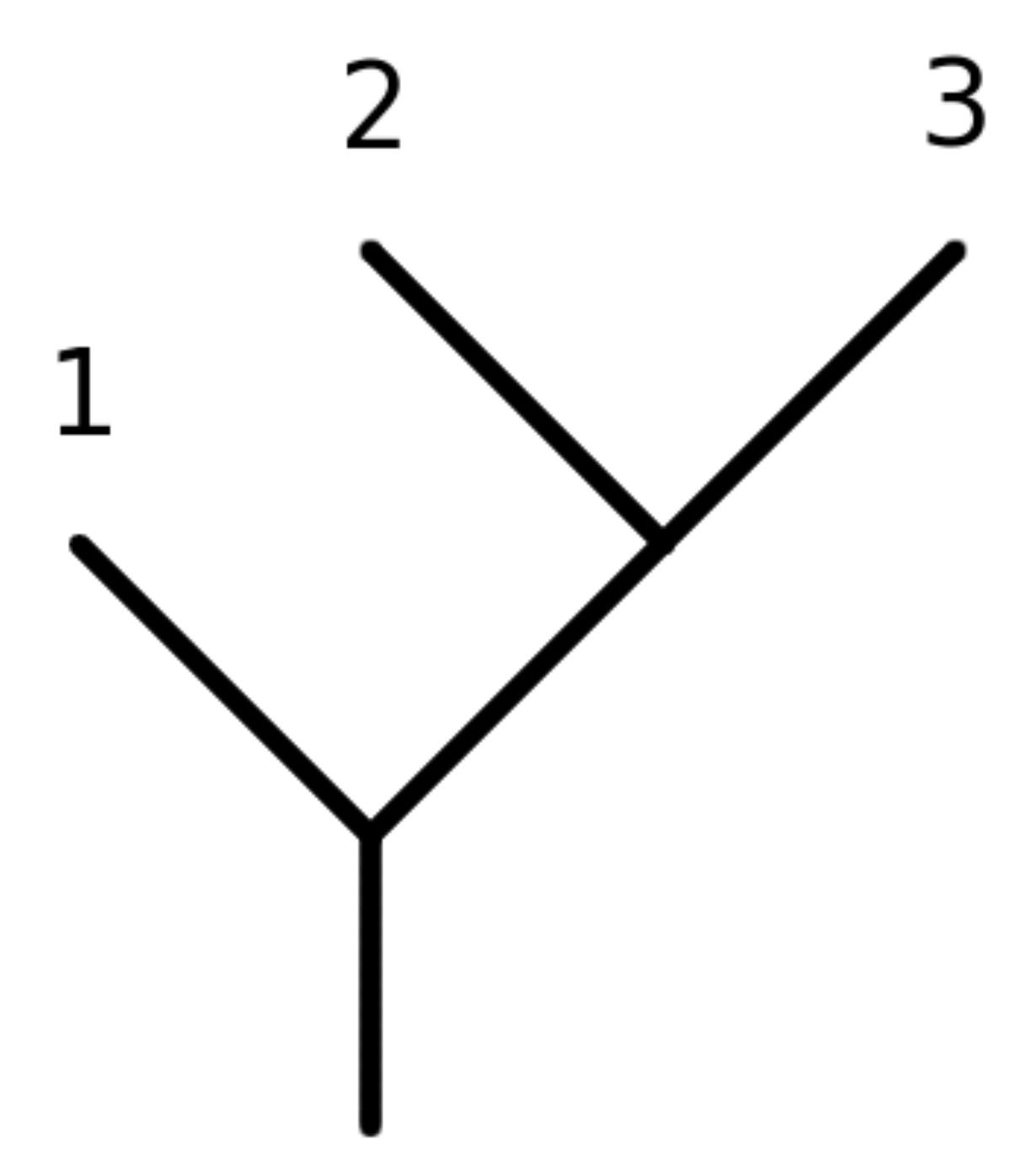}}}$$
%\end{center}
We define the partial compositions  $t\circ_{i}s$ of two such trees
by grafting the root of $s$ on the leaf of $t$ with label $i$ and by shifting the indices accordingly. We leave it to the reader to verify that this endows $\lbrace \KK\,  {MR}_n  \rbrace$ with a nonsymmetric operad structure. 

\noindent
For more properties about this basis of the operad $Lie$, we refer the reader to the end of Section~\ref{subsec:Rewriting}.
\end{itemize}
\end{exams}

\begin{ex}
One can also encode the category of associative algebras with a symmetric operad that we denote $Ass$, where the second ``s'' stands for ``symmetric''. Since the multilinear operations of an associative algebra have in general no symmetry, the component of arity $n$ of the operad $Ass$ is the regular representation of the symmetric group: $Ass(n):=\KK[\Sy_n]$. We leave it to the reader to make the composite maps explicit. 
\end{ex}

\subsection{Changing the underlying monoidal category}\label{subsec:Underlying}
One can also consider operads not only over the category of vector spaces but over other categories as follows. 

In order to define the notion of an operad, we used the fact that the category $(\mathsf{Vect}, \otimes)$ of vector spaces is a \emph{monoidal category}, when endowed with the tensor product $\otimes$. To write down the associativy property of the composite maps $\gamma$, one needs this monoidal category to be \emph{symmetric} because we have to commute terms separated by tensors. 

In the end, one can define the notion of an operad over any symmetric monoidal category. To name but a few:

%\medskip

\noindent
\begin{center}
\begin{tabular}{|cl|c|}
\hline
\multicolumn{2}{|c|}{\textsc{Symmetric monoidal category}} & \textsc{Type of operad}  \\
\hline
\hline 
Vector spaces& $(\mathsf{Vect}, \otimes)$ &  Linear operads   \\
\hline
Graded modules& $(\mathsf{gr \ Mod}, \otimes)$ & Graded operads   \\
\hline 
Differential graded modules& $(\mathsf{dg \ Mod}, \otimes)$  & Differential graded operads   \\
\hline
Sets & $(\mathsf{Set}, \times)$ & Set-theoretic operads   \\
\hline
Topological spaces& $(\mathsf{Top}, \times)$ & Topological operads   \\
\hline
Simplicial sets& $(\mathsf{Set_\Delta}, \times)$ & Simplicial operads   \\
\hline
\end{tabular}
\end{center}

\subsubsection*{$\diamond$ Set-theoretical operad: the example of monoids}
 Let us define the set-theoretic ns operad $Mon$, in the same way as the ns operad $uAs$, by the following one-element sets: 
$$Mon_n:=
\left\lbrace\vcenter{\hbox{\includegraphics[scale=0.2]{FIG14Corolla.pdf}}}\right\rbrace\ ,$$
for $n\in \NN$. 
In this case, the category of $Mon$-algebras is isomorphic to the category of monoids.

\subsubsection*{$\diamond$ Topological operad: the example of the little discs}
 The mother of topological operads is the \emph{little discs operad} $\DDD^2$, which  is a topological operad defined as follows. The topological space $\DDD^2(n)$ is made up of the unit disc (in $\CC$) with $n$  
 sub-discs with non-intersecting interiors. So, an element on $\DDD^2(n)$ is completely determined by a family of $n$ continuous maps $f_{i}: S^{1}\to D^{2}, i=1, \ldots, n$, satisfying the non-intersecting condition, see Figure~\ref{LD1}.

\begin{figure}[!h]
\centering
\includegraphics[scale=0.20]{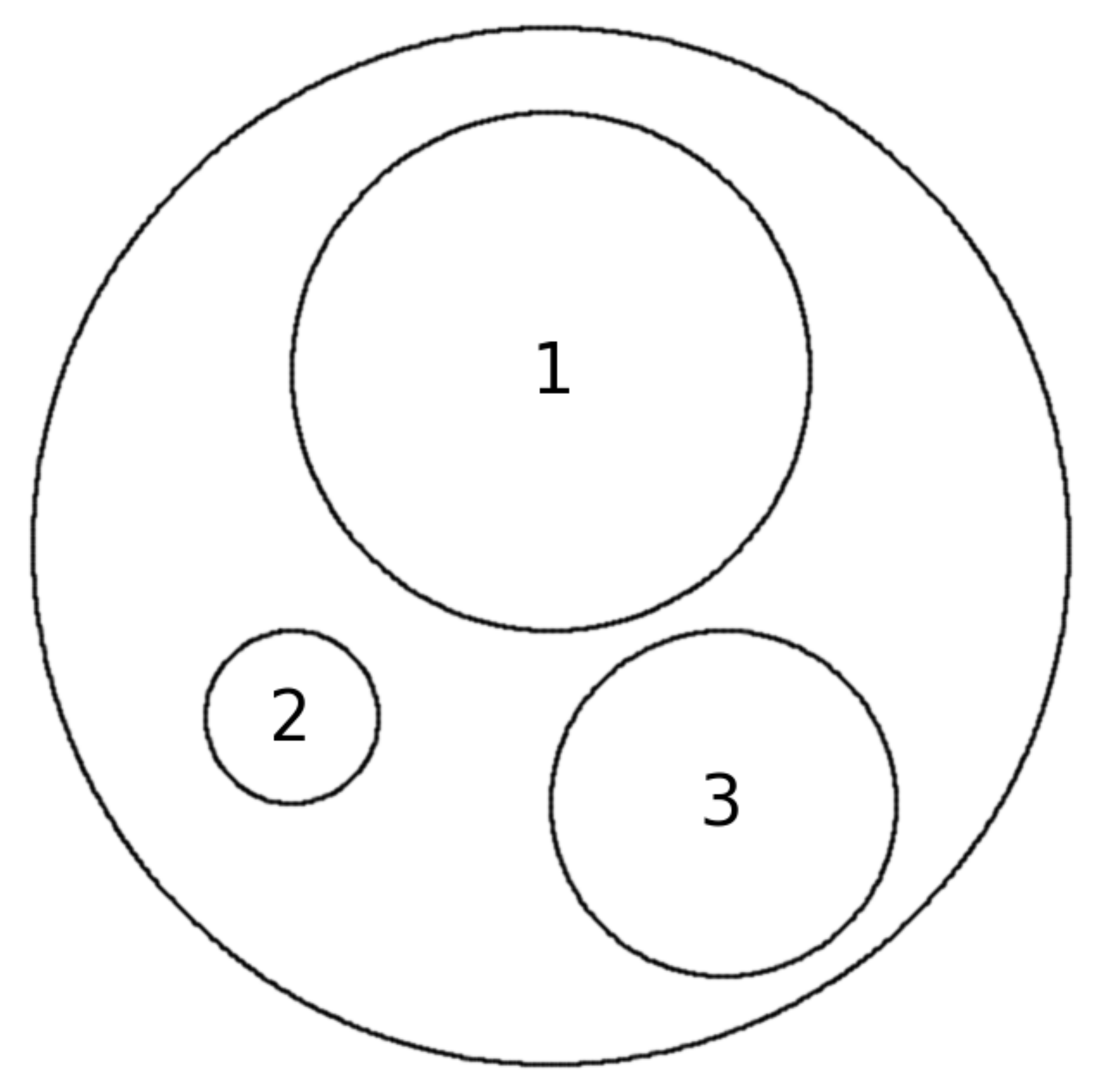}
\caption{Little discs configuration in $\DDD^2(3)$} \label{LD1}
\end{figure}

\noindent
The enumeration of the interior discs is part of the structure. The operadic composition is given by insertion of a disc in an  interior disc. The symmetric group action is given by permuting the labels. Figure~\ref{LDcompo} gives an example of a partial composition. 

\begin{figure}[!h]
\begin{eqnarray*}
&\includegraphics[scale=0.20]{Littlediscs1.pdf}\circ_1
\includegraphics[scale=0.20]{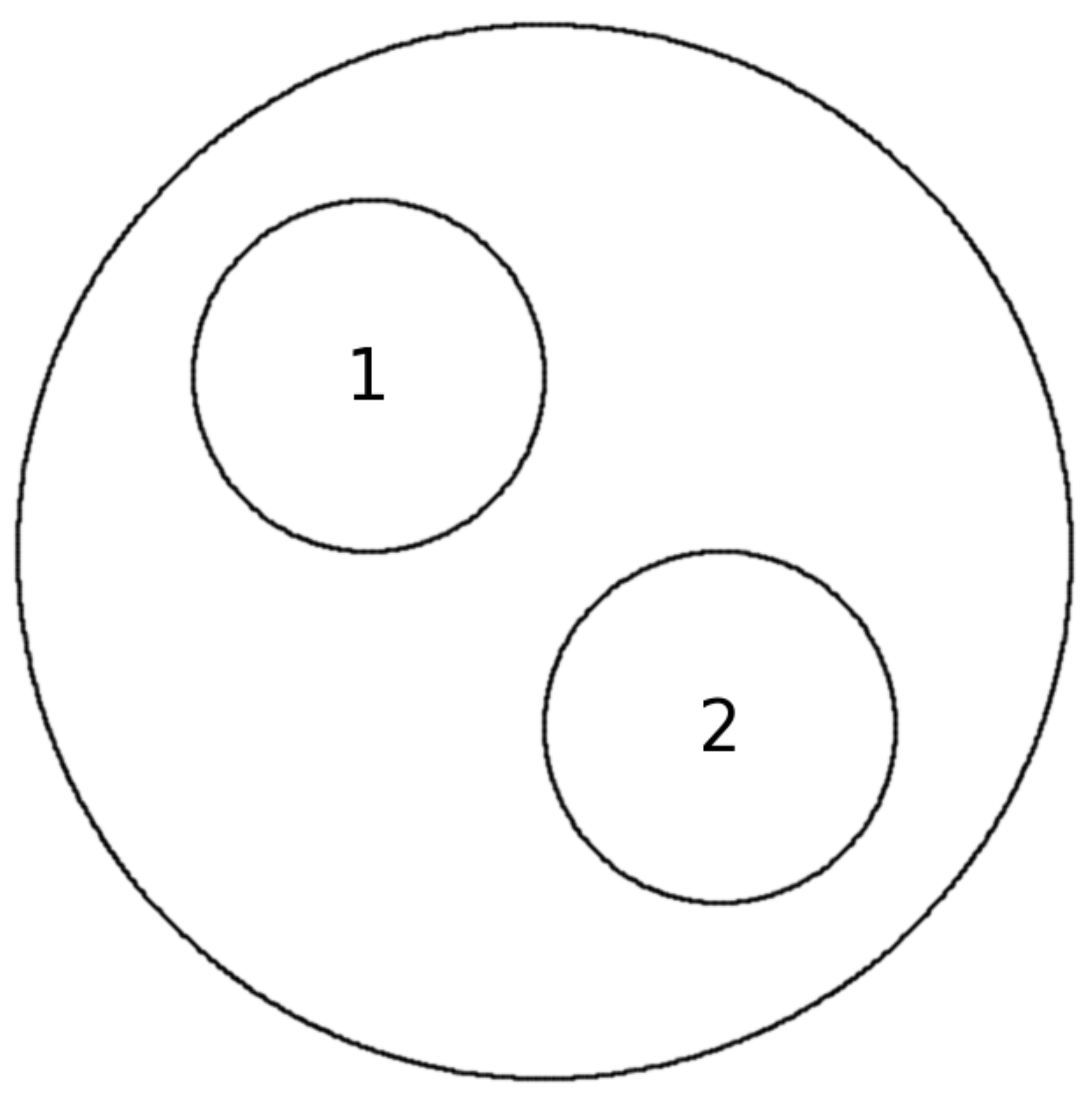}&\\
&=\includegraphics[scale=0.20]{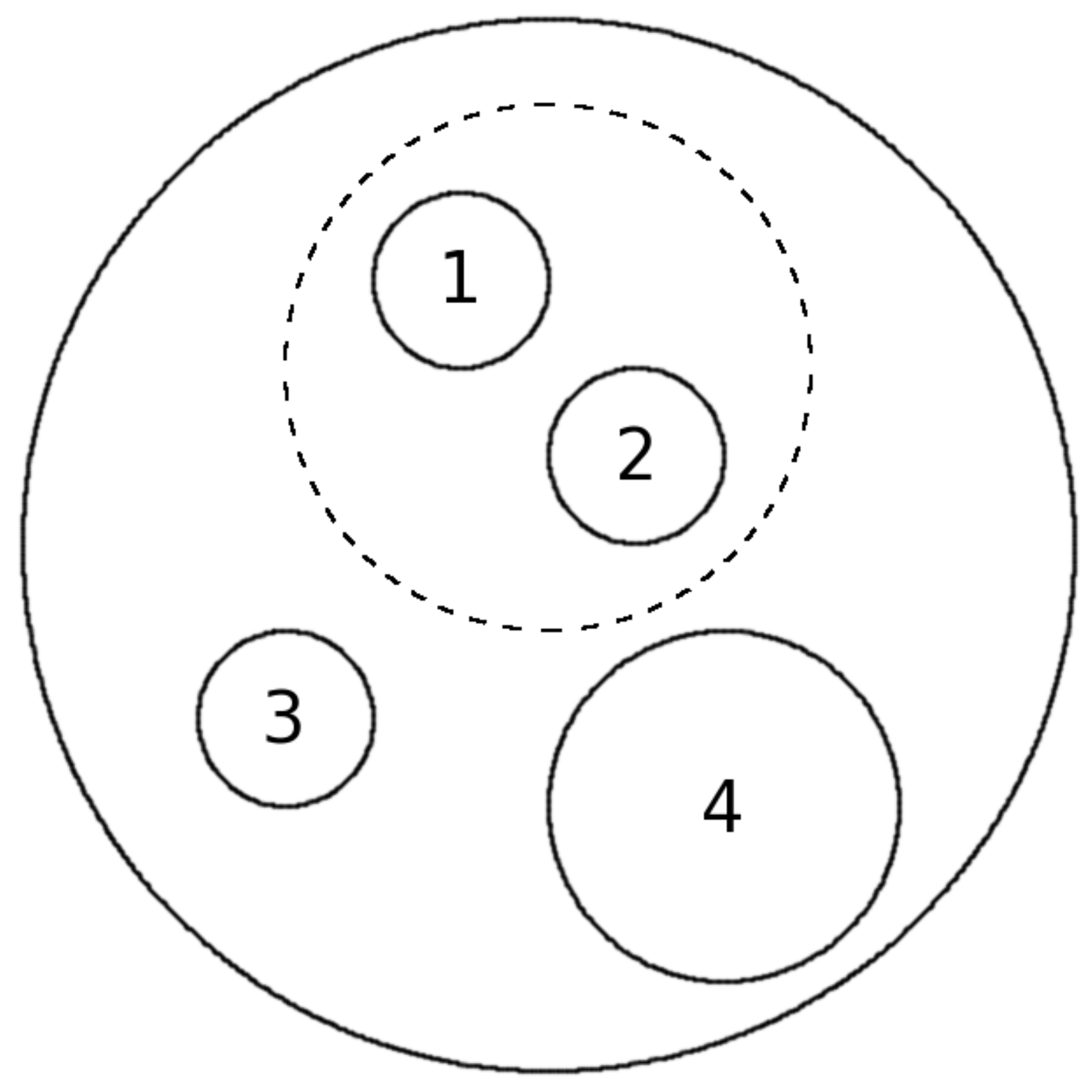}&
\end{eqnarray*}
\caption{Exemple of partial composition in the little discs operad} \label{LDcompo}
\end{figure}
It is clear how to proceed to define the \emph{little $k$-discs operad} $\DDD^k$, for any $k\in \NN$.
In the case $k=1$, one gets the \emph{little intervals} operad. 

\begin{ex}
Prove that any $k$-fold loop space $\Omega^{k}(Y)$ is an algebra over $\DDD^k$.
\end{ex}

\noindent
The main property of the little $k$-discs operad is the following result, which goes  the other way round. 
 
 \begin{theo}[Recognition principle \cite{BoardmanVogt73, May72}]
If the connected topological space $X$ is an algebra over the little $k$-discs operad, then it is homotopy equivalent to the $k$-fold loop space of some other pointed space $Y$:
$$X \sim \Omega^{k}(Y).$$ 
\end{theo}

\noindent
So, up to the homotopy relation of topological spaces, the category of $\DDD^k$-algebras is the category of $k$-fold loop spaces.

\subsubsection*{$\diamond$ Differential graded operad: the example of the Gerstenhaber operad}
A \emph{Gerstenhaber algebra}  is a  graded vector space $A$ endowed with 

\smallskip

\noindent
$\triangleright$  a commutative binary product $\bullet$ of degree
$0$,

\noindent
$\triangleright$  a commutative  bracket $\langle\; ,  \, \rangle$  of degree $+1$, i.e. $|\langle a ,  b \rangle|=|a|+|b|+1$, where the notation $|\ |$ stands for the grading of elements.

\smallskip

such that

\smallskip

\noindent
$\rhd$ the product $\bullet$ is associative,

\noindent
$\rhd$ the bracket satisfies the Jacobi identity

$$\langle \langle\; ,  \, \rangle,  \, \rangle \ +\ \langle \langle\; ,  \, \rangle,  \, \rangle.(123) \ +\ \langle \langle\; ,  \, \rangle,  \, \rangle.(321)\ = \ 0,$$

\noindent
$\rhd$ the product $\bullet$ and the bracket $\langle\;
, \,\rangle$ satisfy the Leibniz relation
$$\langle\, \textrm{-} , \textrm{-}\bullet \textrm{-}\, \rangle\ =\
(\langle\,\textrm{-}, \textrm{-}\,\rangle\bullet \textrm{-})\
+ \ (\textrm{-}\bullet \langle\,\textrm{-},\textrm{-}\,\rangle).(12),   $$

We denote by $Ger$ the graded symmetric operad which encodes Gerstenhaber algebras.\\

Any symmetric monoidal functor between two symmetric monoidal categories induces a functor between the associated notions of operads. For instance, the free vector space functor 
$(\mathsf{Set}, \times) \to (\mathsf{Vect}, \otimes)$ is symmetric monoidal. Under this functor, the  set-theoretic ns operad $Mon$ is sent to the linear ns  operad $uAs=\KK Mon$. 
Since we are working over a field, the homology functor 
$$H^\bullet \ : \ (\mathsf{Top}, \times) \to (\mathsf{gr \ Mod}, \otimes) \  ,  $$
defines such a symmetric monoidal functor. 

\begin{theo}[Homology of the little discs operad \cite{Cohen76}]
The homology of the little discs operad is isomorphic to the operad encoding Gerstenhaber algebras
$$H^\bullet(\DDD^2)\cong Ger\ . $$
\end{theo}

\subsection{Other types of Operads}\label{operads}
Depending on the type of algebraic structure that one wants to encode (several spaces, scalar product, trace, operations with multiple inputs and multiple outputs, etc.), it is often possible to extend the previous definition of an operad. 

\subsubsection*{$\diamond$ Colored operad}
Algebraic structures acting on  several underlying spaces can be encoded with colored operads. In this case, the paradigm is given by  $End_{V_1\oplus \cdots \oplus V_k}$.
 An operation in a colored operad comes equipped with a color for each input  and a color for the output. The set of ``colors'' corresponds to the indexing set of the various spaces, for instance $\lbrace 1, \ldots, k\rbrace$. 
 
\begin{ex}
The data of two associative algebras $(A, \mu_A)$ and $(B, \mu_B)$ with a morphism $f : A \to B$ between them can be encoded into a $2$-colored operad denoted $As_{\bullet \to \circ}$. Show that this operad is made up of corollas of three types : only black inputs and output (for the source associative algebra $A$), only white inputs and output (for the target associative algebra $B$), and corollas with black inputs and white output (either for the composite of iterated products in $A$ and then the morphism $f$ or for the composite of tensors of copies of the morphism $f$ and then iterated products in $B$). 
\begin{center}
\includegraphics[scale=0.2]{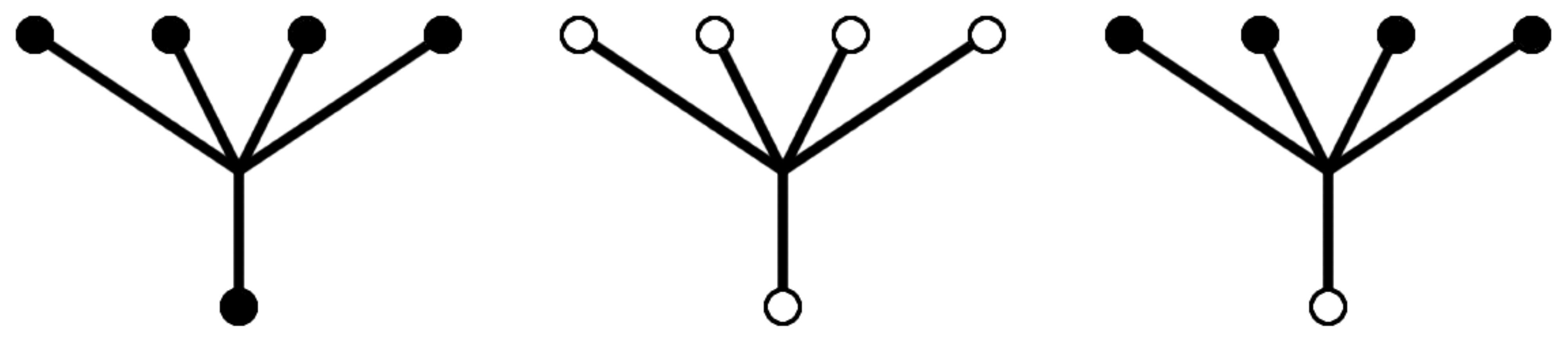} 
\end{center}
\end{ex}
 
\begin{exam}
An interesting example of topological colored operad
is given by  the \emph{operad of tangles} of Vaughan F.R. Jones \cite{JonesV10}, see Figure~\ref{PlanarAlgebraVFRJ}. 
\begin{figure}[h]
\centering
\includegraphics[scale=0.2]{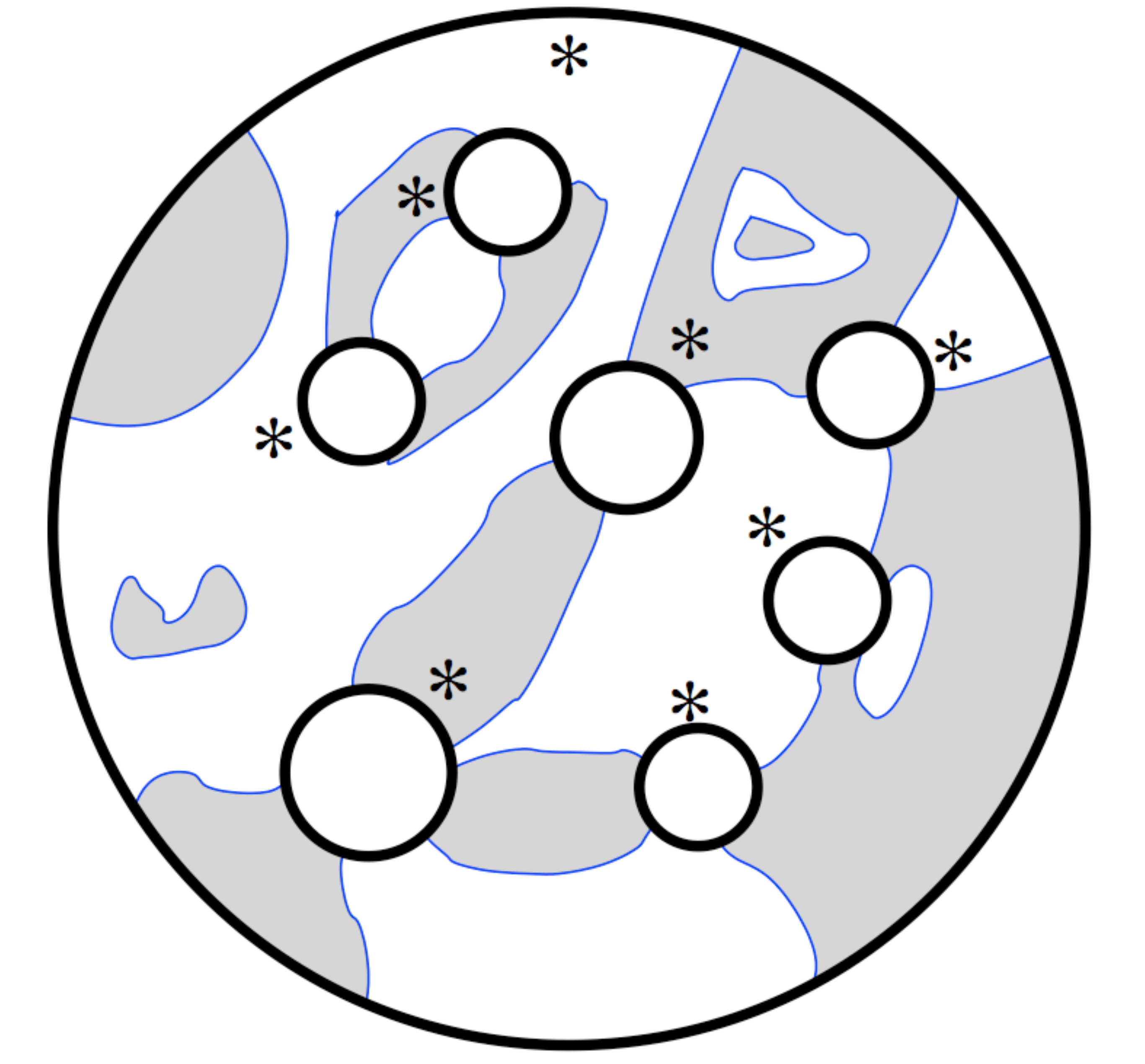}
\caption{A Tangle (Courtesy of V.F.R. Jones)}
 \label{PlanarAlgebraVFRJ}
\end{figure}
(Notice the similarity with the little discs operad.) Algebras over this operad are the \emph{planar algebras} \cite{JonesV99},  notion related to Jones invariants of knots.
\end{exam}

\subsubsection*{$\diamond$ Cyclic operad}
 The purpose of cyclic operads is to encode algebraic structures acting on spaces endowed with a non-degenerate symmetric pairing $\langle -,-\rangle: V\otimes V \to \KK$. In this case, the endomorphism operad $\End_V(n)=\Hom(V^{\otimes n}, V)\cong \Hom(V^{\otimes n+1},\KK)$ is equipped with a natural action of $\Sy_{n+1}$, which extends the action  of $\Sy_n$ and which is compatible with the partial compositions. On the level of cyclic operads, this means that we are given an action of the cycle $(1 \ 2 \ldots n+1)$ which exchanges input and output:

%\begin{figure}[h]
%\centering
\begin{center}
\includegraphics[scale=0.2]{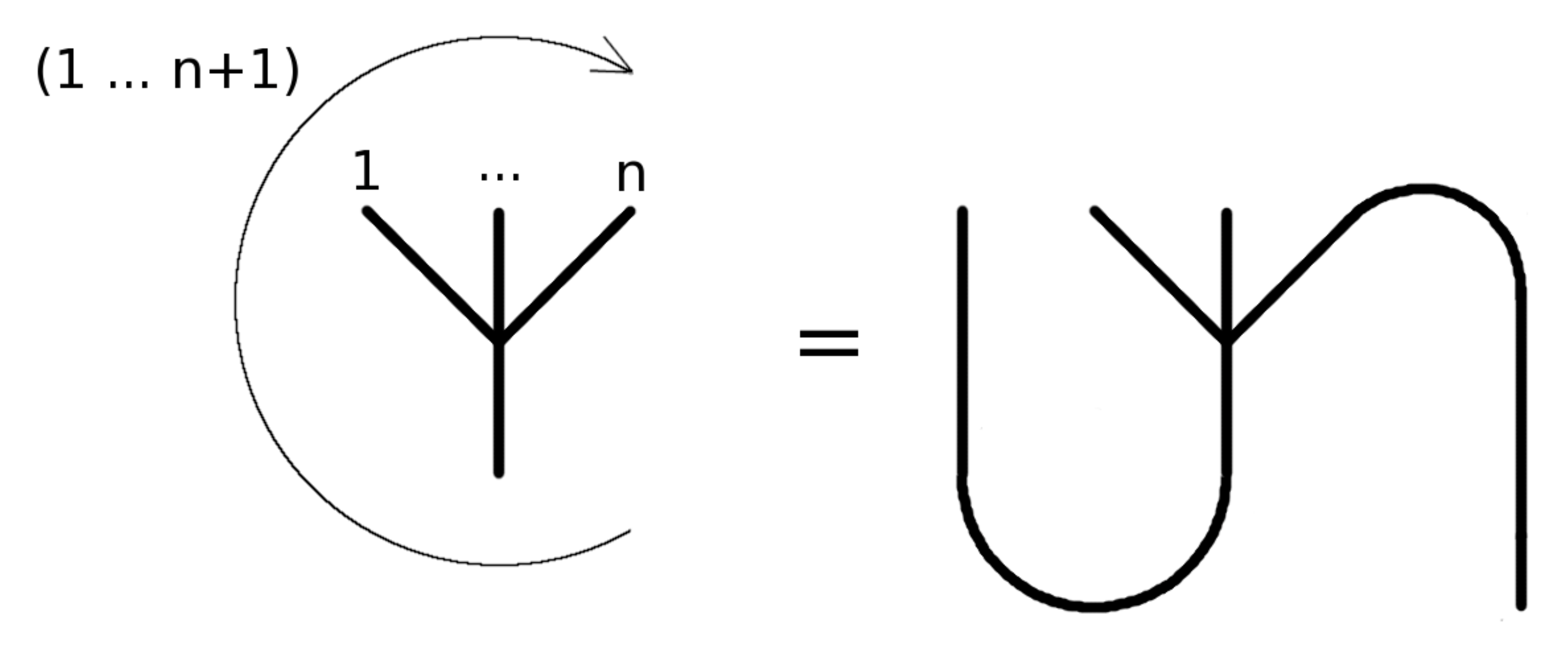} 
\end{center}
%\caption{Definition of  $\mu_n$}
%\label{Fig:Cyclic}
%\end{figure}

\noindent
As usual, the data of an algebra over a cyclic operad is given by a morphism of cyclic operads $\Po \to \End_V$. For example an algebra over the cyclic operad $Ass$ is a finite dimensional associative algebra equipped with a non-degenerate symmetric pairing $\langle -,-\rangle$, which is \emph{invariant}: 
$$\langle ab,c\rangle = \langle a, bc\rangle\ .$$

\noindent
\textsc{Example.} A \emph{Frobenuis algebra}, central notion  in representation theory \cite{CurtisReiner62}, 
is nothing but a cyclic $uAs$-algebra. (The nonsymmetric operad $uAs$-can be endowed with a ``cyclic nonsymmetric operad'' structure, that is with an action of the cyclic groups $\ZZ/(n+1)\ZZ$). A \emph{symmetric Frobenius algebra}, central notion 
 in Topological Quantum Field Theory \cite{Kock04}, is a cyclic $uAss$-algebra.  Let us recall that a Frobenius algebra (resp. symmetric Frobenius algebra) is a unital associative algebra, equipped with a nondegenerate (resp. and symmetric) bilinear form, such that $\langle ab,c\rangle = \langle a, bc\rangle$.

\begin{ex}
Show that the operad $Lie$ can be endowed with a cyclic operad structure. Prove that any finite dimensional Lie algebra $\g$, becomes an algebra over the cyclic operad $Lie$, when equipped with  the Killing form.
\end{ex}

\subsubsection*{$\diamond$ Modular operad}
In a cyclic operad, one makes no distinction between inputs and outputs. So one can compose operations along genus $0$ graphs. The idea of modular operad is to also allow one to  compose operations along any graph of any genus. 

\begin{exam} The  Deligne-Mumford moduli spaces $\overline{\mathcal{M}}_{g, n+1}$  \cite{DeligneMumford69} of stable curves of genus $g$ with $n+1$ marked points is the mother of modular operads. The operadic composite maps are defined by intersecting curves along their marked points: 
\begin{center}
\includegraphics[scale=0.18]{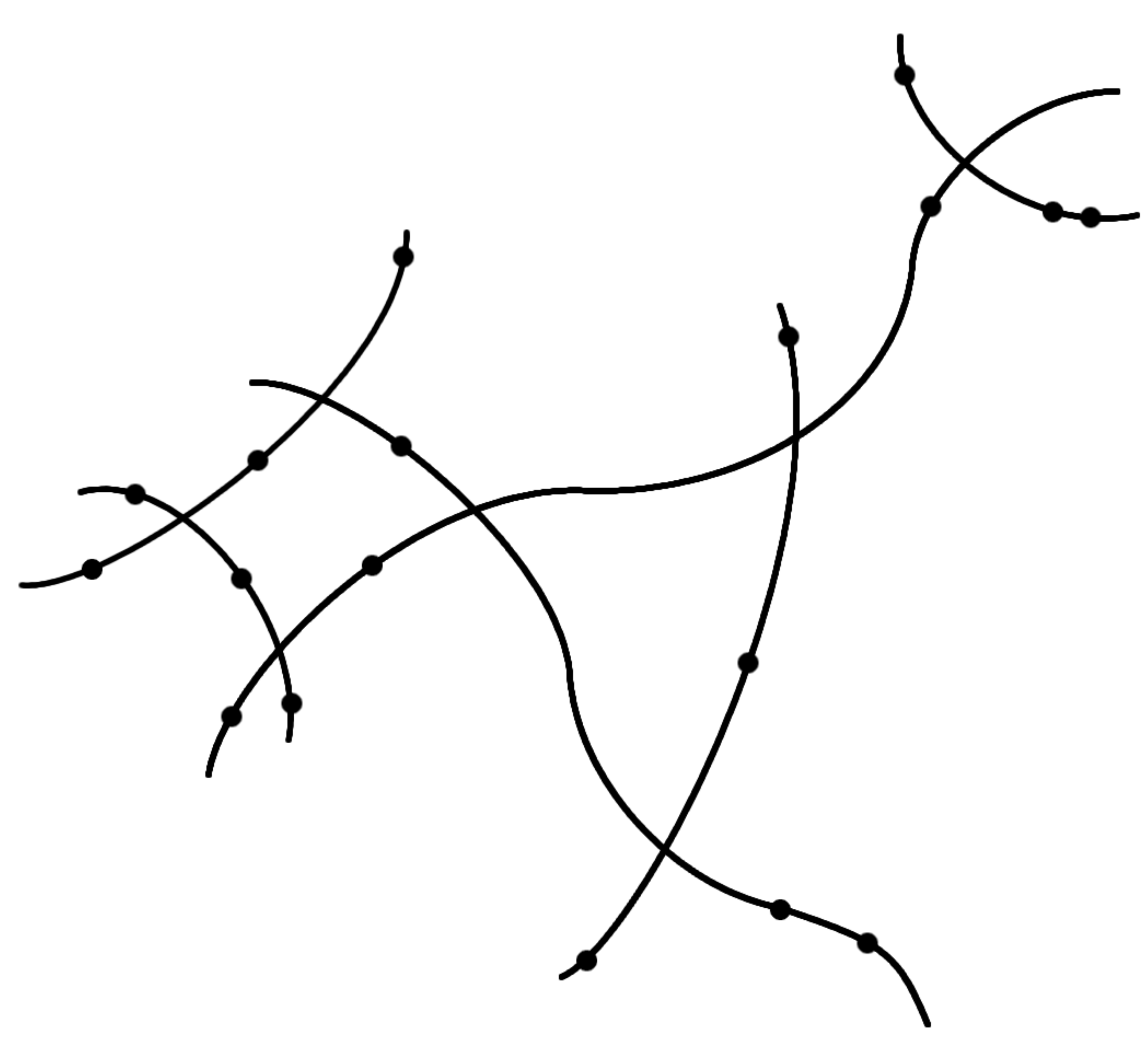} 
\end{center}
The  Gromov-Witten invariants of enumerative geometry endow the cohomology groups $H^\bullet(X)$ of any projective or symplectic variety with a $H_\bullet(\overline{\mathcal{M}}_{g, n+1})$-algebra structure, see \cite{KontsevichManin94}.  A \emph{Frobenius manifold} \cite{Manin99} is an algebra over the cyclic operad $H_\bullet(\overline{\mathcal{M}}_{0, n+1})$. By definition, the \emph{Quantum cohomology ring} is the $H_\bullet(\overline{\mathcal{M}}_{0, n+1})$-algebra structure on $H^\bullet(X)$.
\end{exam}

\subsubsection*{$\diamond$ Properad} A properad $\lbrace \Po(n,m)\rbrace_{m,n\in \NN}$ is meant to encode operations with several inputs and several outputs. But, in contrast to modular operads, where inputs and outputs are confused,  one keeps track of the inputs and the outputs in a properad. Moreover, we consider only compositions along connected graphs as follows.

\begin{center}
\includegraphics[scale=0.2]{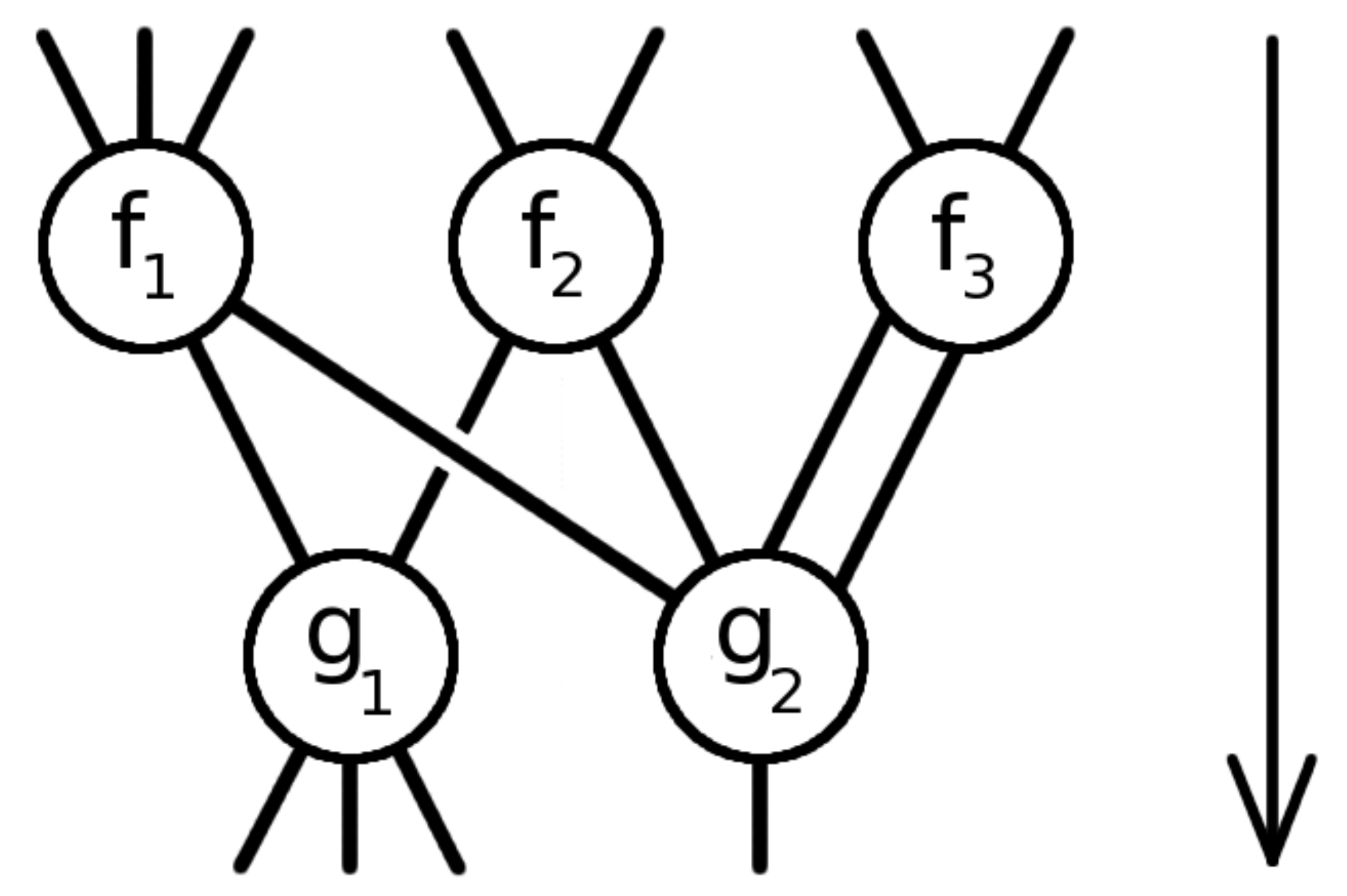} 
\end{center}

\begin{exam} Riemann surfaces, i.e. smooth compact complex curves,  with parametrized holomorphic holes form a properad. We denote by  $\mathcal{R}_{g, n, m}$ its component of genus $g$ with $n$ input holes and $m$ output holes.  The properadic composite maps are defined by sewing the Riemann surfaces along the holes. 
\begin{center}
\includegraphics[scale=0.14]{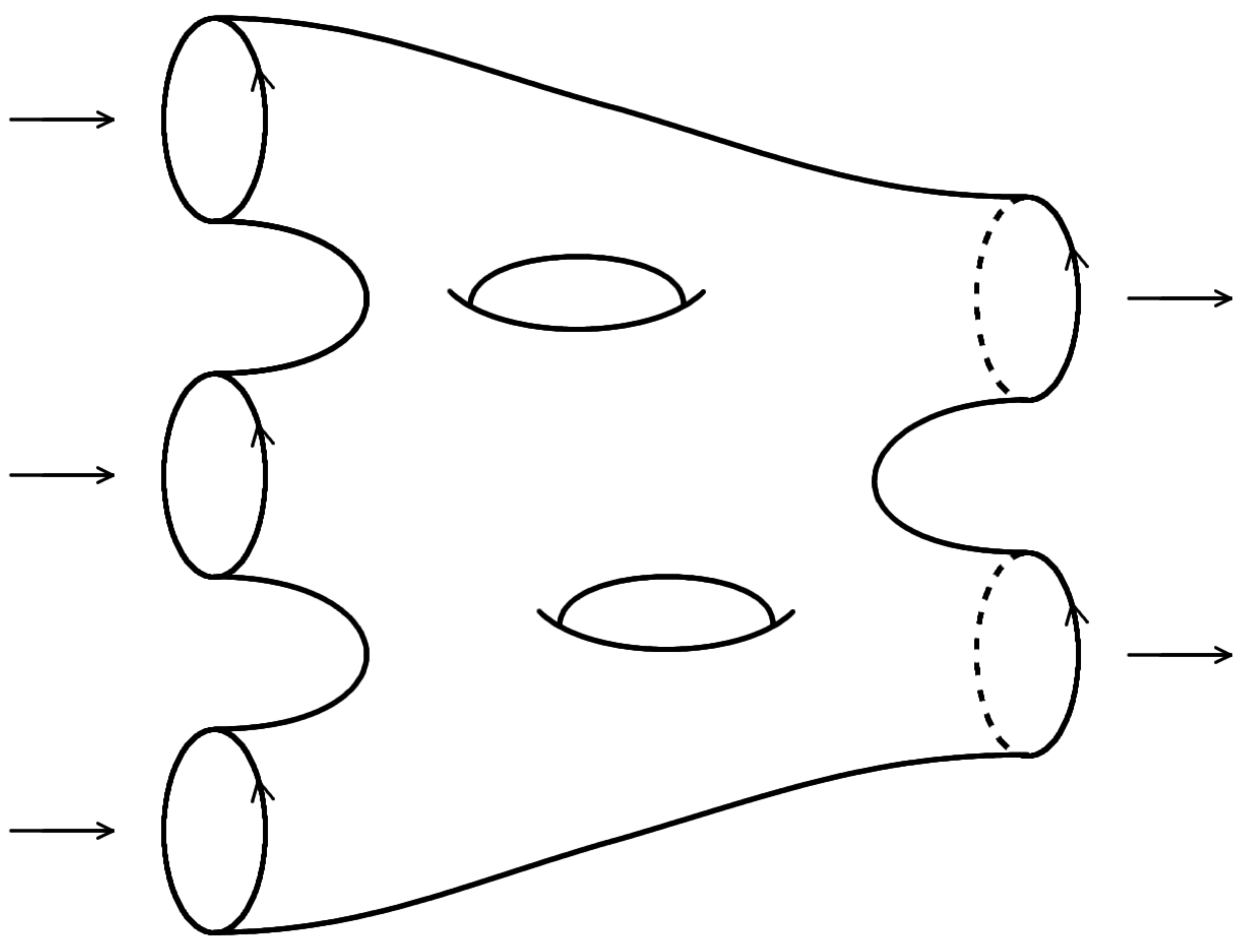} 
\end{center}

\noindent
A \emph{Conformal Field Theory}, as defined by Graeme Segal in \cite{Segal04}, is nothing but an algebra over the properad $\mathcal{R}_{g, n, m}$ of Riemann surfaces. 
\end{exam}

\noindent
In the rest of the text, we will introduce algebraic properads to encode various categories of ``bialgebras''. The name ``properad'' is a porte-manteau word from ``prop'' and ``operad''. 

\subsubsection*{$\diamond$ Prop} A prop is like a properad, but where one can also compose along non-necessarily connected graphs. This is the operadic notion which was introduced first, by Saunders MacLane in \cite{MacLane65} as a symmetric monoidal category $\mathsf{C}$, whose objects are the natural numbers  and whose monoidal product is their sum.  In this case, the elements of $\Hom_{\mathsf{C}}(n,m)$ are seen as the operations with $n$ inputs and $m$ outputs. An algebra over a prop is what F.-W. Lawvere  called an \emph{algebraic theory} in \cite{Lawvere63}.

\begin{exam} The categories of cobordism, where the objects are the $d$-dimensional manifolds and where the morphisms are the $(d+1)$-dimensional manifolds with $d$-dimension boundary, form a prop. 
Sir Michael F. Atiyah proposed to define a \emph{Topological Quantum Field Theory}, or \emph{TQFT} for short, as an algebra over a category of cobordism  \cite{Atiyah88}. 
\end{exam}

\subsubsection*{$\diamond$ Wheeled properad, wheeled prop} Sergei A. Merkulov introduced  in \cite{Merkulov09} the notion of ``wheels'' to encode algebras equipped with traces.  A trace is a linear map $\mathrm{tr} : \Hom(V,V) \to \KK$ satisfying 
$\mathrm{tr} (fg)=\mathrm{tr} (gf)$. We translate this on the operadic level by adding contraction maps:
\begin{eqnarray*}
\xi^i_j : \Po(n,m)&\to& \Po(n-1, m-1) \\
\vcenter{\hbox{\includegraphics[scale=0.2]{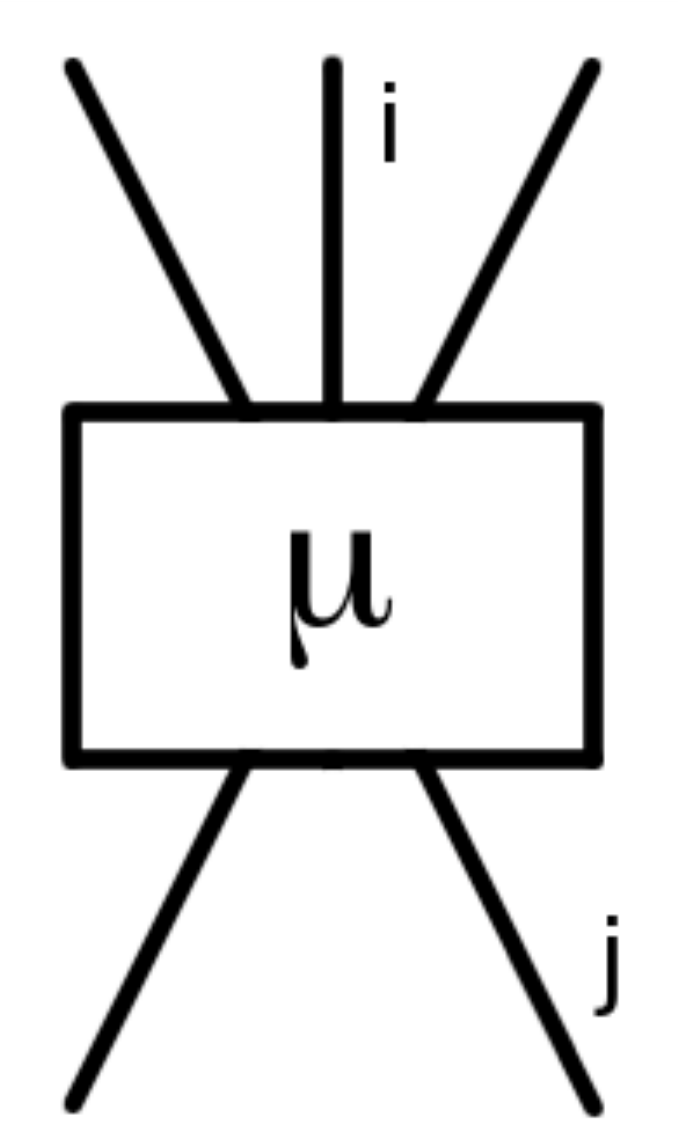}}}
&\mapsto& \vcenter{\hbox{\includegraphics[scale=0.2]{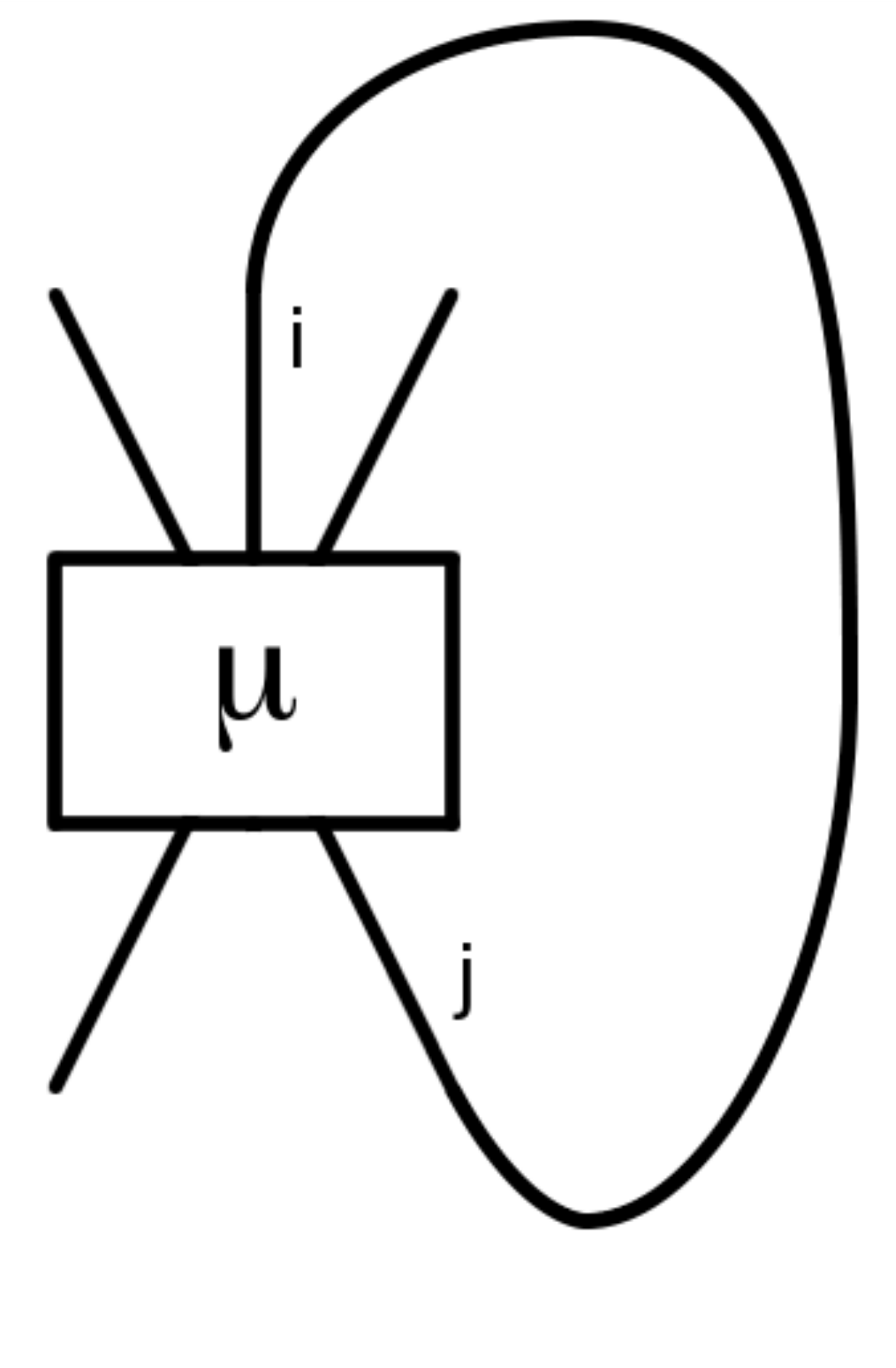}}}
\end{eqnarray*}
satisfying
$$ \vcenter{\hbox{\includegraphics[scale=0.15]{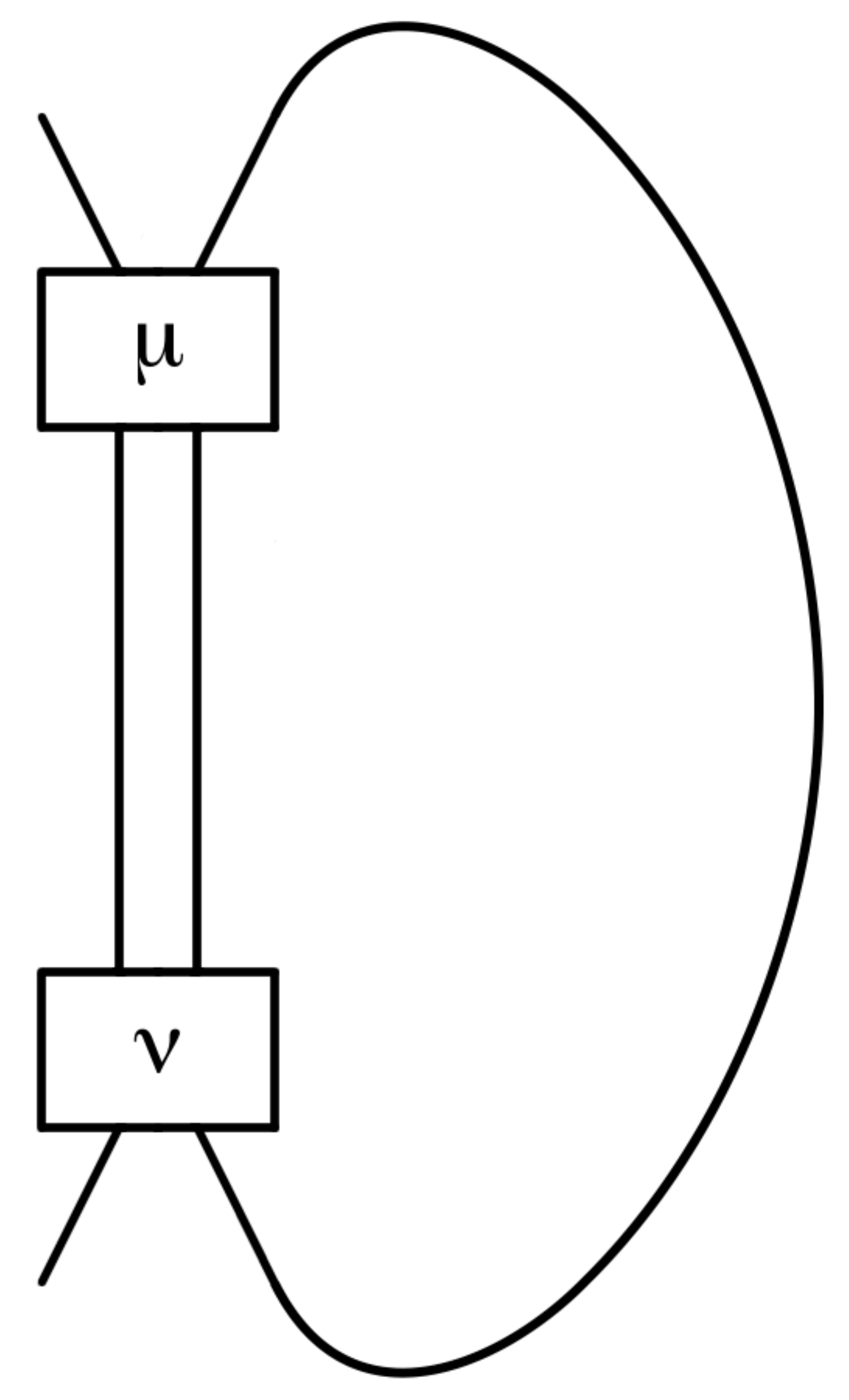}}}
\ \ =\  \vcenter{\hbox{\includegraphics[scale=0.15]{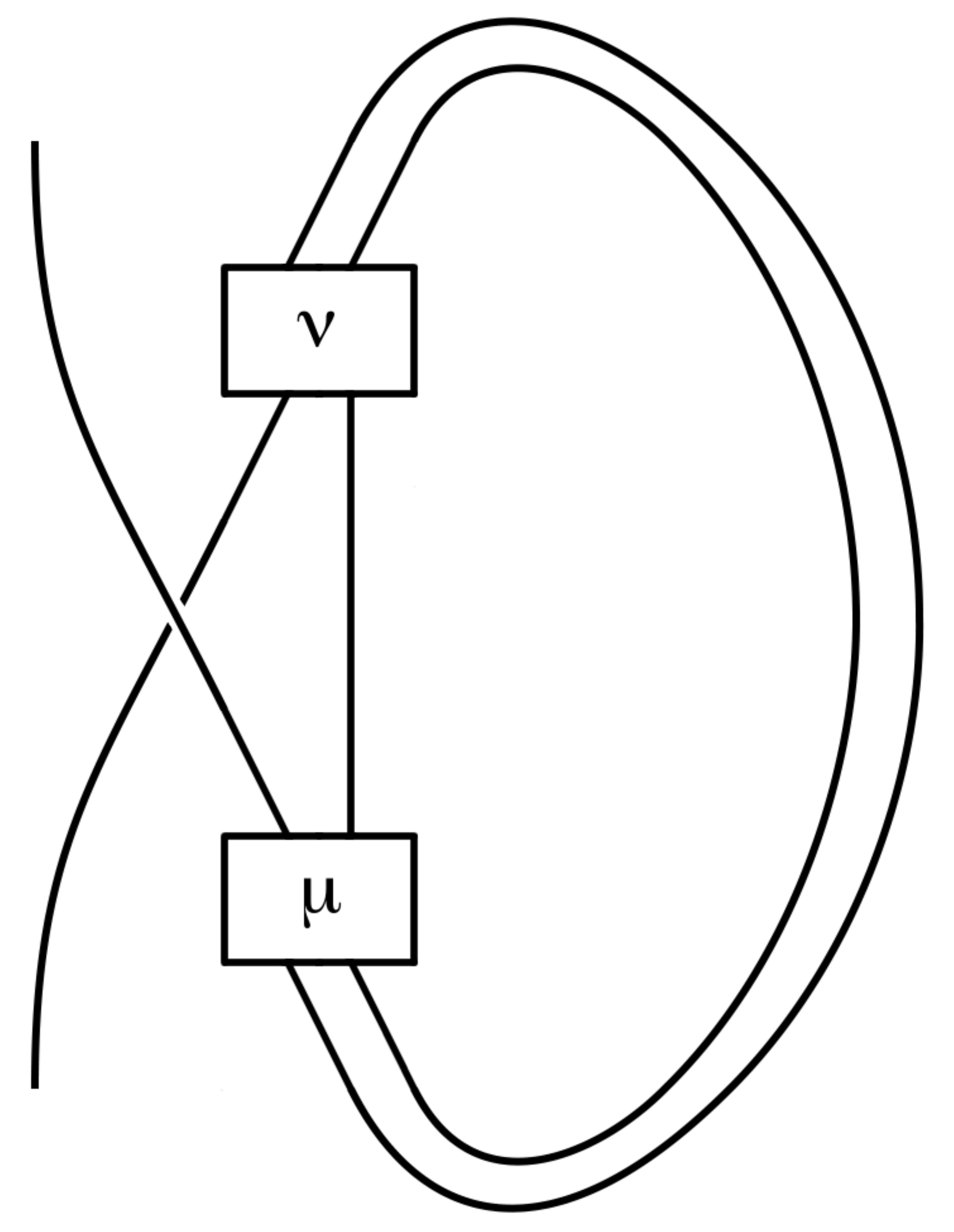}}}. $$
There are obvious functors between these notions. For instance, fixing an output, the genera $0$ and $1$ composition maps of a modular operad induce a wheeled operad structure. In general, a modular operad induces a properad, which contains all the higher genera compositions. The advantage is that properads can act on infinite dimensional vector spaces and that a Koszul duality theory is defined on properads level, see \ref{subsec:KoszulDuality}.Ê\\

Now that we have opened the Pandora's box, we hope that the reader feels comfortable enough  to go on and define new kinds of operads, which suit its problem best.

\section{Operadic syzygies}\label{sec:OpSyzygies}
Using the notion of an operad, we can answer the two questions raised at the end of Section~\ref{Sec:Alg+Homo}.
The example studied in Section~\ref{Sec:Alg+Homo} shows that some categories of  algebras over an operad are not stable under homotopy data: the category of associative algebras is not stable, but the category of $A_\infty$-algebras is stable, as Theorem~\ref{HTT2} shows. The conceptual explanation is that the ns operad $A_\infty$, encoding $A_\infty$-algebras, is \emph{free}. 

So, the method to provide a good framework for homotopy theory with a given category of algebras is as follows. 

First, one encodes the category of algebras by an operad $\Po$. Then, one tries to find a free operad $\Po_\infty$, which resolves the operad $\Po$, that is which is quasi-isomorphic to it. Finally, the category of $\Po_\infty$-algebras  carries the required homotopy properties and the initial category of $\Po$-algebras sits inside it. 
$$\xymatrix{\text{operad}\ \Po\  \ar@{..>}[d] & \ar[l]_(0.66){\sim} \Po_\infty=(\TTT(X),d) \ : \ \text{free resolution} \ar@{..>}[d] \\
\textsf{category of algebras}\  \ar@{^{(}->}[r] & \textsf{category} \  \textsf{of}\ {\textsf{homotopy algebras}  }}$$

\subsection{Quasi-free operad}
As well as for unital associative algebras,  operads admit free operads as follows. 
Forgetting the operadic compositions, one defines a functor $\mathcal U$ from the category of operads to the category of families of right modules over the symmetric groups $\Sy_n$, called \emph{$\Sy$-modules}.

\begin{defi}[Free operad]
The left adjoint  $\TTT$ to the forgetful functor
$$\mathcal U \ : \ \textsf{Operads} \ \rightleftharpoons\
\Sy\textsf{-modules}   \ : \ \TTT$$
 produces the \emph{free operad},  since it freely creates operadic compositions from any $\Sy$-modules. 
\end{defi}

\begin{ex}
Show that the free operad $\TTT(X)$ on an $\Sy$-module $X$ is equivalently characterized by the following universal property: there exists a natural morphism of $\Sy$-modules $X \to \TTT(X)$, such that any morphism of $\Sy$-modules $f : X \to \Po$, where $\Po$ is an operad, extends to  a unique morphism of operads $\tilde f : \TTT(X) \to \Po$: 
$$\xymatrix{ X \ar[r] \ar[dr]_(0.4)f  & \TTT(X) \ar@{..>}[d]^{\tilde f} \\ & \Po \ . }$$
\end{ex}

One defines the same notion on the level of nonsymmetric operads by replacing the category of $\Sy$-modules by the category of families of $\KK$-modules, called \emph{$\NN$-modules}. \\

Let us now make this functor explicit. To begin with the nonsymmetric case, let $X$ be an $\NN$-module and let $t$ be a planar tree. We consider the space $t(X)$ defined by trees of shape $t$ with vertices labelled by elements of $X$ according to their arity.

\begin{center}
\includegraphics[scale=0.2]{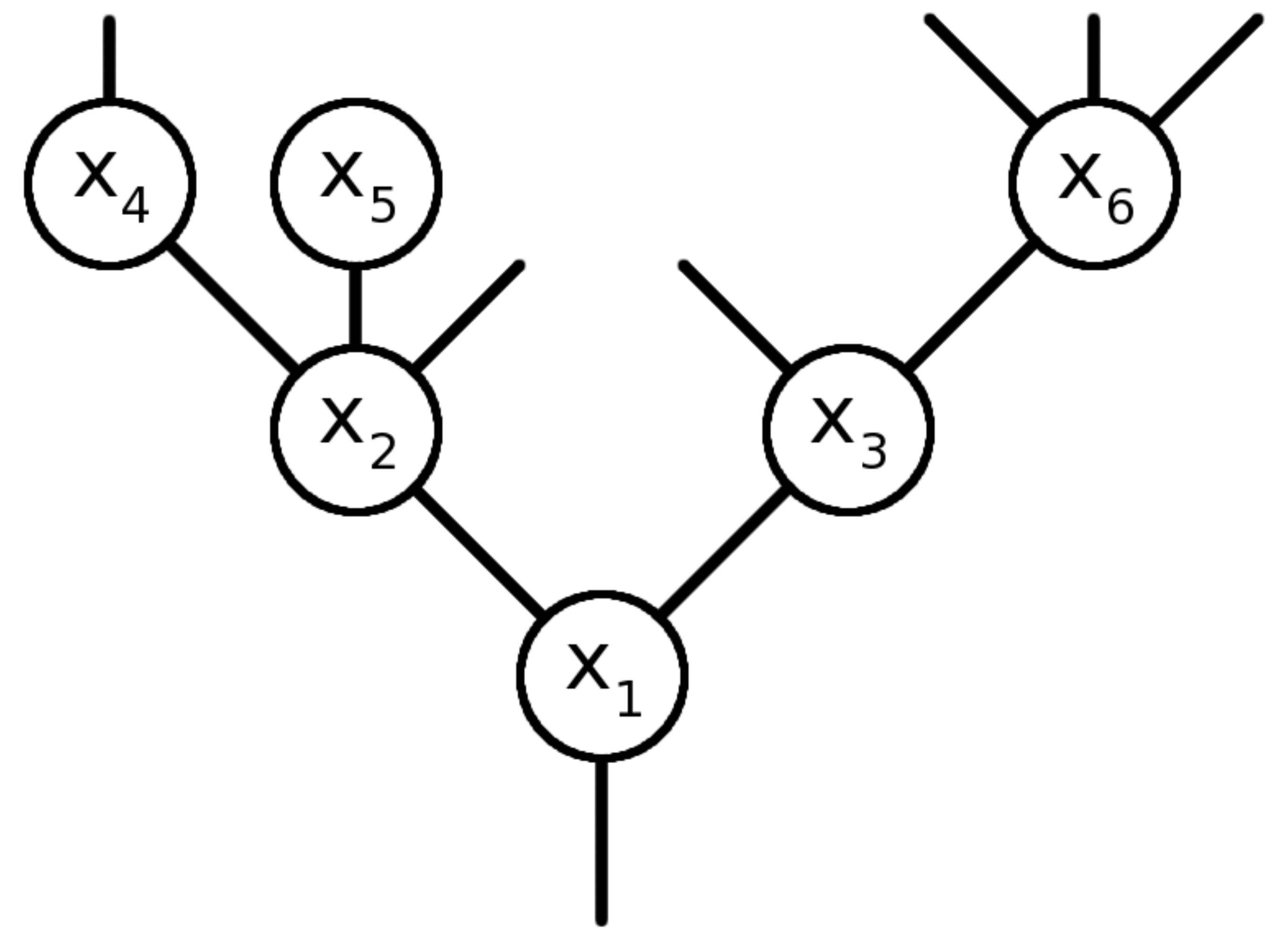} 
\end{center}

Its structure of vector space is given by identifying $t(X)$ with $ X^{\otimes k}$, where $k$ is the number of vertices of the tree $t$. We define the $\NN$-module $\TTT(X)$ by the sum of all the $X$-labelled planar trees:
$$\TTT(X)_n:=\bigoplus_{t \in PT_n} t(X) \ .$$
The operadic compositions are given by the grafting of trees. In the symmetric case, one  considers the set of rooted trees ``in space'', instead of planar rooted trees, with leaves labelled bijectively by $1, \ldots, n$. Then one labels the vertices with the elements of the $\Sy$-module $X$ in a coherent way with respect to the action of the symmetric group on $X$ and on the set of inputs of the vertices, see \cite[Section~$5.5$]{LodayVallette10} for more details. 
In a similar way, one can perform free operad constructions over the various underlying 
monoidal categories 
mentioned in Section~\ref{subsec:Underlying}.

\begin{prop}[Free operad]
The tree-module $\TTT(X)$, with the grafting of trees as operadic compositions, is the free operad on $X$.
\end{prop}

\begin{proof}
As usual, we leave this proof to the reader as a good exercise. 
\end{proof}

According to Section~\ref{subsec:Underlying}, a differential graded operad, dg operad for short, is a differential graded $\Sy$-module $(\Po, d)$ endowed with operadic compositions, for which the differential $d$ is a derivation: 
$$d(\gamma(\mu; \nu_1, \ldots, \nu_k))=\gamma(d(\mu); \nu_1, \ldots, \nu_k)+ \sum_{i=1}^k \pm\,  \gamma(\mu; \nu_1, \ldots, d(\nu_i), \ldots, \nu_k) \ . $$

\begin{ex}
Show that the differential $d$ on a free operad $\TTT(X)$ is completely characterized by the image of the generators $X \to \TTT(X)$. Make explicit the full differential $d$ from its restriction to the space of generators. 
\end{ex}

\begin{defi}[Quasi-free operad]
A dg operad whose underlying graded operad is free, that is after forgetting the differential,  is called a \emph{quasi-free operad}. 
\end{defi}

The operad itself is free but not the differential map. In a free dg operad, the differential is completely characterized by the differential of the space of generators $d_X : X \to X$ and carries no more terms. In a quasi-free operad, the differential associates to a generator of $X$ a sum of labelled trees in $\TTT(X)$, as the example of the next section shows. 

\subsection{Operadic syzygies}\label{subsec:Syzygies}
Let us make explicit a quasi-free resolution for the ns operad $As$. We shall begin with a topological resolution of the set-theoretic ns operad $\overline{Mon}$ defined by 
$$\overline{Mon}_0:=\emptyset, \quad \text{and} \quad \overline{Mon}_n:=Mon_n= 
\left\lbrace\vcenter{\hbox{\includegraphics[scale=0.15]{FIG14Corolla.pdf}}}\right\rbrace, 
\ \text{for} \ n\ge 1 \ ,$$
and which provides a basis for $As$.\\

$\diamond\ $ \emph{In arity $1$}, the ns operad $\overline{Mon}$ has one element of degree $0$, so as the free ns operad $\TTT(X)$ on a generating set $X$ empty in arity $1$: $X_1=\emptyset$. Hence, in arity $1$, one has two $0$-cells, $\overline{Mon}_1=\TTT(X)_1=\lbrace \I \rbrace$, which are obviously homotopy equivalent. 

\begin{center}
\includegraphics[scale=0.18]{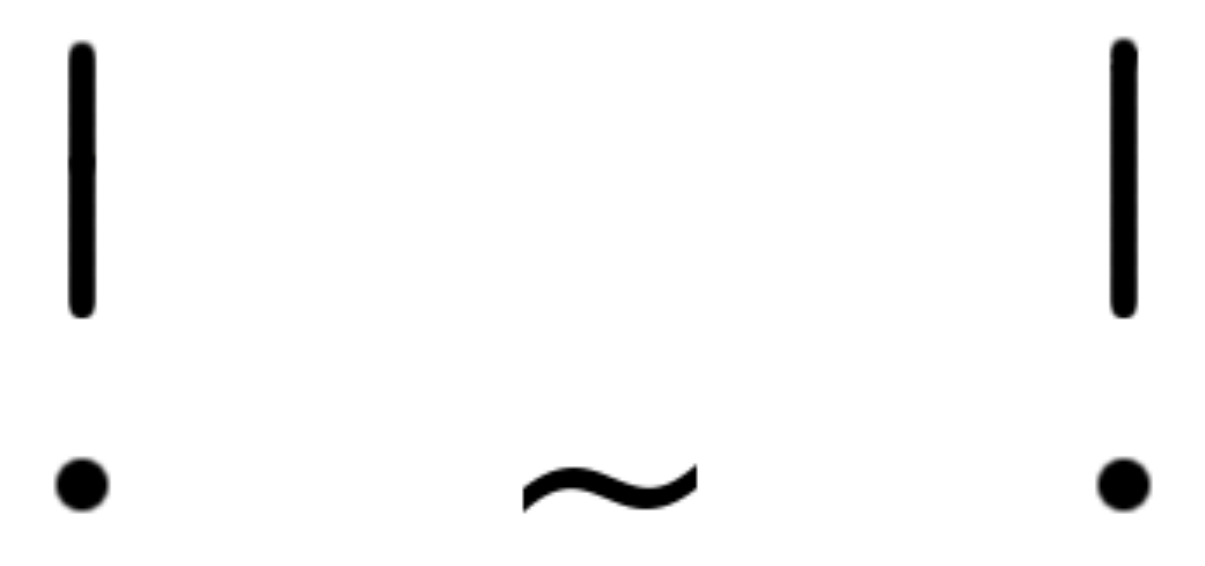} 
\end{center}

$\diamond\ $ \emph{In arity $2$}, the ns operad $\overline{Mon}$ has one element 
$\vcenter{\hbox{\includegraphics[scale=0.1]{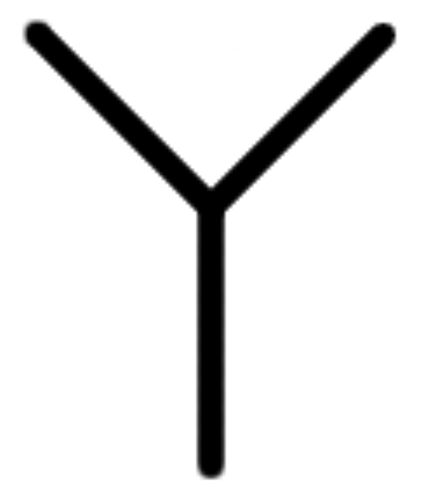}}}$ of degree $0$. So, we introduce a degree $0$ element $\vcenter{\hbox{\includegraphics[scale=0.1]{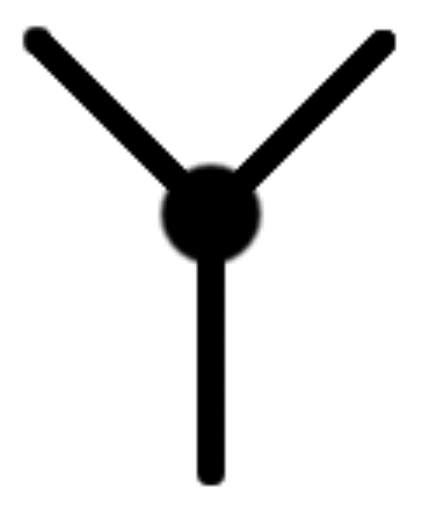}}}$ of arity $2$ in the generating set $X$ of $\TTT(X)$. In arity $2$, one has now two $0$-cells, which are homotopy equivalent. 

\begin{center}
\includegraphics[scale=0.18]{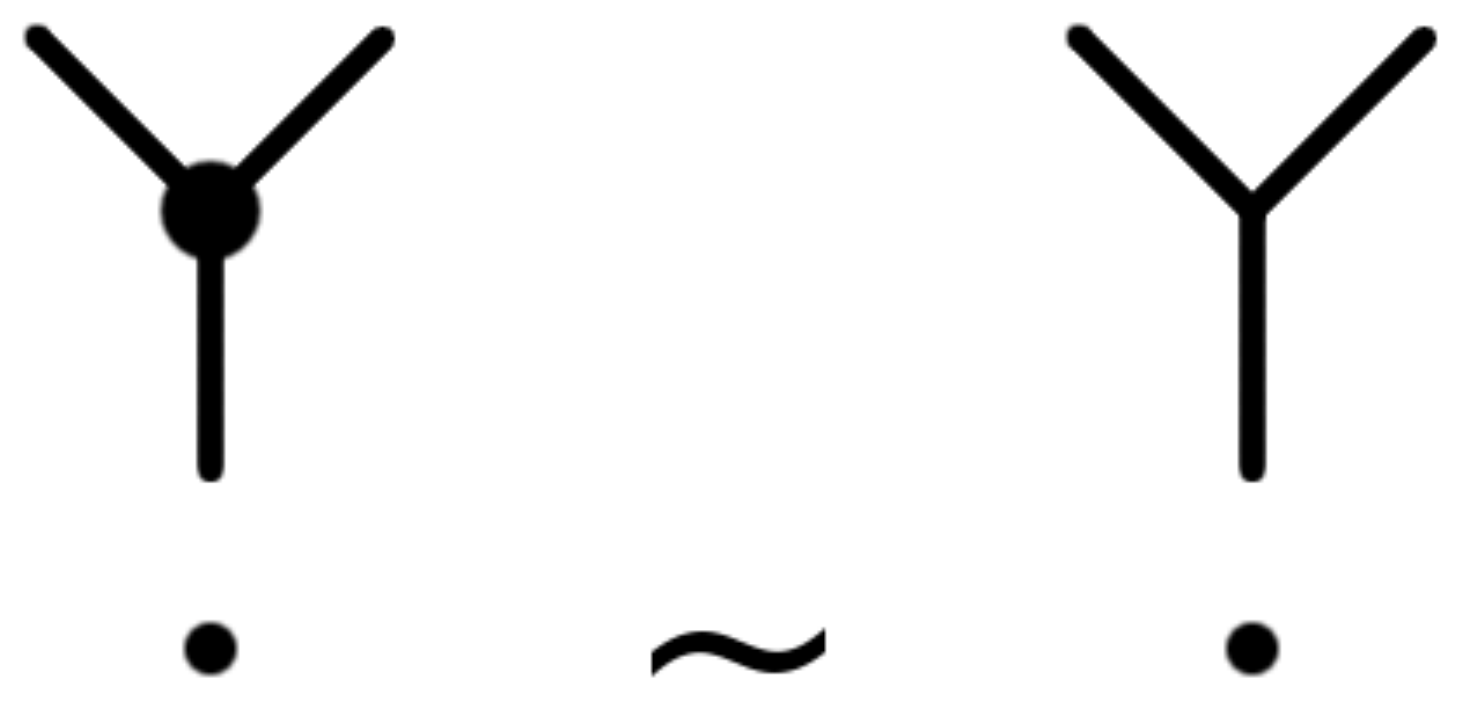} 
\end{center}

$\diamond\ $ \emph{In arity $3$}, on the one hand, the ns operad $\overline{Mon}$ has one element of degree $0$: 
$\vcenter{\hbox{\includegraphics[scale=0.1]{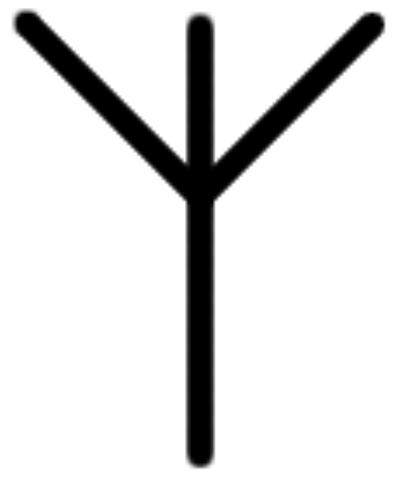}}}$. 
On the other hand, the free ns operad $\TTT(X)$ has $2$ different elements of degree $0$: the left comb 
$\vcenter{\hbox{\includegraphics[scale=0.1]{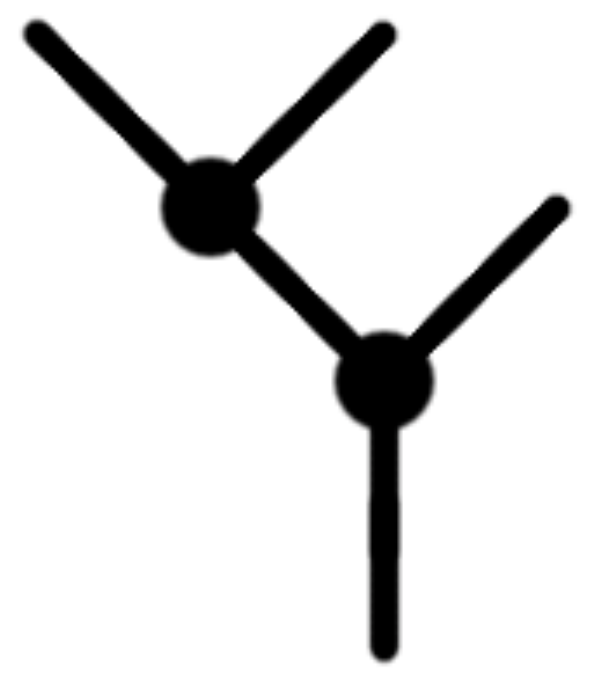}}}$
and the right comb 
$\vcenter{\hbox{\includegraphics[scale=0.1]{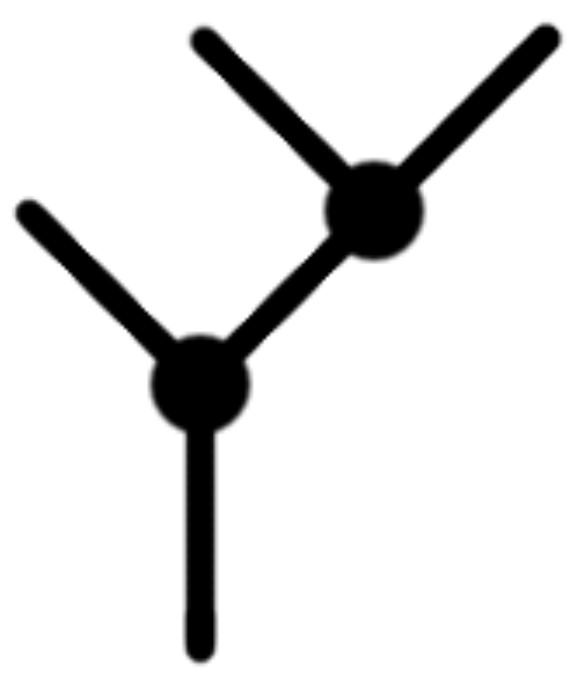}}}$. 
So, we introduce in $X_3$ a degree $1$ element 
$\vcenter{\hbox{\includegraphics[scale=0.1]{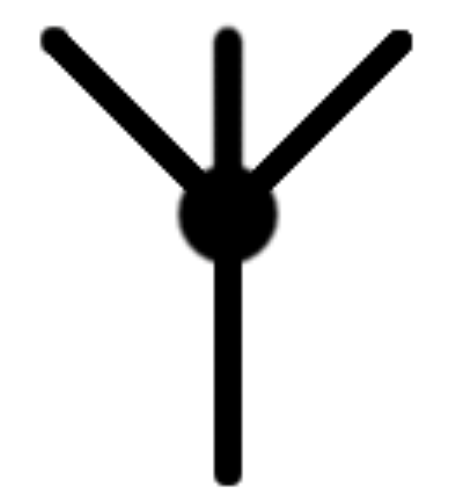}}}$
of arity $3$, whose boundary is equal to these two degree $0$ elements respectively. Finally, in arity $3$, one has now one $0$-cell, on one side, and two $0$-cells linked by a $1$-cell, on the other side. These two topological spaces are homotopy equivalent.

\begin{center}
\includegraphics[scale=0.18]{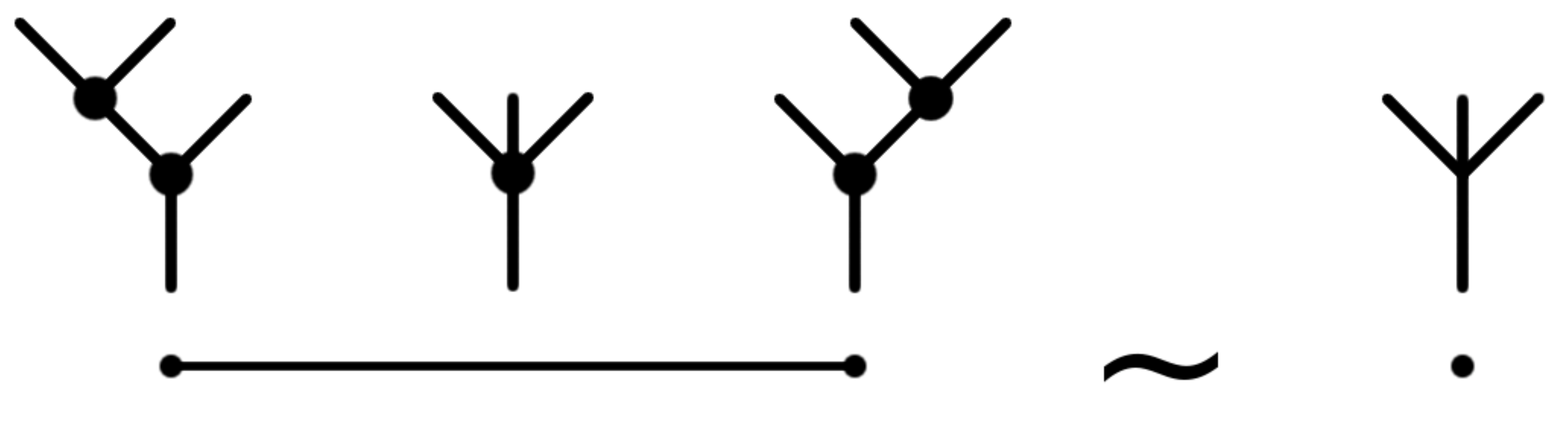} 
\end{center}

$\diamond\ $ \emph{In arity $4$}, on the one hand, the ns operad $\overline{Mon}$ has one element of degree $0$. On the other hand, the operad $\TTT(X)$ has $5$ elements of degree $0$, the $5$ planar binary trees with $4$ leaves, and $5$ elements of degree $1$, the $5$ planar trees with $2$ vertices and $4$ leaves. They are attached to one  another to form a pentagon. 

\begin{center}
\includegraphics[scale=0.12]{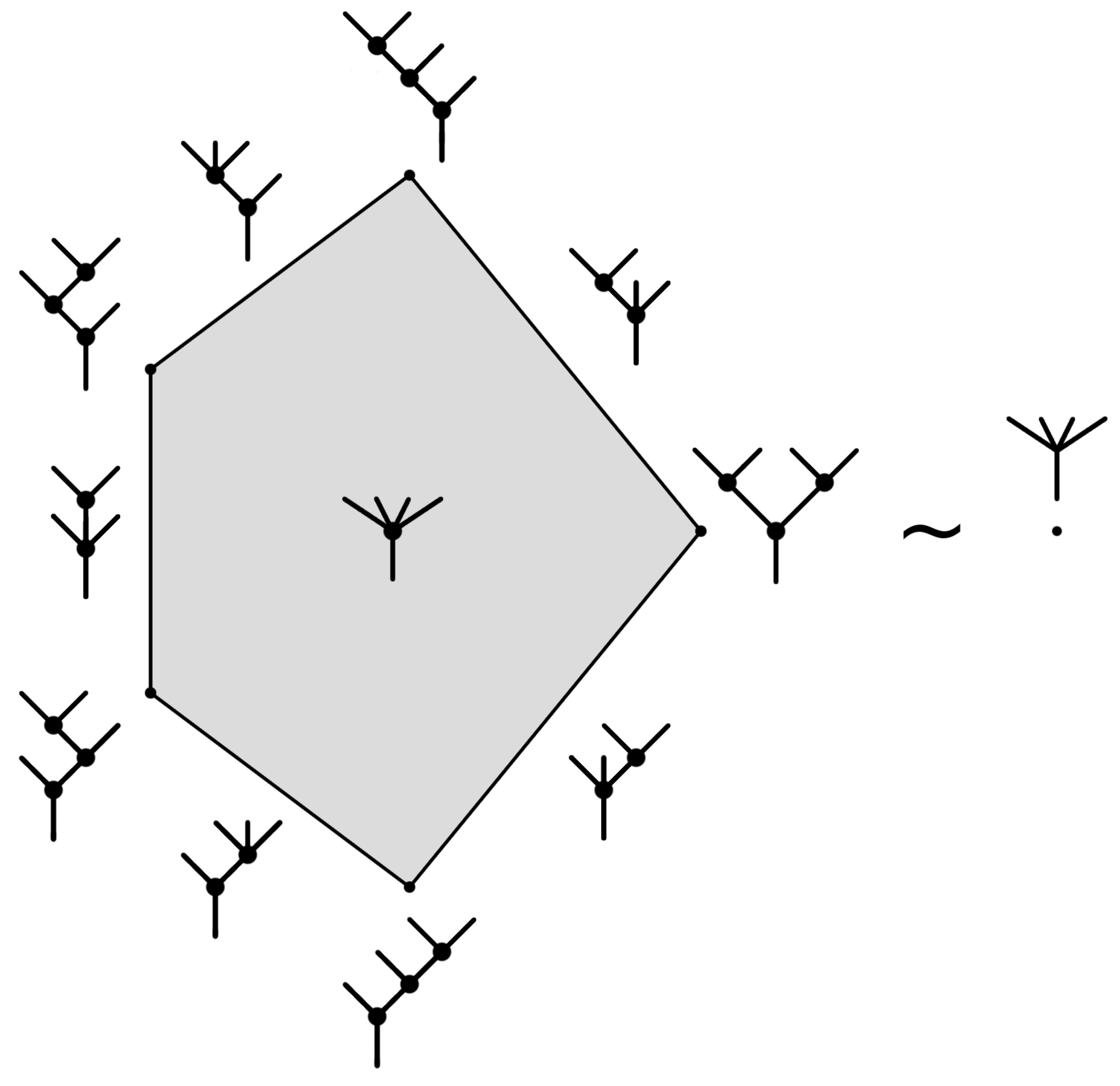} 
\end{center}

To get a topological space, homotopy equivalent to the point, we introduce in $X_4$ a degree $2$ element 
$\vcenter{\hbox{\includegraphics[scale=0.1]{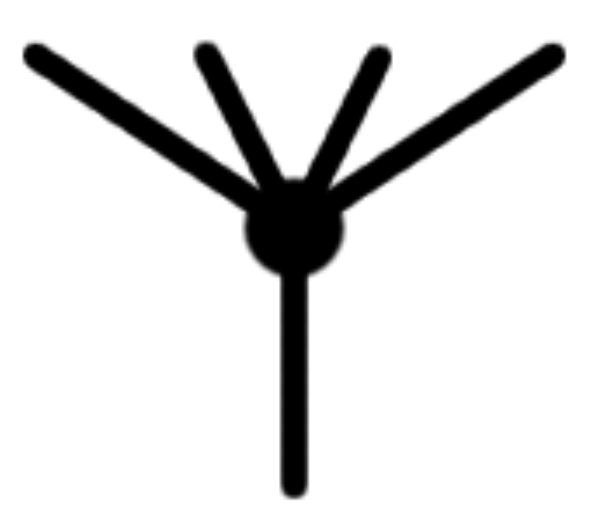}}}$, 
whose boundary is equal to the boundary of the pentagon. This kills the  homology group of degree $1$ in the free operad. 

\begin{ex}
Proceed in the same way in arity $5$ to discover the following polytope of dimension $3$ with $14$ $0$-cells, $20$ $1$-cells, 
$9$ $2$-cells, 
and $1$ new $3$-cell. Label it !
\begin{center}
\includegraphics[scale=0.13]{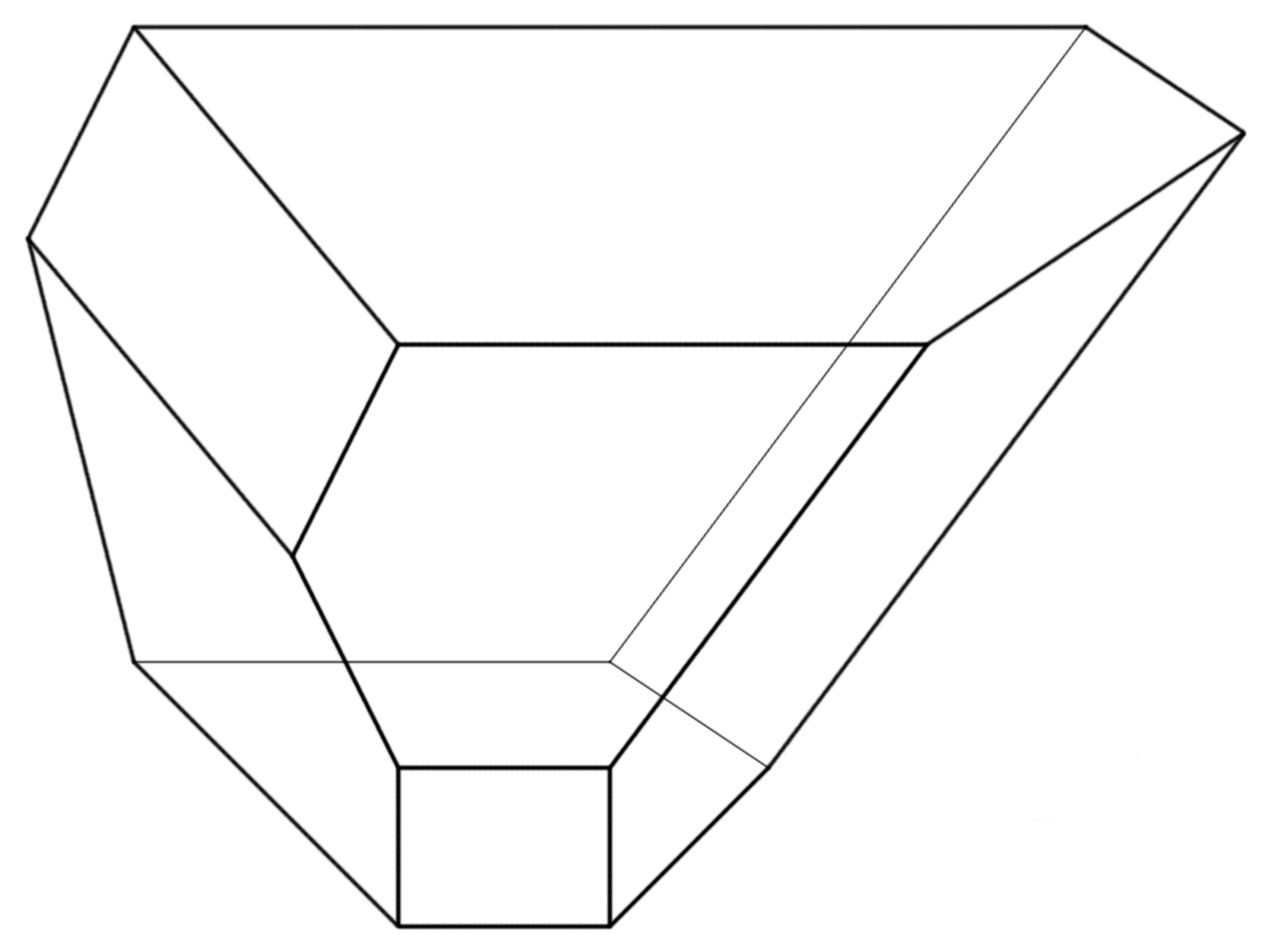} 
\end{center}
\end{ex} 

$\diamond\ $ \emph{In general}, one has to introduce one cell of dimension $n-2$ in arity $n$ to kill the $(n-2)^\text{th}$-homotopy group. In the end, the generating space $X$ is made up of one element of degree $n-2$ in arity $n$, for $n\ge 2$. The boundary of each generator is given, according to some orientation, by the set of planar trees with $2$ vertices. These polytopes are called the \emph{associahedra}, or the \emph{Stasheff polytopes} after Jim Stasheff \cite{Stasheff63}. (They also correspond to the Hasse diagrams of the Tamari lattices \cite{Tamari51} made up of  planar binary trees 
with the partial order generated by the covering relation 
$\vcenter{\hbox{\includegraphics[scale=0.1]{LeftComb.pdf}}}\prec
\vcenter{\hbox{\includegraphics[scale=0.1]{RightComb.pdf}}}$. We could also use the little intervals operad to get a topological resolution of $\overline{Mon}$.) \\

\begin{defi}[Operad $A_\infty$]
We define the dgns operad 
$$A_\infty:=(C_\bullet^{cell} (\TTT(X)),d) $$ by the image, under the cellular chain functor, of the aforementioned free topological resolution $\TTT(X)$ of the ns operad $\overline{Mon}$.
\end{defi}

\begin{theo}[Resolution $A_\infty$]
The dgns operad $A_\infty$ provides the following quasi-free resolution of $As$: 
$$\boxed{A_\infty=\left( \TTT \left(
\vcenter{\hbox{\includegraphics[scale=0.1]{Corolla2.pdf}}}
, \vcenter{\hbox{\includegraphics[scale=0.1]{Corolla3.pdf}}}, 
\vcenter{\hbox{\includegraphics[scale=0.1]{Corolla4.pdf}}}, \ldots \right), d\right) \qi As }$$
where the boundary condition for the top dimensional cells translates into the following differential: 
\begin{center}
\includegraphics[scale=0.18]{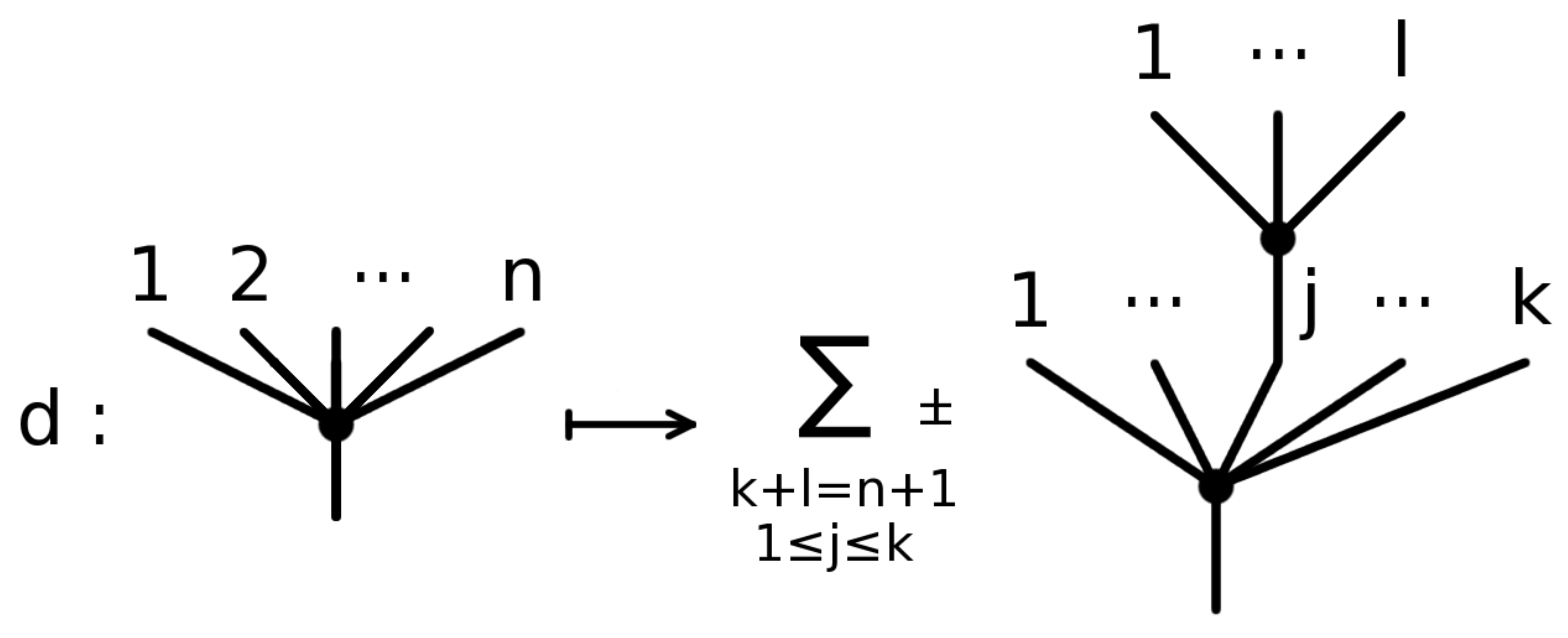} 
\end{center}
\end{theo}

\begin{proof}
Since the topological ns operad $\TTT(X)$ is homotopy equivalent to $\overline{Mon}$, the dgns operad $A_\infty$ is quasi-isomorphic to $As$.
\end{proof}

Last step of the method explained in the introduction:  an algebra over the resolution $A_\infty$ of $As$ 
is an $A_\infty$-algebra, as defined in \ref{subsec:AinftyAlg}. Indeed,  a morphism of dgns operads 
$$A_\infty=\left( \TTT \left(
\vcenter{\hbox{\includegraphics[scale=0.1]{Corolla2.pdf}}}
, \vcenter{\hbox{\includegraphics[scale=0.1]{Corolla3.pdf}}}, 
\vcenter{\hbox{\includegraphics[scale=0.1]{Corolla4.pdf}}}, \ldots \right), d\right) \to (\End_V, \partial)$$
in characterized by the images of the generators, which give here   a family of operations $\mu_n : V^{\otimes n} \to V$ of degree $n-2$,  for any $n\ge 2$. The commutativity of the differentials applied to the generators amounts precisely to the relations given in \ref{subsec:AinftyAlg} of an $A_\infty$-algebra. 
Finally, this section answers Question~$(1)$ raised at the end of Section~\ref{Sec:Alg+Homo}.\\

\begin{defi}[Operadic syzygies]
The generating elements $X$ of a quasi-free resolution $\TTT(X)$ are called the \emph{operadic syzygies}. 
\end{defi}

The problem of finding a quasi-free resolution for an operad is similar to the problem of finding a quasi-free resolution 
$$ \cdots \to A\otimes M_2 \to A\otimes M_1 \to A\otimes M_0 \epi M $$
for a module $M$ over a commutative ring $A$ in algebraic geometry, see \cite{Eisenbud05}. In this case, the 
generating elements 
$\lbrace M_n \rbrace_{n\ge 0}$ 
of the free $A$-module form the classical syzygies. \\

As in the classical case of modules over a commutative ring, we have introduced here  a first syzygy corresponding to the generator of $As$. Then, we have introduced a second syzygy for the relation of $As$. The third syzygy was introduced for the relations among the relations, etc. This iterative process is called the \emph{Koszul-Tate resolution} after Jean-Louis Koszul \cite{Koszul50} and John Tate \cite{Tate57}. (It was  successfully used to compute the BRST cohomology in Mathematical Physics \cite{HenneauxTeitelboim92} and to compute rational homotopy groups in Rational Homotopy Theory \cite{Sullivan77}.)

\begin{rema}
One can develop the homotopy theory for dg operads by endowing this category with a model category structure, see \cite{Quillen67, Hinich97, BergerMoerdijk03}. 
The conceptual reason why the category of algebras over a quasi-free operad behaves well with respect to 
homotopy data is that quasi-free operads are cofibrant, the general notion for being projective. (This is actually true over non-negatively graded chain complexes. In general, one has to require the existence of a good filtration on the syzygies standing for the generators, the relations, the relations among the  relations, etc.)
\end{rema}

The Koszul-Tate method works step-by-step. We would like now to be able to produce all the syzygies at once. 

\subsection{Ideal, quotient operad and quadratic operad}\label{subsec:Quad}
As well as for unital associative algebras,  operads admit ideals and quotient operads. 

\begin{defi}[Operadic ideal]
A sub-$\Sy$-module $\mathcal I \subset \Po$ of an operad is called an \emph{operadic ideal} if the operadic compositions satisfy $\gamma(\mu; \nu_1, \ldots, \nu_k)\in \mathcal I$, when at least one of the 
$\mu, \nu_1, \ldots, \nu_k$ is in $\mathcal I$. 
\end{defi}

\begin{ex}
Show that, in this case, the quotient $\Sy$-module ${\Po}/{\mathcal I}$ carries a unique operad structure which makes the projection of $\Sy$-modules
$\Po \epi \Po/\mathcal I$ into an operad morphism.
\end{ex}

For any sub-$\Sy$-module $R$ of an operad $\Po$, there exists a smallest ideal  containing $R$. It is called the \emph{ideal generated by $R$} and denoted $(R)$. 

The free operad $\TTT(E)$ admits a weight grading $\TTT(E)^{(k)}$  given by the number of vertices of the underlying tree. 

\begin{defi}[Quadratic data and quadratic operad]
A \emph{quadratic data} $(E; R)$ consists of an $\Sy$-module $E$ and a sub-$\Sy$-module $R\subset \TTT(E)^{(2)}$. It induces the \emph{quadratic operad} 
$$\Po(E; R):=\TTT(E)/(R) \ . $$
\end{defi}

\begin{exams}$ \ $

\begin{itemize}
\item[$\diamond$] The algebra of dual numbers $D=T(\KK \delta)/(\delta^2)$ of Section~\ref{subsec:uAsRep} is a quadratic algebra, and thus a quadratic operad concentrated in arity $1$. 

\smallskip

\item[$\diamond$] The ns operad $As$ admits the following quadratic presentation 
$$As\cong\Po\left(\vcenter{\hbox{\includegraphics[scale=0.1]{Corolla2Simple.pdf}}}; 
\vcenter{\hbox{\includegraphics[scale=0.1]{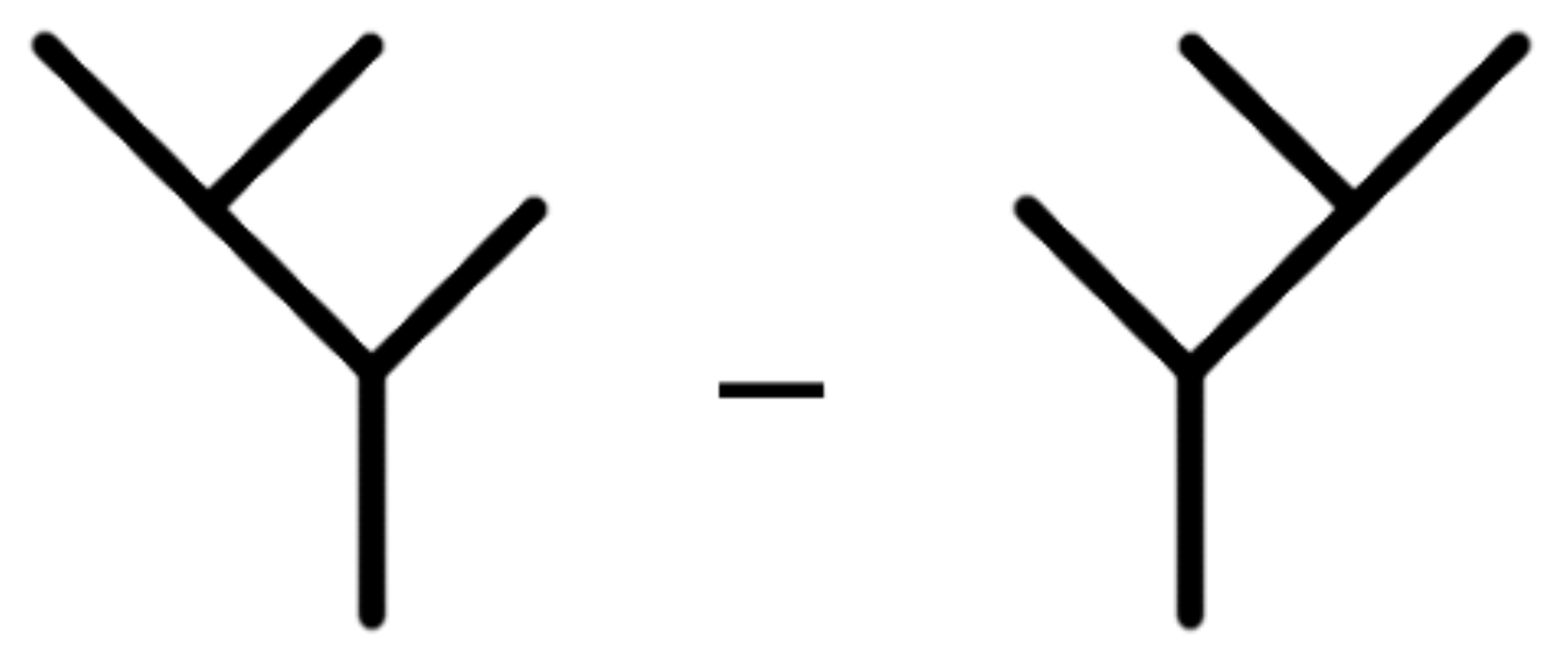}}}   \right)\ ,   $$
with one binary generator. Indeed, the free ns operad  $\TTT(\vcenter{\hbox{\includegraphics[scale=0.1]{Corolla2Simple.pdf}}})$ resumes to planar binary trees and, under the relation among  the subtrees $ 
\vcenter{\hbox{\includegraphics[scale=0.1]{FIG1Assoc.pdf}}}$, the planar binary trees with the same number of leaves identify to one another. 

\smallskip

 \item[$\diamond$] In the same way, the operad $Com$ admits the following quadratic presentation 
$$Com\cong\Po\left( \vcenter{\hbox{\includegraphics[scale=0.1]{FIG16Y12.pdf}}}; 
\vcenter{\hbox{\includegraphics[scale=0.1]{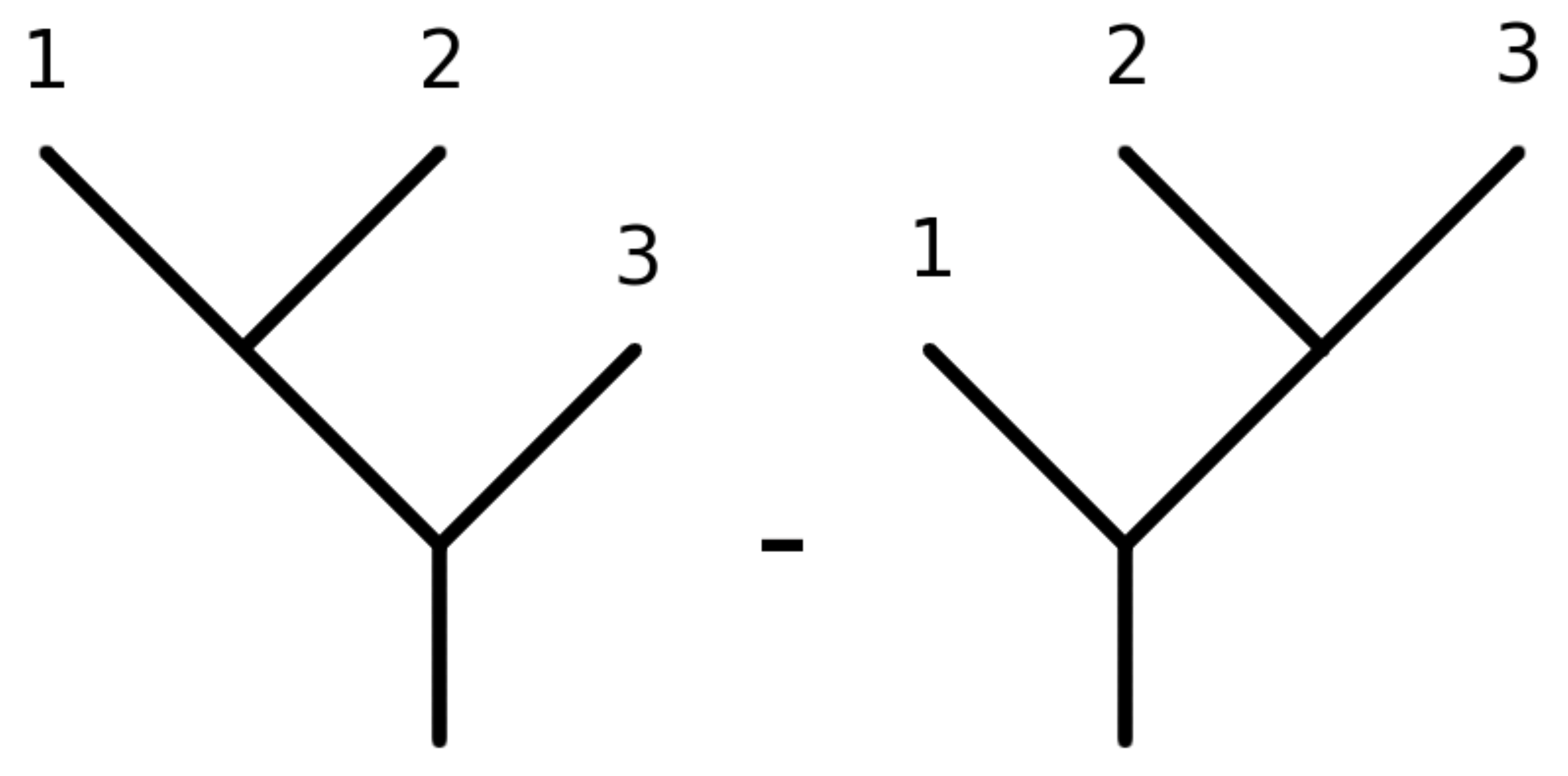}}},  
\vcenter{\hbox{\includegraphics[scale=0.1]{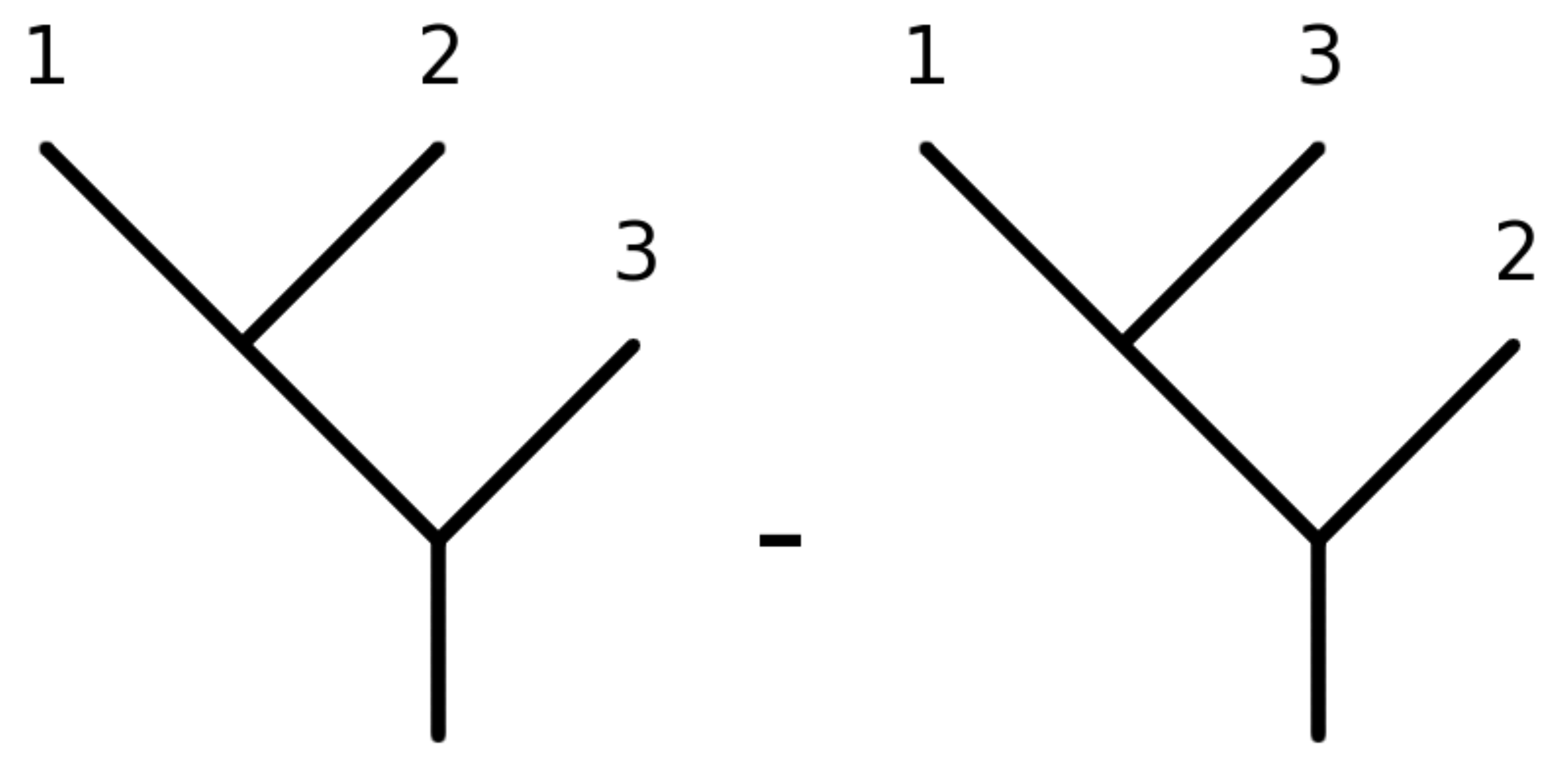}}}   \right) \ ,   $$
with one binary generator invariant under the action of $\Sy_2$. 

\smallskip

\item[$\diamond$] The operad $Lie$ is defined by the quadratic data 
$$Lie:=\Po\left( \vcenter{\hbox{\includegraphics[scale=0.1]{FIG16Y12.pdf}}};  
\vcenter{\hbox{\includegraphics[scale=0.1]{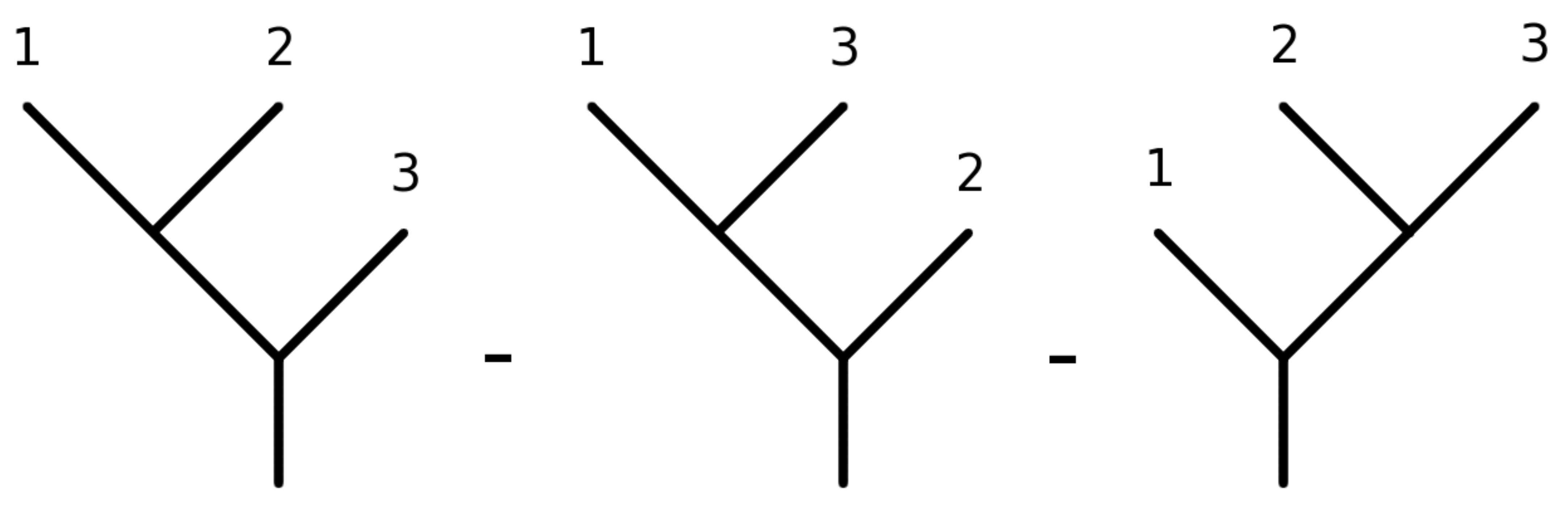}}} \right) \ .   $$

\smallskip

\item[$\diamond$] The presentation of the operad $Ger$ encoding Gerstenhaber algebras, given in 
Section~\ref{subsec:Underlying}, is quadratic.

\smallskip

\item[$\diamond$] In the same way, there exists a notion of free properad, explicitly given by labelled directed connected graphs. Properadic ideals and quotient are defined similarly, see \cite{Vallette07} for more details. 
The category of \emph{Lie bialgebras}, introduced by Vladimir Drinfel'd in \cite{Drinfeld87}, is encoded by the following quadratic properad $BiLie:=\Po(E; R)$, where 
$$
E:=\left\{    \vcenter{\hbox{\includegraphics[scale=0.1]{FIG16Y12.pdf}}}, \, 
 \vcenter{\hbox{\includegraphics[scale=0.1]{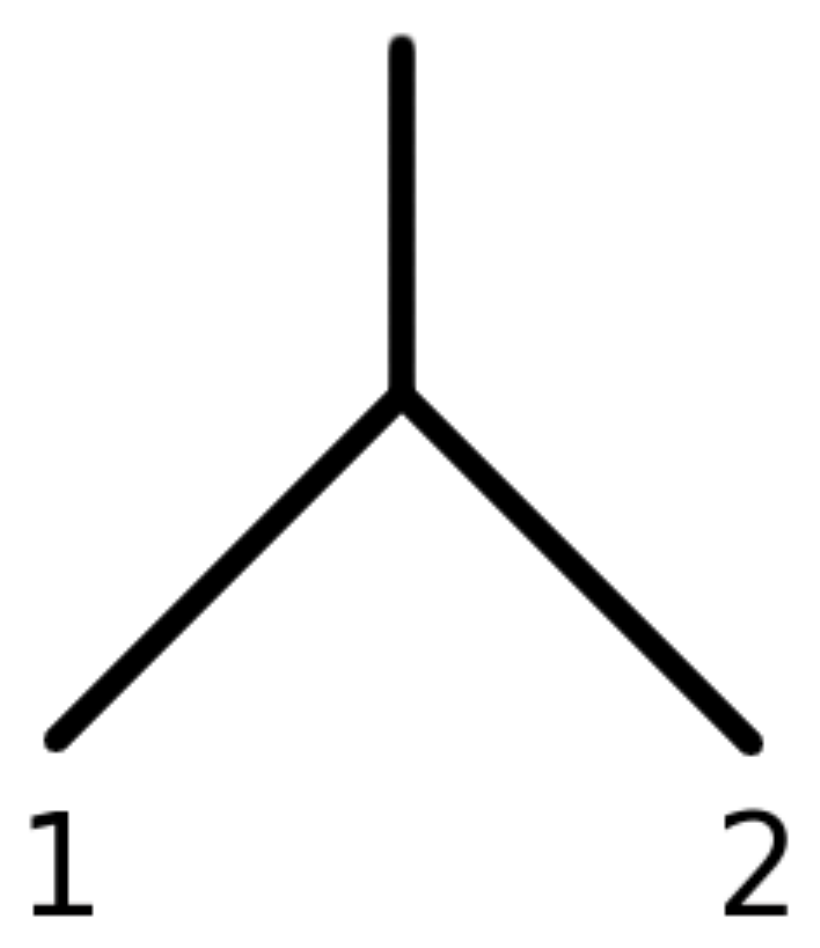}}}
\right\} $$
with skew-symmetric action of $\Sy_2$ in both cases, and where
\begin{eqnarray*}
R&:=&\left\{
\vcenter{\hbox{\includegraphics[scale=0.1]{FIG31Jacobi.pdf}}}, \, 
\vcenter{\hbox{\includegraphics[scale=0.1]{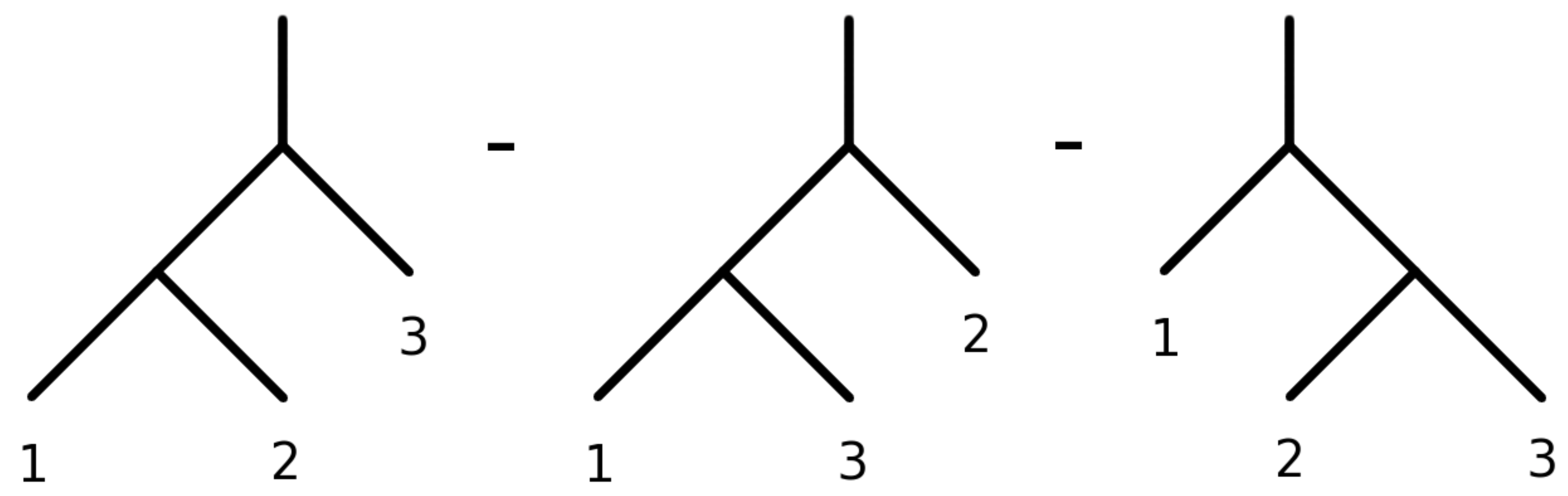}}}, \right. \\
&&\quad\ \left.\vcenter{\hbox{\includegraphics[scale=0.1]{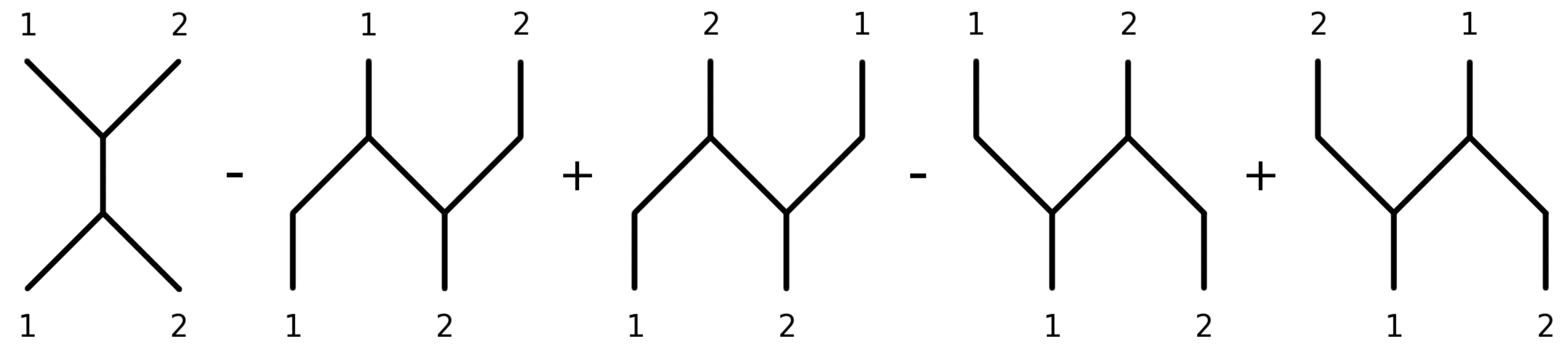}}} \right\}\ .
\end{eqnarray*}
An \emph{involutive Lie bialgebra} is a Lie bialgebra satisfying the following extra ``diamond'' condition: 
$$\vcenter{\hbox{\includegraphics[scale=0.1]{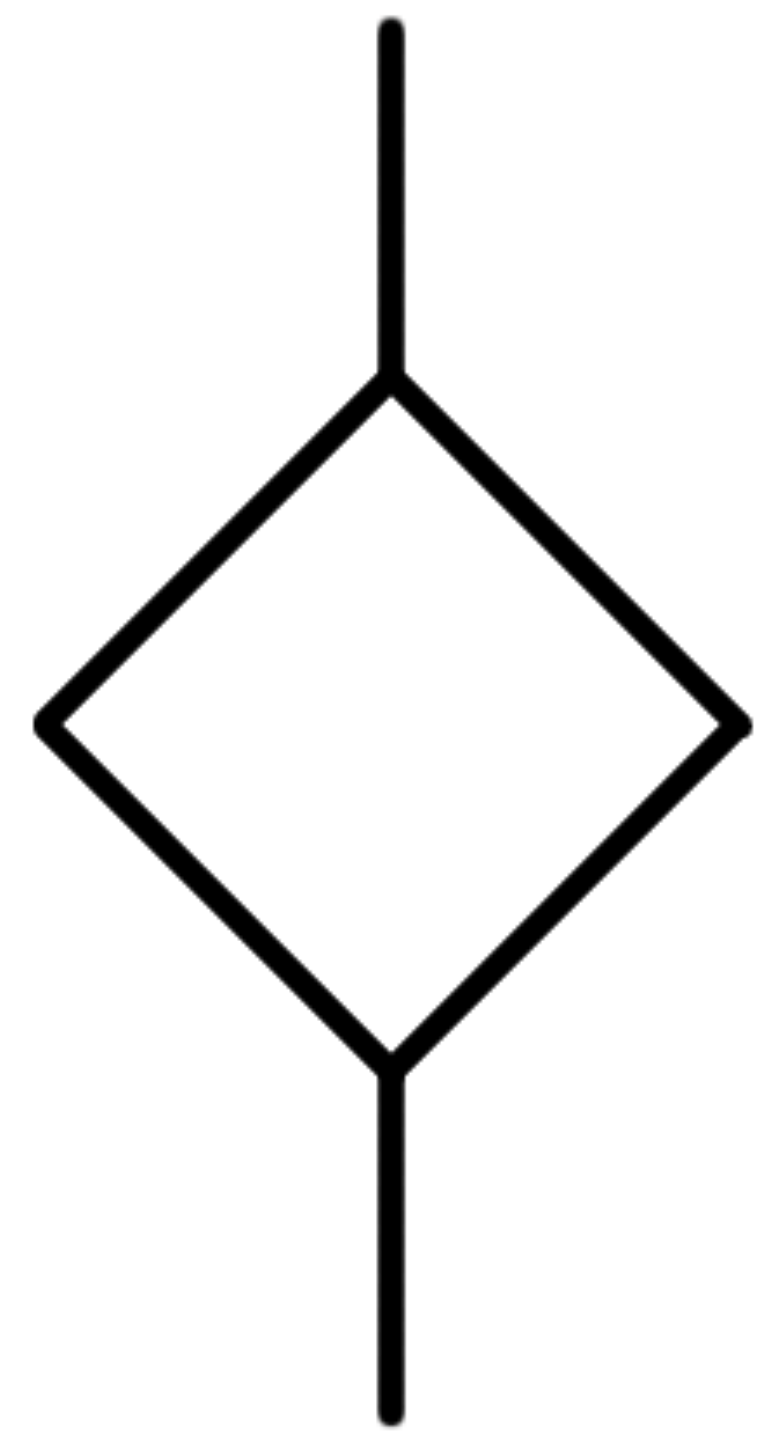}}}\ .$$
Such a structure can be found in String Topology, see \cite{ChasSullivan99}. The properad $BiLie_\diamond$ encoding involutive Lie bialgebras is quadratic. 

\smallskip

\item[$\diamond$] In any symmetric Frobenius algebra, one can dualize the product via the non-degenerate bilinear form to define a cocommutative coproduct. This gives rise to the notion of \emph{Frobenius bialgebra}, which is encoded by the following quadratic properad:
$$Frob:=    \Po\left(\vcenter{\hbox{\includegraphics[scale=0.1]{Corolla2Simple.pdf}}}, \,
\vcenter{\hbox{\includegraphics[scale=0.1]{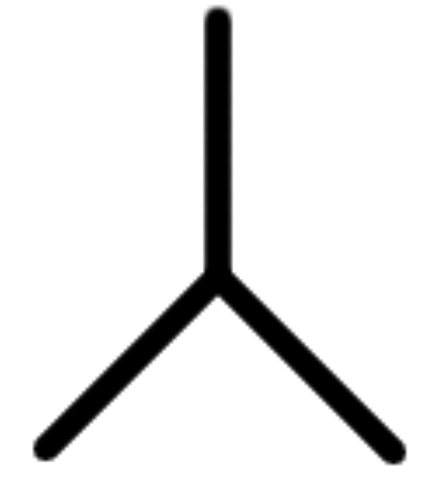}}}; \,
\vcenter{\hbox{\includegraphics[scale=0.1]{FIG1TERAssociator.pdf}}}, \,
\vcenter{\hbox{\includegraphics[scale=0.1]{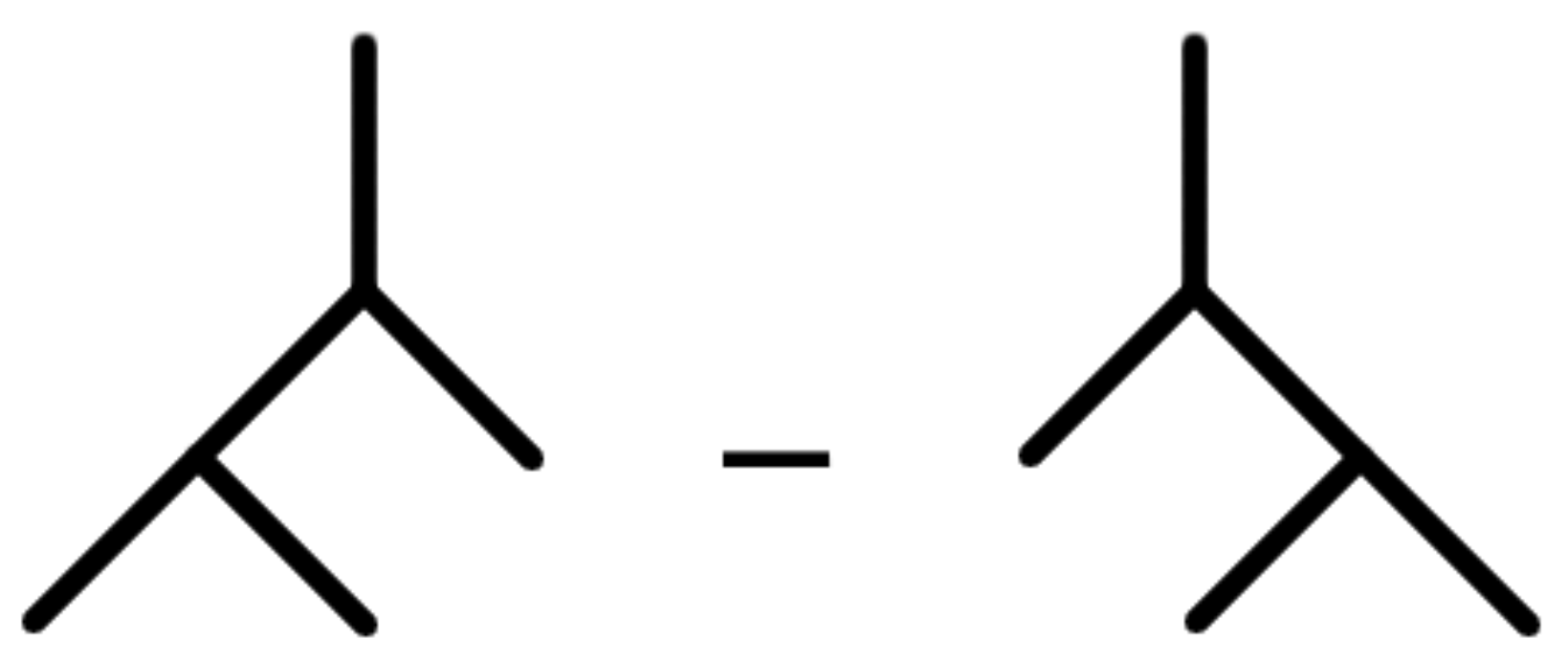}}}, \,
\vcenter{\hbox{\includegraphics[scale=0.1]{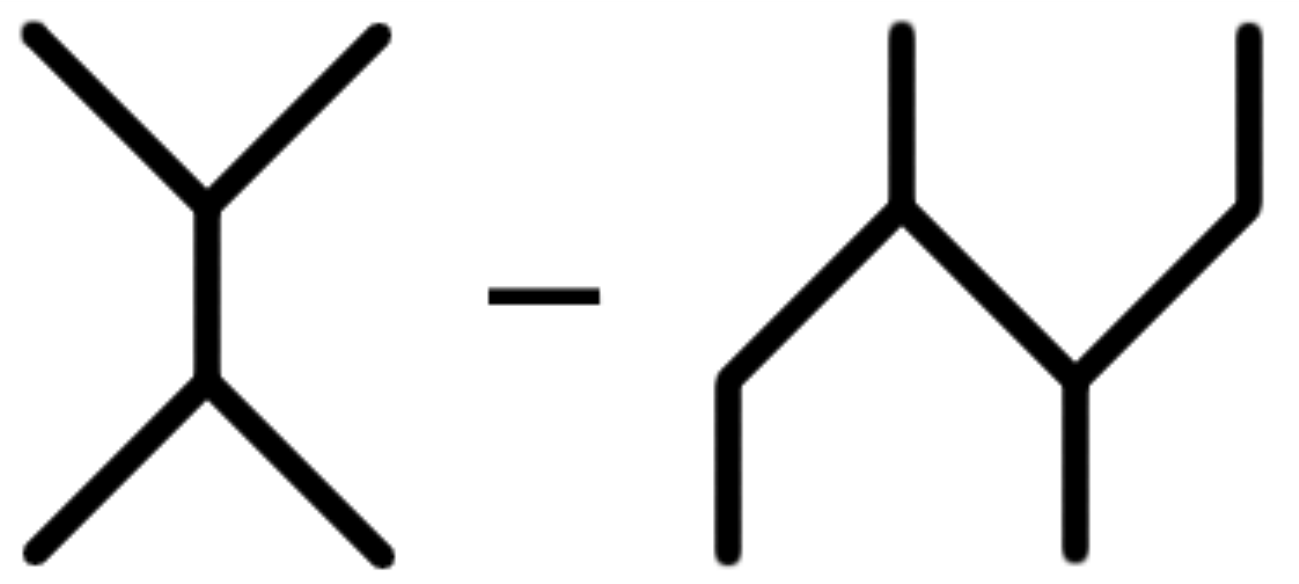}}}, \,
\vcenter{\hbox{\includegraphics[scale=0.1]{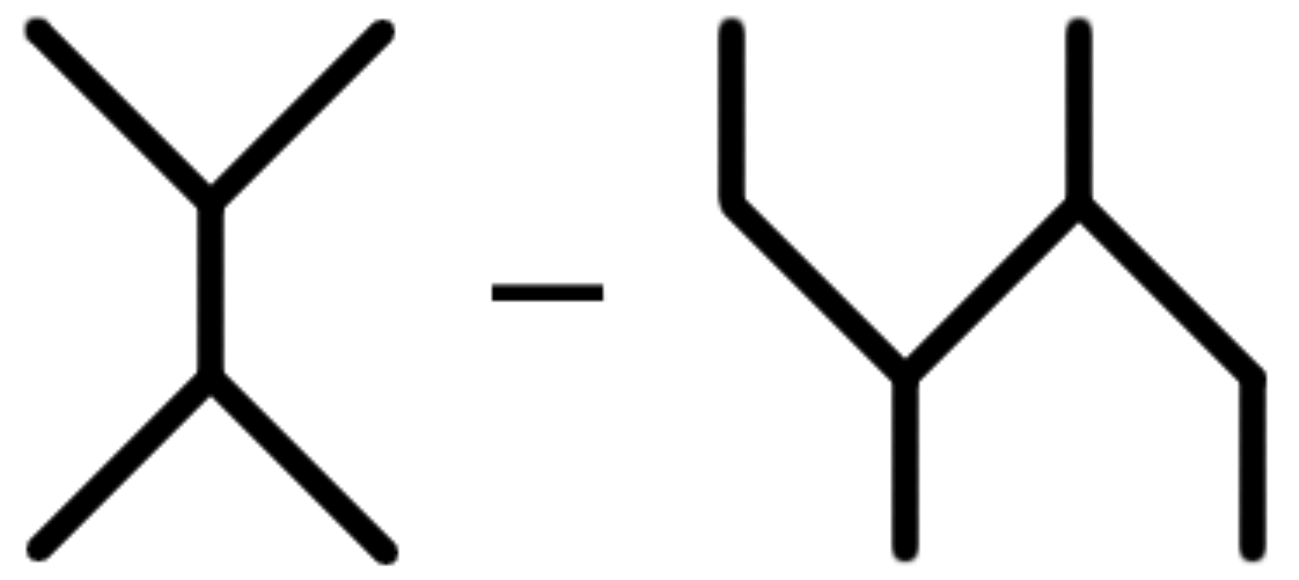}}}  
\right) \ .$$

\begin{ex}\label{ex:Frob}
Show that the properad $Frob$ is graded by the genus of the underlying graphs. Prove that the component $Frob_g(n,m)$ of genus $g$, with $n$ inputs and $m$ outputs,
is one dimensional. Then, make the properadic composition maps explicit. 
\end{ex}

\noindent
An \emph{involutive Frobenius bialgebra} is a Frobenius bialgebra satisfying the extra ``diamond'' condition: 
$\vcenter{\hbox{\includegraphics[scale=0.1]{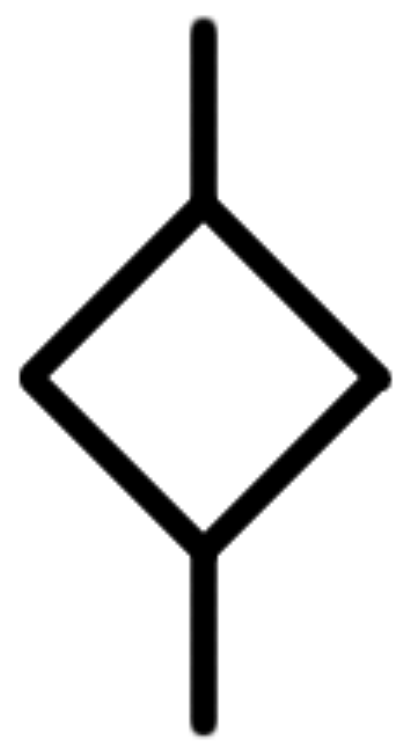}}}$. The associated properad admits the following quadratic presentation :
$$Frob_\diamond:=    \Po\left(\vcenter{\hbox{\includegraphics[scale=0.1]{Corolla2Simple.pdf}}}, \,
\vcenter{\hbox{\includegraphics[scale=0.1]{FIG50Coprod.pdf}}}; \,
\vcenter{\hbox{\includegraphics[scale=0.1]{FIG1TERAssociator.pdf}}}, \,
\vcenter{\hbox{\includegraphics[scale=0.1]{FIG51CoAssociator.pdf}}}, \,
\vcenter{\hbox{\includegraphics[scale=0.1]{FIG52MorphG.pdf}}}, \,
\vcenter{\hbox{\includegraphics[scale=0.1]{FIG53MorphD.pdf}}}, \,
\vcenter{\hbox{\includegraphics[scale=0.1]{FIG54Diamond.pdf}}}    
\right) \ .$$

\begin{ex}\label{ex:FrobD}
Show that in the properad $Frob_\diamond$, the composite of operations along graphs of genus greater or equal than $1$ vanishes. Prove that the component $Frob_\diamond(n,m)$  with $n$ inputs and $m$ outputs
is one dimensional. Then, make the properadic composition maps explicit. 
\end{ex}

\end{itemize}
\end{exams}

\begin{defi}[Quadratic-linear operad]
A \emph{quadratic-linear operad} is a quotient operad $\Po(E; R):=\TTT(E)/(R)$ generated by a \emph{quadratic-linear data} $(E; R)$: 
 $$R\subset  E \oplus \TTT(E)^{(2)} \ . $$
\end{defi}

\begin{exams}$ \ $

\begin{itemize}
\item[$\diamond$] The universal enveloping algebra $U(\g):=T(\g)/(x\otimes y - y \otimes x- [x,y])$ of a Lie algebra $\g$ is a quadratic-linear algebra. 

\smallskip 

\item[$\diamond$] The Steenrod algebra  $\mathcal{A}_2$  is the  quadratic-linear algebra
$$\mathcal{A}_2:=A(\lbrace Sq^i\rbrace_{i\ge 1}, R_{{Adem}}) $$
over the characteristic $2$ field $\mathbb{F}_2$, where $|Sq^i|=i$ and where $R_{{Adem}}$ stands for the \emph{Adem relations}:
$$Sq^i Sq^j = \binom{j-1}{i} Sq^{i+j} + \sum_{k=1}^{\left[ \frac{i}{2} \right]} \binom{j-k-1}{i-2k} Sq^{i+j-k} Sq^{k}\ , $$
for $i,j >0$ with $i<2j$.

\smallskip 

\item[$\diamond$] A \emph{Batalin-Vilkovisky algebra} is a Gerstenhaber algebra equipped with an extra square-zero unary operator $\Delta$ of degree $+1$ satisfying the following quadratic-linear relation: 
the bracket is the obstruction to $\Delta$ being a derivation with respect to  the product $\bullet$
$$\langle\,\textrm{-} , \textrm{-}\,\rangle\ =\ \Delta (\textrm{-} \bullet \textrm{-})\ -\
(\Delta(\textrm{-}) \bullet \textrm{-})   \ - \ (\textrm{-} \bullet
\Delta(\textrm{-})).$$
So, the operad $BV$ encoding Batalin-Vilkovosky algebras is a quadratic-linear operad.
\end{itemize}
\end{exams}

\begin{ex}
Show that any operad admits a quadratic-linear presentation. 
\end{ex}

\begin{defi}[Quadratic-linear-constant operad]
A \emph{quadratic-linear-constant operad} is a quotient operad $\Po(E; R):=\TTT(E)/(R)$ generated by a \emph{quadratic-linear-constant data} $(E; R)$: 
 $$R\subset  \KK \I \oplus E \oplus \TTT(E)^{(2)} \ . $$
\end{defi}

\begin{exams}$ \ $

\begin{itemize}
\item[$\diamond$] Let $V$ be symplectic vector space, with symplectic form $\omega : V^{\otimes 2} \to \KK$. Its \emph{Weyl algebra} is the quadratic-linear-constant algebra defined by 
$$W(V, \omega):= T(V)/(x\otimes y -y \otimes x -\omega(x,y))\ .$$
In the same way, let $V$ be a quadratic space, with quadratic form $q : V \to \KK$.  Its \emph{Clifford algebra} is the quadratic-linear-constant algebra defined by 
$$Cl(V, q):= T(V)/(x\otimes x -q(x))\ .$$

\smallskip

\item[$\diamond$] The nonsymmetric operad $uAs$ admits the following quadratic-linear-constant presentation:
$$uAs\cong\Po\left( \vcenter{\hbox{\includegraphics[scale=0.1]{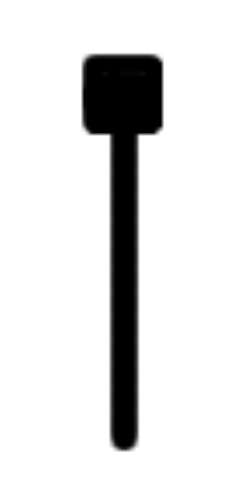}}}, 
\vcenter{\hbox{\includegraphics[scale=0.1]{Corolla2Simple.pdf}}}; 
\vcenter{\hbox{\includegraphics[scale=0.1]{FIG1TERAssociator.pdf}}}, 
\vcenter{\hbox{\includegraphics[scale=0.1]{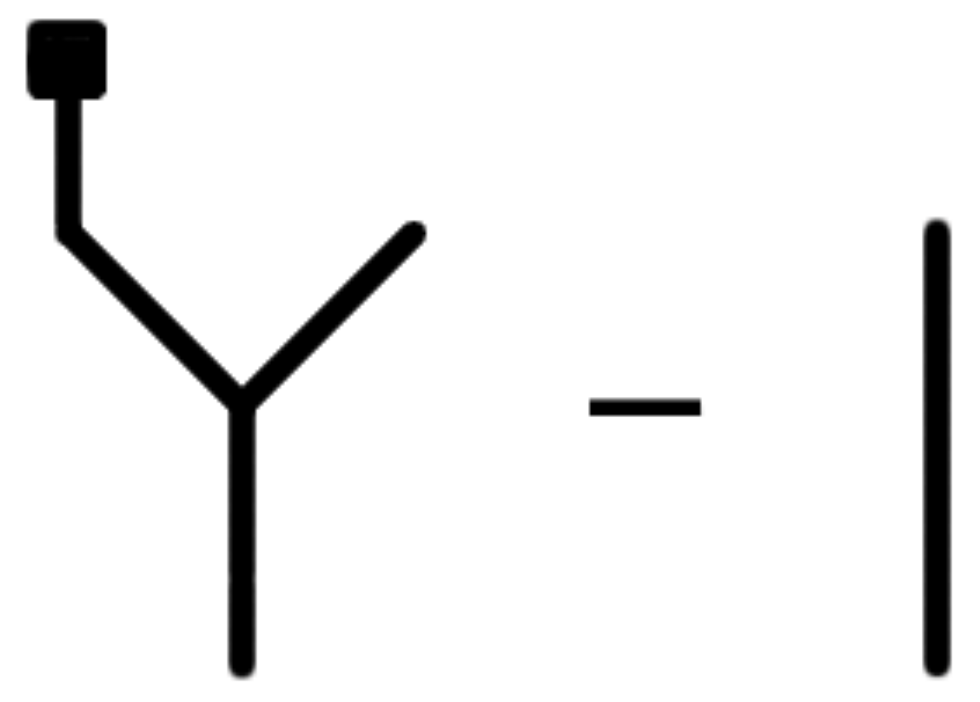}}}, 
\vcenter{\hbox{\includegraphics[scale=0.1]{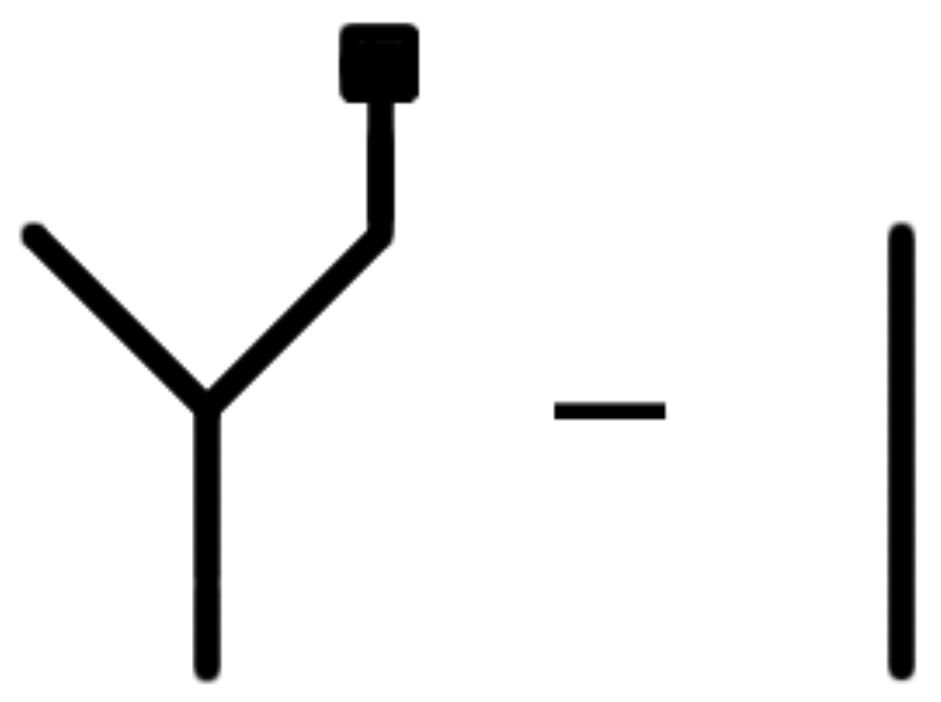}}}   \right)\ ,   $$
where the arity $0$ generator $\vcenter{\hbox{\includegraphics[scale=0.1]{FIG33Cork.pdf}}}$ encodes the unit.

\smallskip 

\item[$\diamond$] The properad $uFrob$ encoding unital and countial Frobenius bialgebras admits the following quadratic-linear-constant presentation: $uFrob:=$
$$\Po\left( 
\vcenter{\hbox{\includegraphics[scale=0.1]{FIG33Cork.pdf}}}, 
\vcenter{\hbox{\includegraphics[scale=0.1]{Corolla2Simple.pdf}}}, 
\vcenter{\hbox{\includegraphics[scale=0.1]{FIG50Coprod.pdf}}}; 
\vcenter{\hbox{\includegraphics[scale=0.1]{FIG1TERAssociator.pdf}}}, 
\vcenter{\hbox{\includegraphics[scale=0.1]{FIG51CoAssociator.pdf}}}, 
\vcenter{\hbox{\includegraphics[scale=0.1]{FIG52MorphG.pdf}}}, 
\vcenter{\hbox{\includegraphics[scale=0.1]{FIG53MorphD.pdf}}},  
\vcenter{\hbox{\includegraphics[scale=0.1]{FIG34CorkLeftUnit.pdf}}}, 
\vcenter{\hbox{\includegraphics[scale=0.1]{FIG35CorkRightUnit.pdf}}}, 
\vcenter{\hbox{\includegraphics[scale=0.1]{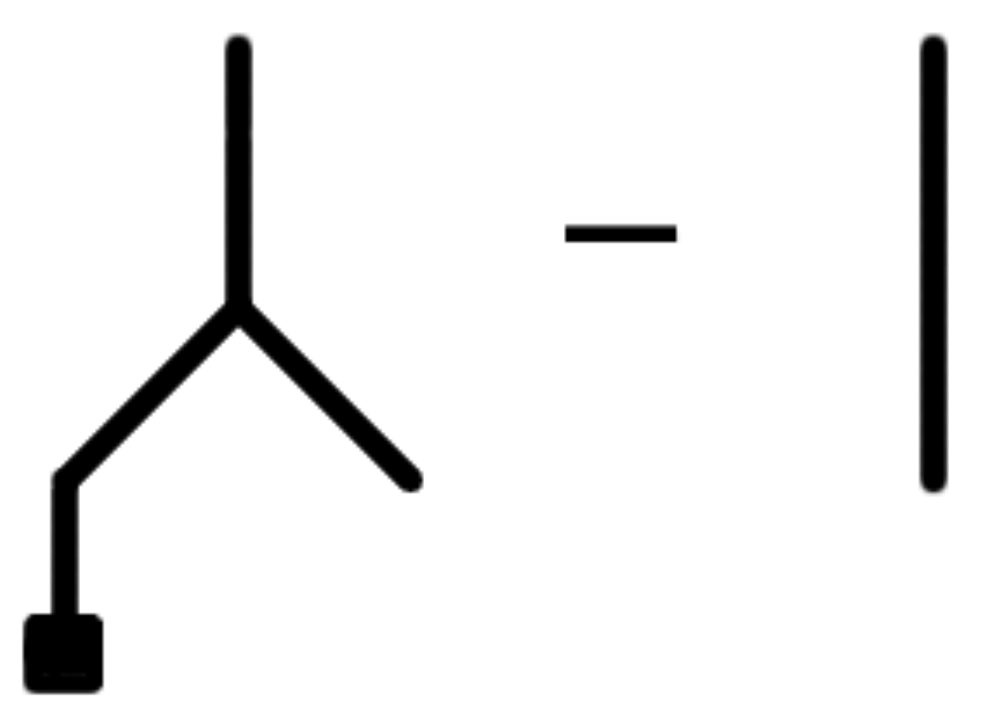}}}, 
\vcenter{\hbox{\includegraphics[scale=0.1]{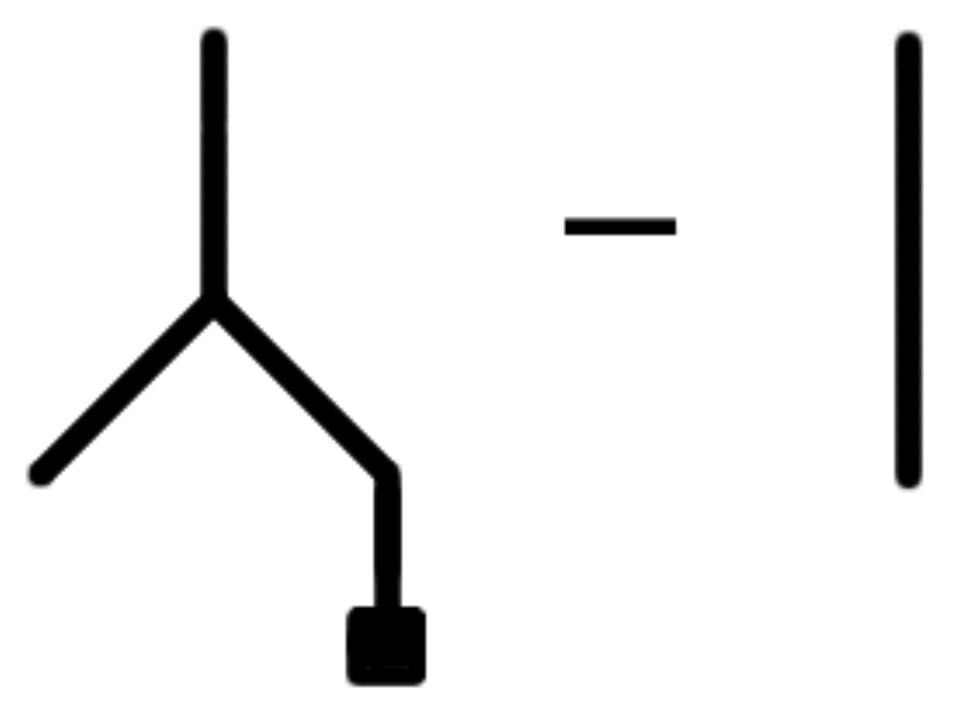}}}   \right) $$
 
 \noindent
 This notion classifies $2$-dimensional Topological Quantum Field Theories \cite{Kock04}.

\smallskip 

\end{itemize}
\end{exams}

\begin{ex}\label{Ex:Quotient=Initial}
Let $E$ be an $\Sy$-module and let $R$ be a sub-$\Sy$-module of $\TTT(E)$. We consider the category 
of operads $\TTT(E) \to \Po$ under $\TTT(E)$ such that the composite $R \mono \TTT(E) \to \Po $ is zero. A morphism in this category is a morphism of operads $\xymatrix@C=12pt{\Po \ar@{..>}[r] & \Qo}$ such that the following diagram 
$$\xymatrix@C=16pt@R=16pt{\TTT(E)  \ar[r] \ar[dr] &   \Po \ar@{..>}[d] \\
& \Qo} $$
is commutative. Show that the quotient operad $\Po(E; R)=\TTT(E)/(R)$ is the initial object of this category. 
\end{ex}

\subsection{Koszul duality theory}\label{subsec:KoszulDuality}

The Koszul duality theory first applies to operads having a quadratic presentation
$$\Po:=\Po(E;R)=\TTT(E)/(R), \quad R\subset \TTT(E)^{(2)} \ .$$ 
Working in the opposite category of vector spaces, that is changing the orientation of every map, one defines the notion of a cooperad. 
Dually, the  quadratic data $(E;R)$ induces a quadratic cooperad, which, in the Koszul case,  provides all the syzygies as follows. 

\begin{defi}[Nonsymmetric Cooperad]
A \emph{nonsymmetric cooperad} is an $\NN$-module $\lbrace \CCC_n \rbrace_{n\in \NN}$
 equipped with counital and coassociative \emph{decomposition maps}
$$\Delta_{i_1, \ldots, i_k} \ : \  \CCC_{i_1+\cdots+i_k} \to \CCC_k \otimes \CCC_{i_1} \otimes \cdots \otimes \CCC_{i_k} \ . $$
\end{defi}

The notion of \emph{(symmetric) cooperad} takes care of the symmetric groups action. 
 Dualizing the partial compositions, one gets \emph{partial decompositions}  
$$\Delta_{(1)} : \CCC \to \TTT(\CCC)^{(2)} \ , $$
which splits elements of $\CCC$ into two. 

\begin{ex}\label{Ex:LinearDual}
Show that the linear dual $\Po:=\CCC^*$ of a nonsymmetric cooperad is a nonsymmetric operad. 
Give an example of a nonsymmetric operad $\Po$ whose linear dual $\Po^*$ is not a nonsymmetric cooperad with the dual decomposition maps. 
Give a sufficient finite dimensional condition on $\Po$ for $\CCC:=\Po^*$ to be a nonsymmetric cooperad.
\end{ex}

One can dualize the definition of the free operad to get the notion of the \emph{cofree cooperad} $\TTT^c(E)$, see \cite[Section~$5.7.7$]{LodayVallette10} for more details. It shares the same underlying labeled-tree module as the free operad, with the decomposition maps given by the pruning of trees. (Since the notion of a cooperad is not the exact linear dual of the notion of an operad, as the above exercise shows, this cooperad only satisfies the universal property of the cofree cooperad only among \emph{connected} cooperads, that is cooperads for which the iteration of the decomposition maps produces only finite sums, cf. loc. cit.)\\

Let $(E, R)$ be a quadratic data, we consider the category dual to  the one introduced in Exercise~\ref{Ex:Quotient=Initial}: this is  the category 
of cooperads $\CCC \to \TTT^c(E) $ over $\TTT^c(E)$ such that the following composite $\CCC \to \TTT^c(E) \epi 
\TTT^c(E)^{(2)}/R$ is zero. 

\begin{defi}[Quadratic cooperad]
The \emph{quadratic cooperad} $\CCC(E;R)$ is the terminal object in the aforementioned category.  
\end{defi}

\begin{defi}[Koszul dual cooperad]
The \emph{Koszul dual cooperad} is the quadratic  cooperad
$$\Po^{\ac}:=\CCC(sE; s^2 R)\ , $$
where $s$  stands for the degree $+1$ suspension shift. 
\end{defi}

\begin{theo}[Koszul duality theory \cite{GinzburgKapranov94, GetzlerJones94}]
To any quadratic data $(E; R)$, there is a morphism of dg operads 
$$\boxed{\Po_\infty:=\big( \TTT(s^{-1}\overline{\Po}^{\ac}), d_2     \big)   \longrightarrow \Po(E;R)  }$$
where $\overline{\Po}^{\ac}:=\Po^{\ac}/\I$ and where $d_2$ is the unique derivation which extends the cooperad structure on $\Po^{\ac}$. 
\end{theo}

\begin{defi}
A quadratic operad $\Po$ is called a \emph{Koszul operad}, when the above map  is a quasi-isomorphism.
\end{defi}

 In this Koszul case, \underline {all the syzygies and the differential map of the resolution are produced at once}  by the quadratic data under the universal property of a quadratic cooperad. There are many methods for proving that $\Po$ is a Koszul operad; we will give an elementary one in the next section. 
\\

The careful reader might object that it is not a common task to compute a quadratic cooperad, though this is the proper conceptual object which produces all the syzygies. In general, one proceeds  as follows. We consider the \emph{Koszul dual operad} $\Po^!$ defined as (some desuspension of) the linear dual of the Koszul dual cooperad ${\Po^{\ac}}^*$. Exercise~\ref{Ex:LinearDual} shows that ${\Po^{\ac}}^*$ is always an operad. 

\begin{prop}[Koszul dual operad \cite{GinzburgKapranov94, GetzlerJones94}]
When the generating space $E$ is finite dimensional, the Koszul dual operad admits the following quadratic presentation 
$$\Po^!\cong  \Po(E^\vee; R^\perp)\ , $$
where 
$E^\vee=E(2)^*\otimes \sgn_{\Sy_2}$, 
when $E$ is concentrated in arity $2$ and degree $0$. 
\end{prop}
 
Finally, the method consists in computing the Koszul dual operad, via this formula, and linearizing the result to get the space of syzygies together with the differential. (For more details on how to compute $R^\perp$,  we refer the reader to \cite[Section~$7.6$]{LodayVallette10}.)

\begin{exams}$ \ $

\begin{itemize}

\item[$\diamond$] The algebra $D^!$,  Koszul dual to the algebra of dual numbers $D:=T(\delta)/(\delta^2)$, is the algebra of formal series in one variable 
$ D^!=\KK[[\hbar]]$.
The  coalgebra $D^{\ac}$, Koszul dual to the algebra of dual numbers $D$, is the cofree coassociative coalgebra $T^c(s\delta)$ on one generator.
The algebra of dual numbers is Koszul. A $D_\infty$-module, also know as \emph{multicomplex} in the literature,  is a chain complex $(A, d)$ equipped
with linear maps 
$$d_{n}: A\to A,\quad  \mathrm{ for }\  n\geq 1,\quad  |d_{n}|=n-1,$$
such that  the following identities hold
$$d d_n + (-1)^n d_n d =\sum_{\substack{i+j=n \\ i, j \geq 1}} (-1)^{i}d_{i}d_{j}, \quad \text{for any} \  n\geq 1\ .$$

\smallskip

\item[$\diamond$] The orthogonal $R_{As}^\perp$ of the associativity relation of the ns operad $As$ is again generated by the associator 
$$\vcenter{\hbox{\includegraphics[scale=0.15]{FIG1TERAssociator.pdf}}} $$ thanks to the signs and the suspensions. So the ns operad $As$ is Koszul autodual: $As^!=As$. Finally the Koszul resolution  $A_\infty \qi As$ coincides with the resolution given in Section~\ref{subsec:Syzygies}. 

\smallskip

\item[$\diamond$] The presentations of the operads $Lie$ and $Com$, given in \ref{subsec:Quad}, show  that  $Lie^!\cong Com$. The operad $Lie$ is Koszul, see next section. 

\noindent
A \emph{homotopy Lie algebra},  also called an \emph{$L_\infty$-algebra}, is 
a dg module $(A, d)$ equipped with  a family of skew-symmetric maps
$\ell_n \, : \,  A^{\otimes n} \to A$ of degree $|\ell_n|=n-2$, for all $n\ge 2$, which satisfy the relations
$$\sum_{p+q=n+1 \atop p, q>1} \sum_{\sigma\in Sh^{-1}_{p,q}} \textrm{sgn}(\sigma) (-1)^{(p-1)q} (\ell_p \circ_1 \ell_q)^\sigma=\partial(\ell_n)\ , $$
where $Sh_{p,q}$ stands for the set of 
$(p,q)$-\emph{shuffles}, that is permutations $\sigma \in
\Sy_{p+q}$ such that
$$\sigma(1)<\cdots <\sigma(p) \textrm{ and } \sigma(p+1)<\cdots <\sigma(p+q)\ .$$

\begin{ex}
 Let $(A, d, \{\mu_n\}_{n\ge 2})$ be an $A_\infty$-algebra structure on a dg module $A$.
Show that the anti-symmetrized maps $\ell_n \, :\, A^{\otimes n} \to A$, given by $$\ell_n:=\sum_{\sigma \in \Sy_n} \mathop{sgn}(\sigma){\mu_n}^\sigma,$$ endow the dg module $A$ with an $L_\infty$-algebra structure. 

\noindent
\texttt{Hint}. Use the morphism of cooperads $Lie^{\ac} \to Ass^{\ac}$ induced by the morphism of operads $Lie \to Ass$.
\end{ex}

\item[$\diamond$] 
The presentations of the operads $Com$ and $Lie$  show  that  $Com^!\cong Lie$. The operad $Com$ is Koszul, see next section. 

\noindent
A \emph{homotopy commutative algebra},  also called a \emph{$C_\infty$-algebra}, is 
 an $\Ai$-algebra $(A, d, \{\mu_{n}\}_{n\geq 2})$ such that each map 
$\mu_{n} : A^{\otimes n} \to A$  vanishes on the sum of all $(p,q)$-shuffles  for $p+q=n$.
This is proved by considering  the epimorphism of cooperads $Ass^{\ac} \epi Com^{\ac}$ and by showing  that its kernel is given by the sum of the shuffles (Ree's Theorem \cite{Ree58}). 
\smallskip

\item[$\diamond$] The operad  $Ger^{!}$, Koszul dual to the operad $Ger$ encoding Gerstenhaber algebras, is equal to $Ger$, up to some suspension. The notion of $G_\infty$-algebra is made explicit in \cite[Section~$2.1$]{GCTV09}.

\smallskip

\item[$\diamond$] The Koszul duality theory was extended to properads in \cite{Vallette07}. The properads $BiLie$ and $Frob_\diamond$ are sent to one another under the Koszul dual functor. In the same way,  the properads $BiLie_\diamond$ and $Frob$ are sent to one another: 
$$BiLie  \stackrel{!}{\longleftrightarrow}   Frob_\diamond     \quad \& \quad 
BiLie_\diamond  \stackrel{!}{\longleftrightarrow}   Frob \   . $$
One can prove, by means of distributive laws, that the properads $BiLie$ and $Frob_\diamond$ are Koszul. 
It is still an open conjecture to prove that the properads $BiLie_\diamond$ and $Frob$ are Koszul.

\noindent 
This theory defines the notion of \emph{homotopy Lie bialgebra} (resp. \emph{homotopy involutive Lie bialgebra}). One makes it explicit by using that the coproperad structure on $Frob_\diamond^*$ (resp. on $Frob^*$) is obtained by dualizing the result of Exercice~\ref{ex:FrobD} (resp. Exercice~\ref{ex:Frob}) , see \cite{Merkulov06} (resp. \cite{DCTT10}) for more details.  

\noindent
This theory also defines the notion of  \emph{homotopy Frobenius bialgebra} and \emph{homotopy involutive Frobenius bialgebra}, whose explicit form is more involved.

\end{itemize}
\end{exams}

When the operad $\Po:=\Po(E; R)$ admits a quadratic-linear presentation, $R\subset  E \oplus \TTT(E)^{(2)} $, we first consider the homogeneous quadratic operad 
$$q\Po:=\Po(E, qR) \ ,  $$
where $qR$ is the image of $R$ under the projection $q: \TTT(E) \epi \TTT(E)^{(2)}$. On the Koszul dual cooperad $q\Po^{\ac}$, we define a square-zero coderivation $d_\varphi$, which takes care of the inhomogenous relations. The \emph{Koszul dual dg cooperad} $\Po^{\ac}$ of $\Po$ is define by the dg cooperad  $(q\Po^{\ac}, d_\varphi)$.

\begin{theo}[Inhomogeneous Koszul duality theory \cite{GCTV09}]
For any quadratic-linear data $(E; R)$, there is a morphism of dg operads 
$$\boxed{\Po_\infty:=\big( \TTT(s^{-1}\overline{q\Po}^{\ac}), d_1+ d_2     \big)   \longrightarrow \Po(E;R)  }$$
where  $d_1$ is the unique derivation which extends $d_\varphi$.
\end{theo}

A quadratic-linear operad $\Po$ is called a \emph{Koszul operad}, when the above map  is a quasi-isomorphism. This  happens when the generating space $E$ is minimal, when the space of relations $R$ is maximal and when the homogeneous quadratic operad $q\Po$ is Koszul.

\begin{exams}$ \ $

\begin{itemize}
\item[$\diamond$] The Koszul dual dga coalgebra of the universal enveloping algebra $U(\g)$ is the following dga coalgebra
$$ U(\g)^{\ac}\cong (\Lambda^c(s\g), d_{CE}),$$
where $\Lambda^c$ stands for the cofree exterior coalgebra and where $d_{CE}$ is the Chevalley-Eilenberg boundary map defining the homology of the Lie algebra $\g$.

\smallskip 

\item[$\diamond$] The linear dual of the Koszul dual dga coalgebra of the Steenrod algebra $\mathcal{A}_2$ is a dga algebra, which is anti-isomorphic to the \emph{$\Lambda$ algebra} of \cite{BCKQRS66}. Notice that its homology gives the second page of the Adams spectral sequence which computes homotopy groups of spheres.

\smallskip

\noindent
The examples of the Steenrod algebra and the universal enveloping algebra of a Lie algebra prompted the Koszul duality theory, which was created by Stewart Priddy in \cite{Priddy70} for associative algebras.
\smallskip 

\item[$\diamond$] The operad $BV$ is an inhomogeneous Koszul operad, see \cite{GCTV09}, with Koszul dual dg cooperad 
$$BV^{\ac}\cong (T^c(s\Delta) \otimes Ger^{\ac}, d_\varphi) \ .$$
The dual ${}^td_\varphi$ of the differential is equal to $(s\Delta)^*\otimes d_{CE}$, where $d_{CE}$ is  the Chevalley-Eilenberg boundary map defining the homology of the free Lie algebra. The notion of $BV_\infty$-algebra is made explicit in \cite[Section~$2.3$]{GCTV09}.

\end{itemize}
\end{exams}

When the operad $\Po:=\Po(E; R)$ admits a quadratic-linear-constant presentation $R\subset  \KK \I \oplus E \oplus \TTT(E)^{(2)}$, we endow  the Koszul dual cooperad 
$q\Po^{\ac}$ with a coderivation $d_\varphi : q\Po^{\ac} \to q\Po^{\ac}$, which takes care of the linear part of the relations, and with a \emph{curvature} $\theta : q\Po^{\ac} \to \KK \I$, which takes care of the constant part of the relations. All together, these two maps satisfy the same relations as the curvature in geometry. 

\begin{theo}[Curved Koszul duality theory \cite{HirschMilles10}]
For any quadratic-linear-constant data $(E; R)$, there is a morphism of dg operads 
$$\boxed{\Po_\infty:=\big( \TTT(s^{-1}\overline{q\Po}^{\ac}), d_0+d_1+ d_2     \big)   \longrightarrow \Po(E;R)  }$$
where  $d_0$ is the unique derivation which extends the curvature $\theta$.
\end{theo}

A quadratic-linear-constant operad $\Po$ is called a \emph{Koszul operad}, when the above map  is a quasi-iso\-mor\-phi\-sm. This  happens under the same conditions as in the quadratic-linear case. Notice that, in this curved case, the Koszul dual cooperad $(q\Po^{\ac}, d_\varphi, \theta)$ is \underline{not} a dg object, whereas the Koszul resolution $\Po_\infty$ is a dg operad.

\begin{exams}$ \ $

\begin{itemize}
\item[$\diamond$] The curved Koszul duality theory was settled for algebra in \cite{Positselski93, PolischukPositselski05}, where the examples of the Weyl algebras and the Clifford algebras are treated.

\smallskip

\item[$\diamond$] The nonsymmetric operad $uAs$ is treated in detail in \cite{HirschMilles10} and is proved therein to be Koszul. A \emph{homotopy unital associative algebra}, or \emph{$uA_\infty$-algebra}, is 
a dg module $(A, d)$ equipped with  a family of  maps
$$\left\{ \mu_n^S=\vcenter{\hbox{\includegraphics[scale=0.18]{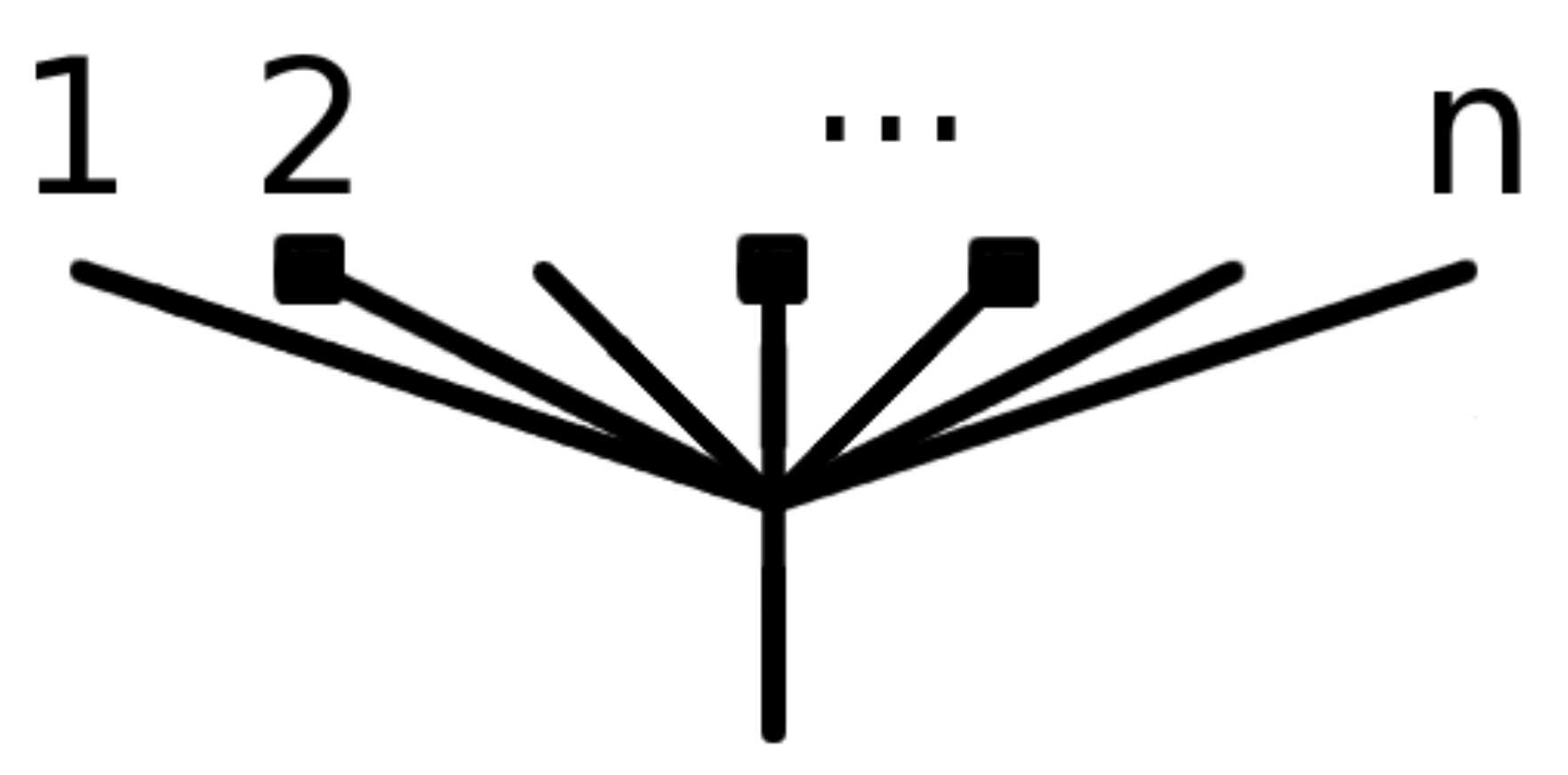}}}\ : \ A^{\otimes (n-|S|)}  \to A, \  | \mu_n^S|=n-2+|S|      \right\}\  ,  $$
 where the set $S$
runs over the subsets of $\lbrace1, \ldots, n\rbrace$  for any integer $n\ge 2$, and  where $S = \lbrace 1 \rbrace$,  for $n = 1$. We choose to represent them  by planar corollas with $n$ leaves and with the inputs in $S$ labelled by a ``cork''.  
They satisfy the following relations: 
\begin{eqnarray*}
&\vcenter{\hbox{\includegraphics[scale=0.2]{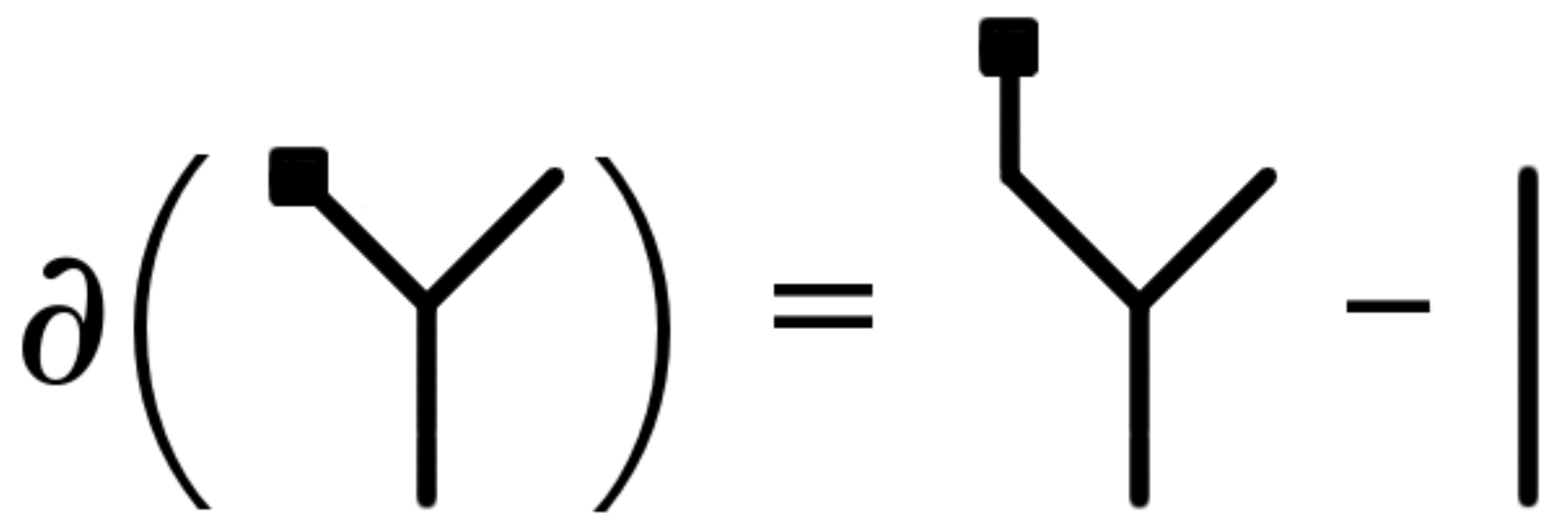}}} \ ,
\qquad \qquad 
\vcenter{\hbox{\includegraphics[scale=0.2]{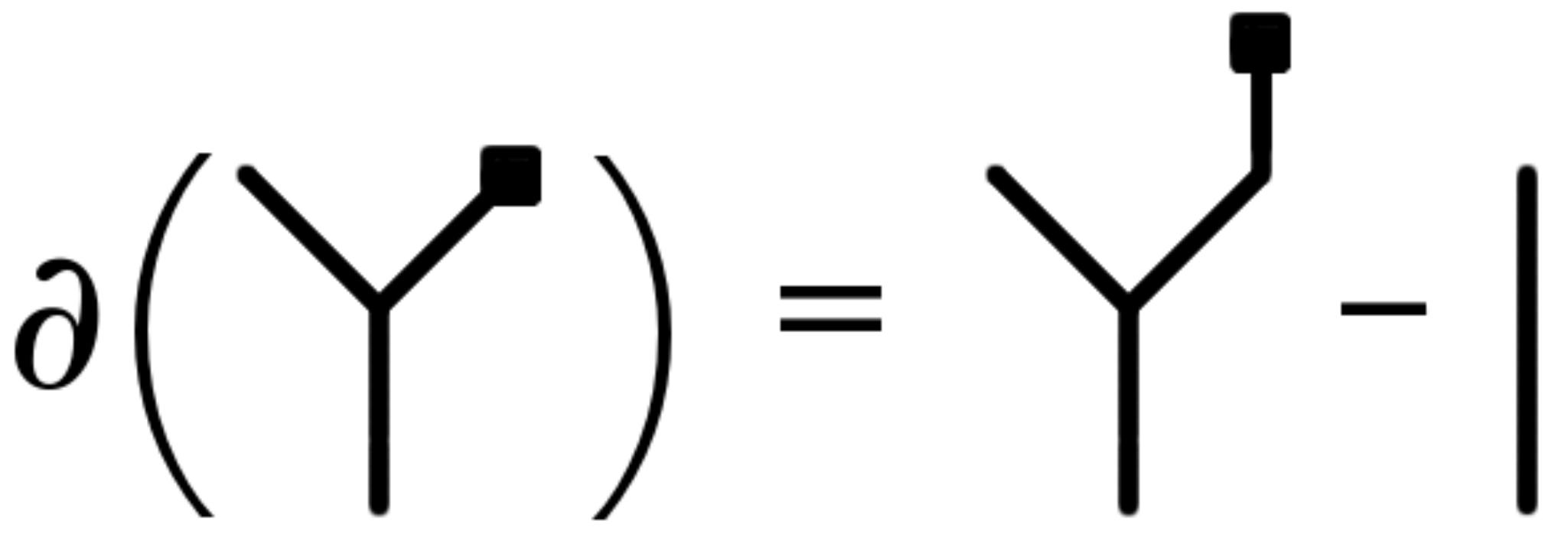}}} \ , &\\
&\vcenter{\hbox{\includegraphics[scale=0.2]{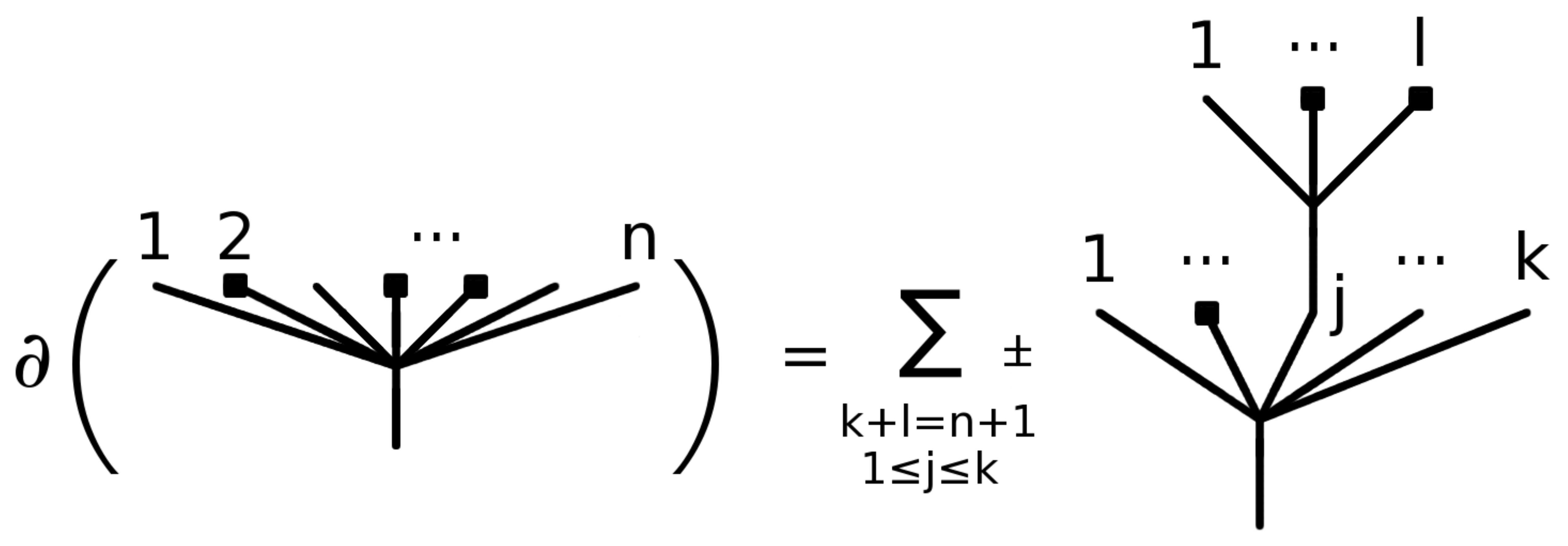}}}\ .&
\end{eqnarray*}

\noindent
This notion of a $uA_\infty$-algebra, obtained by the Koszul duality theory,  coincides with the notion of an ``$A_\infty$-algebra with a homotopy unit''
given and used in \cite{FOOO09I, FOOO09II}. The associated topological cellular operad was recently constructed by F. Muro and A. Tonks in \cite{MuroTonks11}.

\smallskip 

\item[$\diamond$] No Koszul resolution have been proved, so far,  for the properad $uFrob$. However, Joseph Hirsh and Joan Mill\`es introduced in \cite{HirschMilles10} a bar-cobar resolution of it (of curved origin).
\smallskip 

\end{itemize}
\end{exams}

\subsection{Rewriting method}\label{subsec:Rewriting}
In this section, we explain  the rewriting method, which provides a short algorithmic method, based on the rewriting rules given by the relations, to prove that an operad is Koszul. \\

Let us first explain the rewriting method for binary quadratic nonsymmetric operads $\Po(E;R)$, with the example of the ns operad $As$ in mind. 

\subsection*{Step 1}  We consider an ordered basis $\mu_1<\mu_2< \cdots < \mu_k$ for the generating space 
$E_2$ of binary operations. 

\begin{exam}
The ns operad $As$  is generated by one operation of arity $2$: $\vcenter{\hbox{\includegraphics[scale=0.15]{Corolla2Simple.pdf}}}$.
\end{exam}

\subsection*{Step 2}
The weight $2$ part of the free ns operad $\TTT(E)$ is spanned by the left combs  $\mu_{i}\circ_{1}\mu_{j}$ and by the right combs $\mu_{i}\circ_{2}\mu_{j}$. We consider the following total order on this set:
$$
\left\{
\begin{array}{ll}
\mu_{i}\circ_{2}\mu_{j} \ < \ \mu_{i}\circ_{1}\mu_{j},&  \quad \mathrm{for}\ \textrm{any}\ i,j,\\
\mu_{i}\circ_{a}\mu_{j} \ < \ \mu_{k}\circ_{a}\mu_{l}, & \quad \mathrm{whenever}\ i<k, a=1\ \mathrm {or} \ 2,\  \mathrm{and} \ \mathrm{for} \ \mathrm{any}\ j,l,\\
\mu_{i}\circ_{a}\mu_{j} \ < \ \mu_{i}\circ_{a}\mu_{l},& \quad \mathrm{whenever}\ j<l, a=1\ \mathrm {or} \ 2. 
\end{array}
\right.$$

The operad $\Po$ is determined by the space of relations $R$, which is spanned by a set of relators written in this basis as 
$$r \ = \ \lambda\,  \mu_{i}\circ_{a}\mu_{j} - \sum\lambda_{k,b,l}^{i,a,j}\ \mu_{k}\circ_{b}\mu_{l}, \quad \lambda, \lambda_{k,b,l}^{i,a,j}\in \KK\ \mathrm{and} \ \lambda\neq 0\ ,$$
where  the sum runs over the indices satisfying $ \mu_{i}\circ_{a}\mu_{j} > \mu_{k}\circ_{b}\mu_{l}$. The operation $ \mu_{i}\circ_{a}\mu_{j} $  is called the \emph{leading term} of the relator $r$. One can always suppose  that $\lambda$ is equal to $1$, that the leading terms of the set of relators are all distinct and that there is no leading term of any other relator in the sum on the right hand side. This is called a \emph{normalized form} of the presentation.

\begin{exam}
There is only one relator in $As$: 
$$r=\vcenter{\hbox{\includegraphics[scale=0.15]{FIG1TERAssociator.pdf}}}
 \ .$$
\end{exam}

\subsection*{Step 3}

Every relator gives rise to a \emph{rewriting rule} in the operad $\Po$: 
$$\mu_{i}\circ_{a}\mu_{j} \mapsto \sum\lambda_{k,b,l}^{i,a,j}\ \mu_{k}\circ_{b}\mu_{l}.$$

\begin{exam}$ \ $

\begin{center}
\includegraphics[scale=0.15]{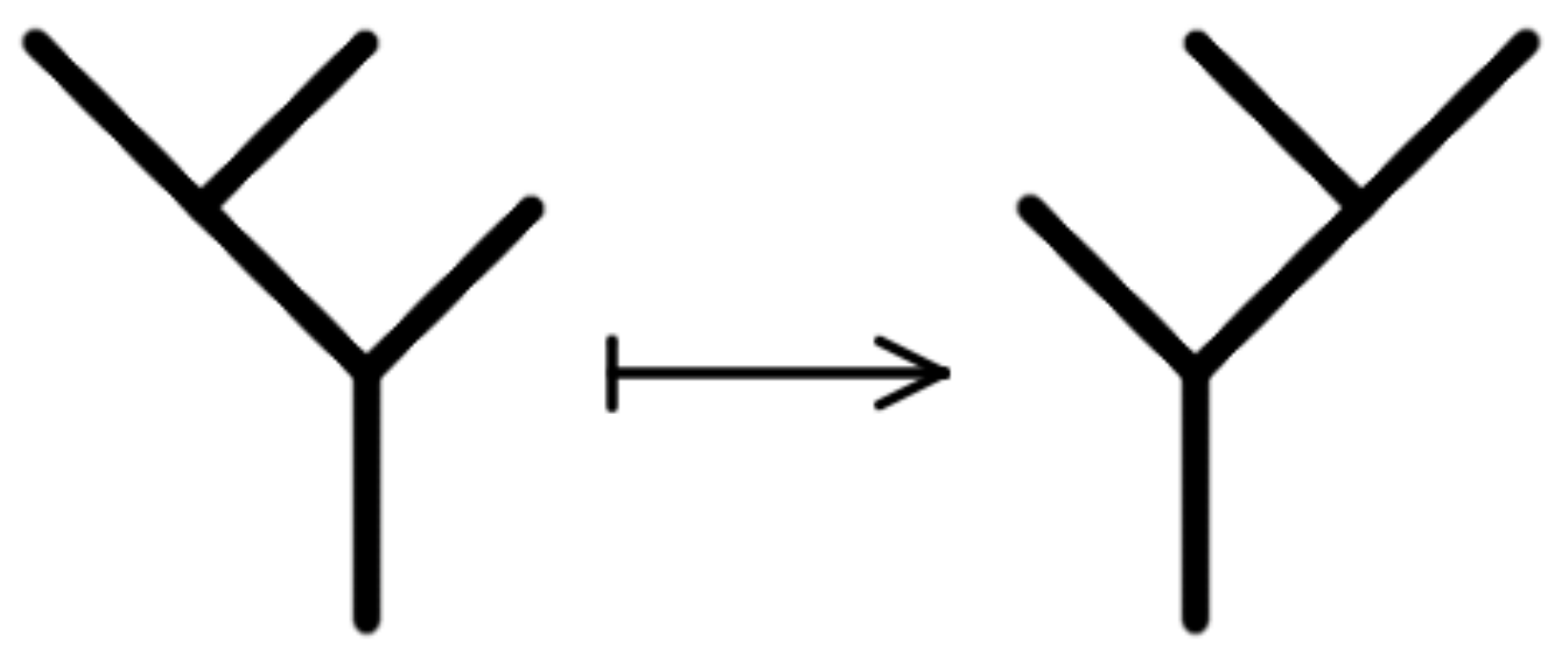}
\end{center}
\end{exam}

Given three generating binary operations $\mu_{i}, \mu_{j}, \mu_{k}$, one can compose them in $5$ different ways: they correspond to the $5$ planar binary trees with $3$ vertices. Such a monomial, i.e.\  decorated planar tree,   is called \emph{critical}  if the two sub-trees with $2$ vertices are leading terms. 

\begin{exam}
For the nonsymmetric operad $As$, the left comb  is the only following critical monomial.
\begin{center}
\includegraphics[scale=0.15]{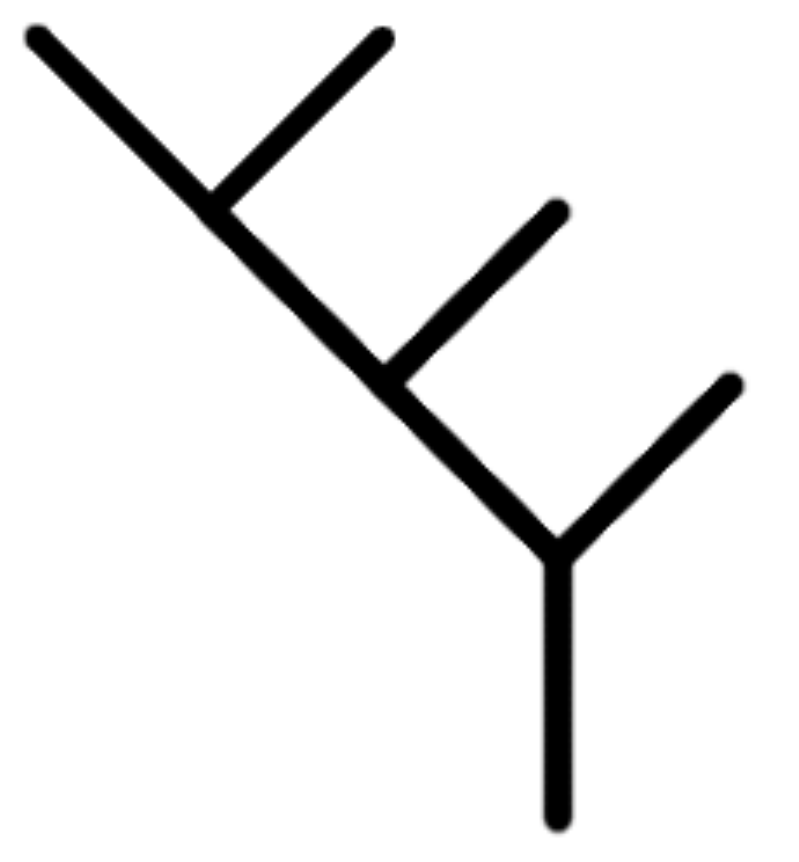}
\end{center}
\end{exam}

\subsection*{Step 4} There are at least two ways of rewriting a critical monomial ad libitum, that is, until no rewriting rule can be applied any more. If all these ways lead to the same element, then the critical monomial is said to be \emph{confluent}. 

\begin{exam}
The critical monomial of $As$ gives rise to the confluent graph.
\begin{center}
\includegraphics[scale=0.15]{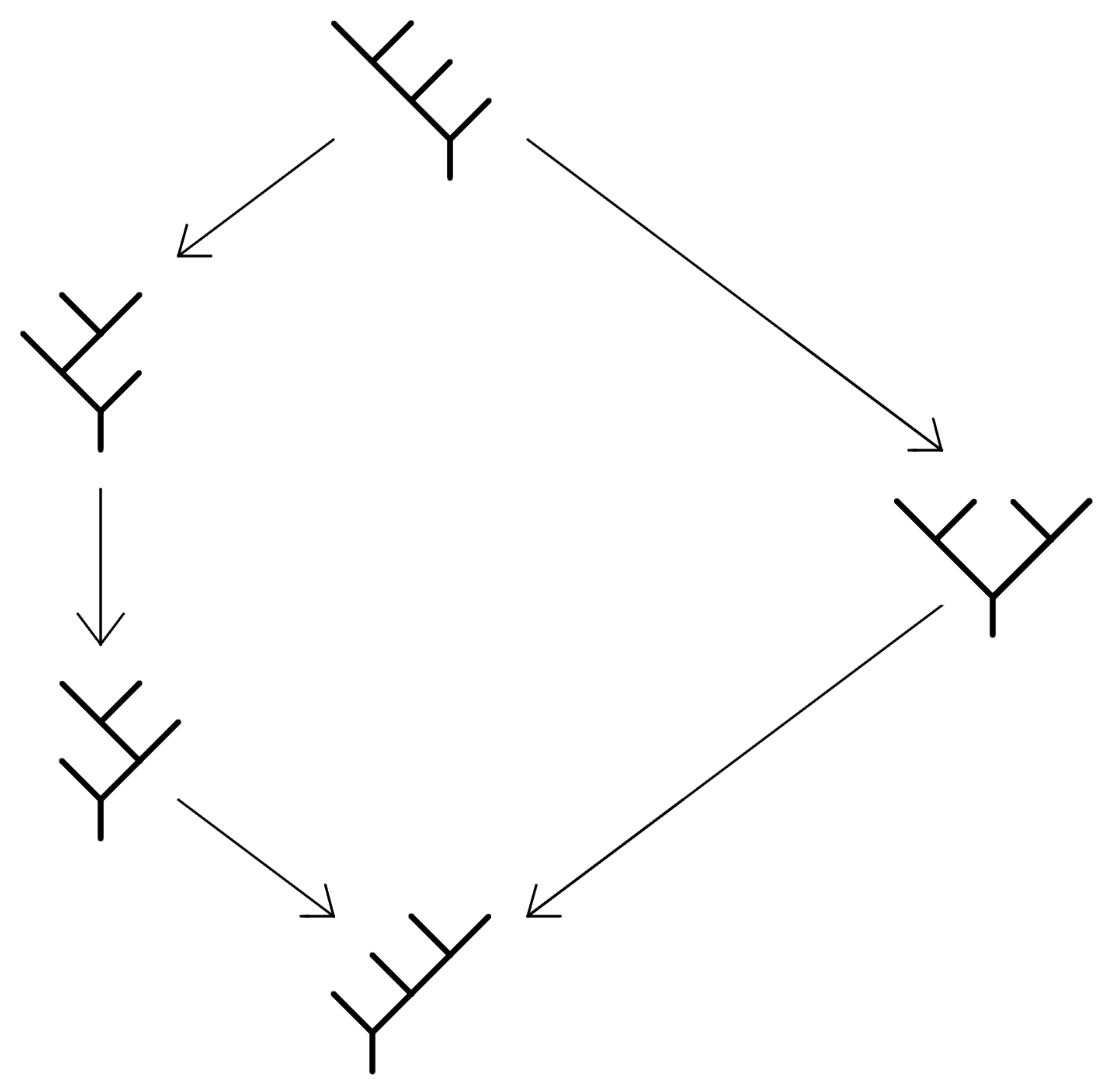}
\end{center}
\end{exam}

\subsection*{Conclusion} If every critical monomial is confluent, then the ns operad $\Po$ is Koszul. \\

\begin{ex}
Consider the same example but with the modified associativity relation 
$$\vcenter{\hbox{\includegraphics[scale=0.15]{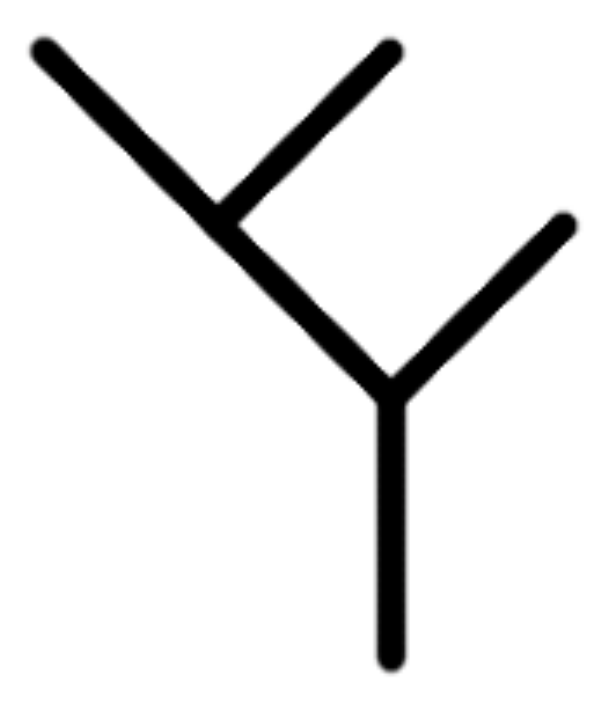}}}
=2\, \vcenter{\hbox{\includegraphics[scale=0.15]{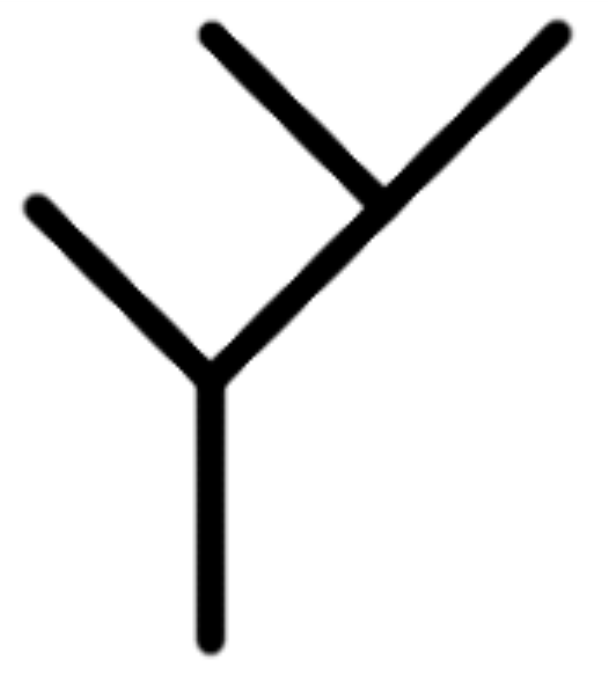}}} \ .$$
Show that the only critical monomial is not confluent. (It can be proved that this ns operad is not Koszul.)
\end{ex}

The general method holds for quadratic  operads $\Po(E;R)$ as follows. We forget the symmetric group action and we consider a $\KK$-linear ordered basis $\mu_1<\mu_2< \cdots < \mu_k$ for the generating space 
$E$. Then, the only difference lies in Step~$2$, where one has to use a good  basis for the free operad $\TTT(E)$. To this end, we introduce the set $\TTT_\shuffle$ of  \emph{shuffle trees} , which are 
 planar rooted trees equipped with a bijective labeling of the leaves by integers $\{1,2,  \ldots ,n\}$ satisfying the following condition. First, we label each edge  by the minimum of the labels of the inputs of the above vertex. Secondly, we require that, for each vertex, the labels of the inputs, read from left to right, are increasing.
 
\begin{center}
\includegraphics[scale=0.3]{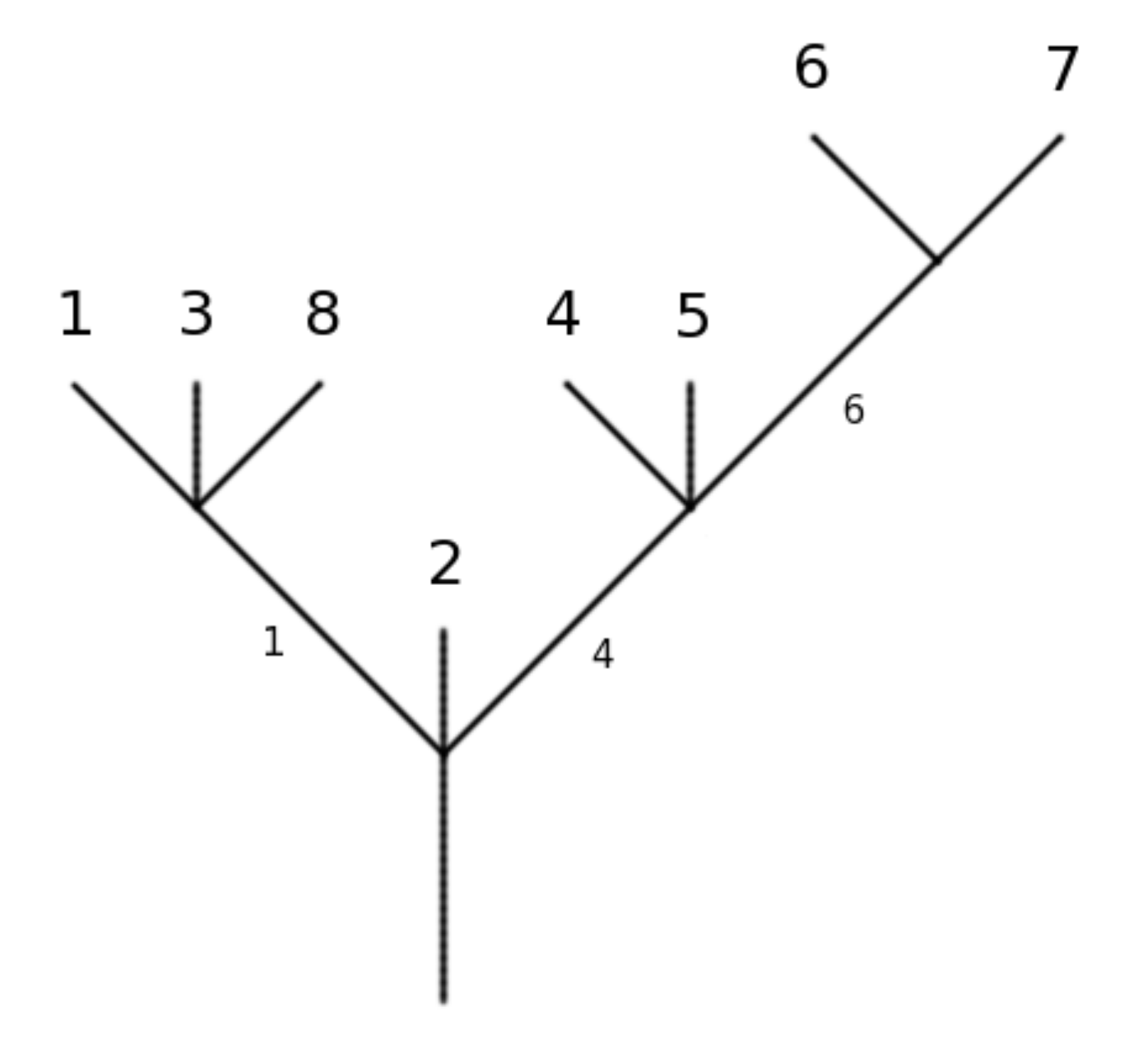}
\end{center}

\begin{prop}[Shuffle tree basis \cite{Hoffbeck10, DotsenkoKhoroshkin10}]
The set $\TTT_\shuffle(\lbrace 1, 2, \ldots , k \rbrace)$ of shuffle trees with vertices labelled by $1, 2, \ldots, k$ 
forms  a $\KK$-linear basis of the free operad $\TTT(E)$.
\end{prop}

We then consider a \emph{suitable order} on shuffle trees, that is a total order, which makes the partial compositions into strictly increasing maps.  
For instance, the \emph{path-lexicographic order}, defined in \cite{Hoffbeck10},  is a suitable order. On binary shuffle trees with $2$ vertices, it is given by 
\begin{eqnarray*}
& \vcenter{\xymatrix@R=8pt@C=8pt@M=0pt{ &*+<8pt>{2}\ar@{-}[dr]& &*+<8pt>{3} \ar@{-}[dl]\\
*+<8pt>{1}\ar@{-}[dr] && *+<10pt>[o][F-]{1}\ar@{-}[dl]&    \\
&*+<10pt>[o][F-]{1}\ar@{-}[d] && \\
& && }}
<\vcenter{\xymatrix@R=8pt@C=8pt@M=0pt{ &*+<8pt>{2}\ar@{-}[dr]& &*+<8pt>{3} \ar@{-}[dl]\\
*+<8pt>{1}\ar@{-}[dr] && *+<10pt>[o][F-]{2}\ar@{-}[dl]&    \\
&*+<10pt>[o][F-]{1}\ar@{-}[d] & \\&
& && }}< \cdots <\vcenter{\xymatrix@R=8pt@C=8pt@M=0pt{ &*+<8pt>{2}\ar@{-}[dr]& &*+<8pt>{3} \ar@{-}[dl]\\
*+<8pt>{1}\ar@{-}[dr] && *+<10pt>[o][F-]{k}\ar@{-}[dl]&    \\
&*+<10pt>[o][F-]{1}\ar@{-}[d] && \\
& && }}<
\vcenter{\xymatrix@R=8pt@C=8pt@M=0pt{ &*+<8pt>{2}\ar@{-}[dr]& &*+<8pt>{3} \ar@{-}[dl]\\
*+<8pt>{1}\ar@{-}[dr] && *+<10pt>[o][F-]{1}\ar@{-}[dl]&    \\
&*+<10pt>[o][F-]{2}\ar@{-}[d] && \\
& && }}<& \\
%%%%%%%%%%%%%%%%%%%%%%%%
&\cdots <\vcenter{\xymatrix@R=8pt@C=8pt@M=0pt{ &*+<8pt>{2}\ar@{-}[dr]& &*+<8pt>{3} \ar@{-}[dl]\\
*+<8pt>{1}\ar@{-}[dr] && *+<10pt>[o][F-]{k}\ar@{-}[dl]&    \\
&*+<10pt>[o][F-]{k}\ar@{-}[d] && \\
& && }}<
\vcenter{\xymatrix@R=8pt@C=8pt@M=0pt{ *+<8pt>{1}\ar@{-}[dr]& &*+<8pt>{3}\ar@{-}[dl]  &\\
 & *+<10pt>[o][F-]{1}\ar@{-}[dr]& & *+<8pt>{2} \ar@{-}[dl]  \\
&&*+<10pt>[o][F-]{1}\ar@{-}[d] & \\
& && }}<
\vcenter{\xymatrix@R=8pt@C=8pt@M=0pt{ *+<8pt>{1}\ar@{-}[dr]& &*+<8pt>{2}\ar@{-}[dl]  &\\
 & *+<10pt>[o][F-]{1}\ar@{-}[dr]& & *+<8pt>{3} \ar@{-}[dl]  \\
&&*+<10pt>[o][F-]{1}\ar@{-}[d] & \\
& && }}<\vcenter{\xymatrix@R=8pt@C=8pt@M=0pt{ *+<8pt>{1}\ar@{-}[dr]& &*+<8pt>{3}\ar@{-}[dl]  &\\
 & *+<10pt>[o][F-]{2}\ar@{-}[dr]& & *+<8pt>{2} \ar@{-}[dl]  \\
&&*+<10pt>[o][F-]{1}\ar@{-}[d] & \\
& && }}<
\cdots \ .&
\end{eqnarray*}

The rewriting method relies on the following result. 

\begin{theo}[Rewriting method] Let $\Po(E;R)$ be a quadratic (resp. nonsymmetric) operad. If its generating space $E$ admits an ordered basis, for which there exists a suitable order on shuffle (resp. planar) trees such that every critical monomial is confluent, then the (resp. nonsymmetric) operad $\Po$ is Koszul. 
\end{theo}

\begin{proof}
This theorem is proved by a version of the Diamond Lemma of George M. Bergman \cite{Bergman78} for operads, see \cite{DotsenkoKhoroshkin10} and \cite[Chapter~$8$]{LodayVallette10}.
\end{proof}

\begin{ex}
Apply the  rewriting method to the operad $Lie$ to show that it is Koszul. 
First, prove that, under the path-lexicographic order, the Jacobi relation becomes the following rewriting rule:

\begin{center}
\includegraphics[scale=0.2]{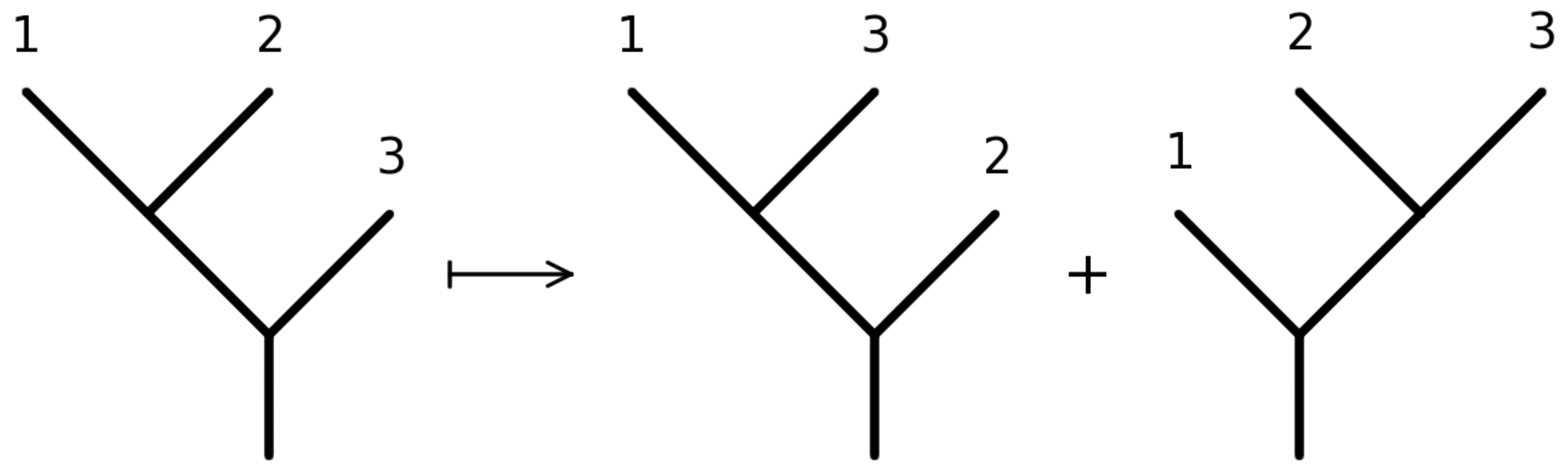} 
\end{center}

Finally, show that there is only one critical monomial, which is confluent:

\begin{center}
\includegraphics[scale=0.15]{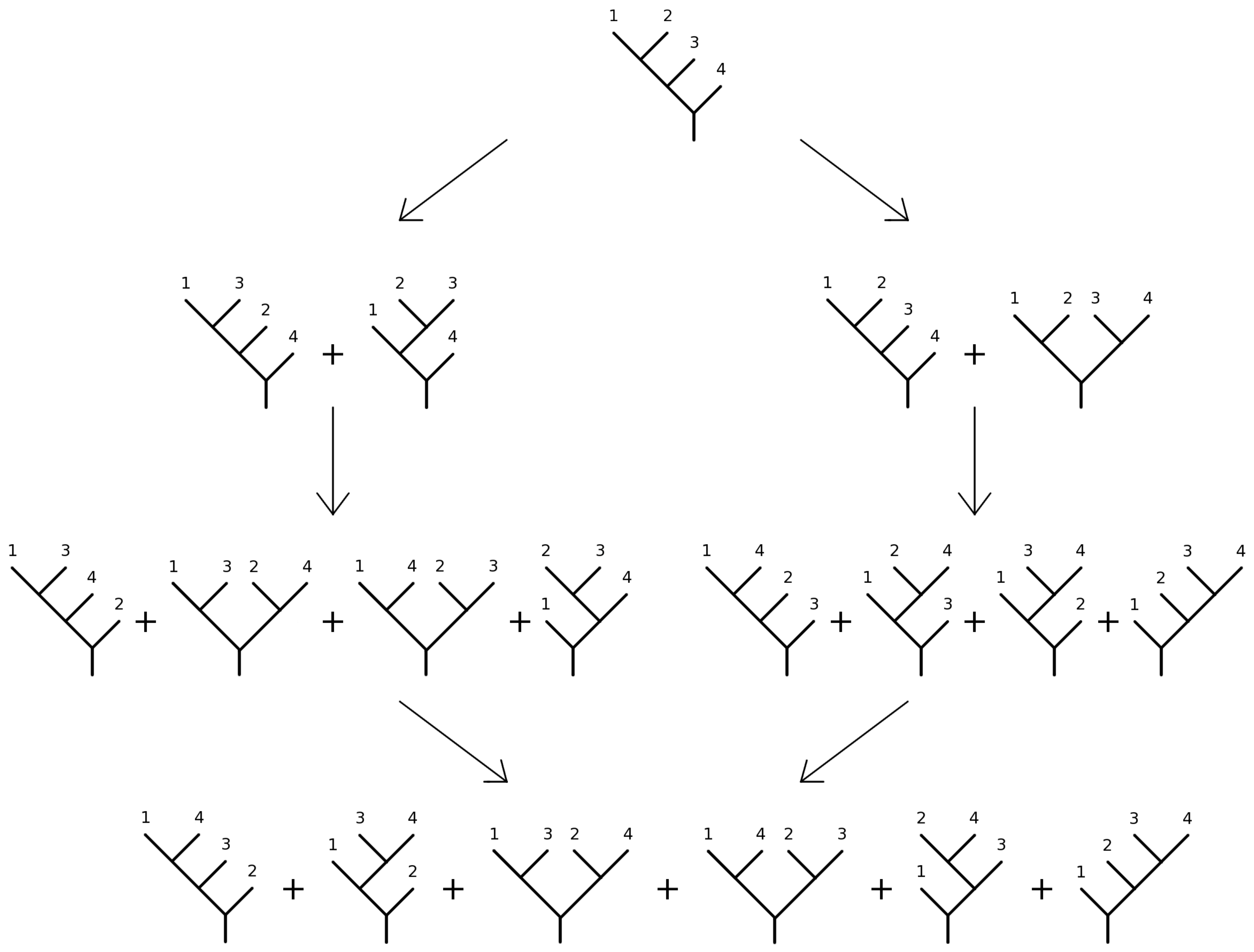}
\end{center}
\end{ex}

\begin{rema}
Notice that the graph given here is a compact version of the full rewriting graph of the operad $Lie$. There is another way to draw this rewriting diagram, which gives the Zamolodchikov tetrahedron equation. It leads to the categorical notion of \emph{Lie $2$-algebra}, see Baez-Crans \cite[Section~$4$]{BaezCrans04}.
\end{rema}

\begin{ex}
Apply the  rewriting method to the operad $Com$, to show that it is Koszul. 
\end{ex}

When the rewriting method applies, the operad $\Po$ admits a $\KK$-linear basis made up of some planar trees called a \emph{Poincar\'e-Birkhoff-Witt basis}, \emph{PBW basis} for short. The PBW basis is given by the set of shuffles (resp. planar) trees with vertices labelled by $1, 2, \ldots, k$ such that no sub-tree with $2$ vertices is a leading term. 

\begin{exams}$ \ $ 

\begin{itemize}
\item[$\diamond$] The ns operad $As$ admits a 
PBW basis  made up of the right combs: 
\begin{center}
\includegraphics[scale=0.2]{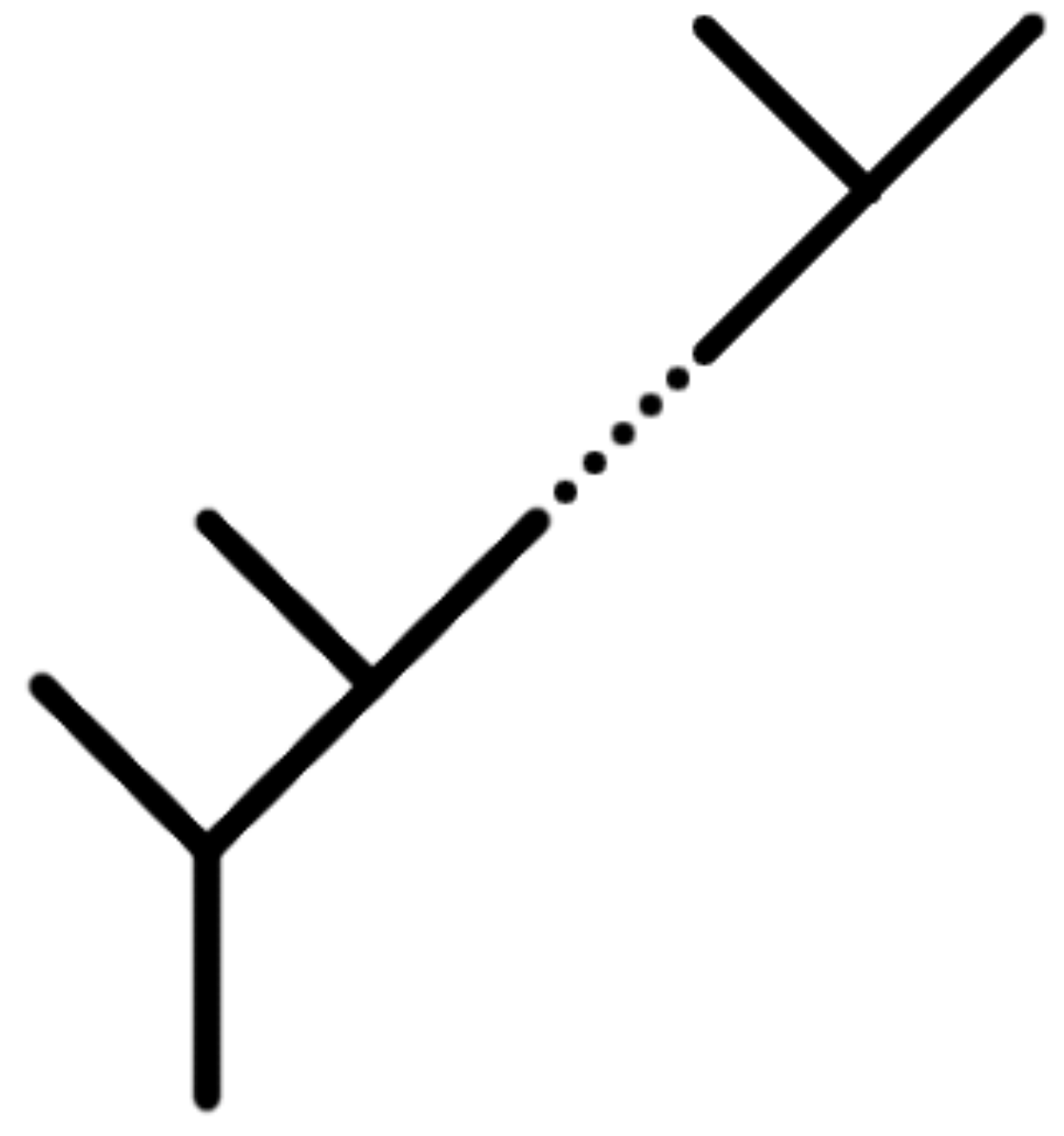}
\end{center}

\noindent
It is interesting to notice that one recovers exactly Mac Lane's coherence theorem for (non-unital) monoidal categories \cite{MacLane98}. This is not a so big surprise since both actually  relies on the Diamond Lemma. With this remark in mind, the reading of \cite[Section~$\text{VII-}2$]{MacLane95} enjoys another savor. 

\smallskip

\item[$\diamond$] The 
$MR_n$-trees, defined in Section~\ref{subsec:SymmetricOp}, form the PBW basis of the 
operad $Lie$. It appears quite often in the literature when one wants to prove that some representation of the symmetric group, coming either from algebraic topology or algebraic combinatorics, is isomorphic to $Lie(n)$, see for instance \cite{Cohen76, RobinsonWhitehouse02,  Turchin06, SalvatoreTauraso09}.
\end{itemize}
\end{exams}

The notion of operadic PBW basis provides a special basis for quotient operads $\Po(E; R)=\TTT(E)/(R)$. It was introduced by Eric Hoffbeck in \cite{Hoffbeck10}. The dual notion of \emph{Gr\"obner basis} for the operadic ideal $(R)$ was introduced by Vladimir Dotsenko and Anton Khoroshkin in \cite{DotsenkoKhoroshkin10}.

\section{Homotopy Transfer Theorem}

The purpose of this section is to extend all the results given in Section~\ref{Sec:Alg+Homo} on the level of associative algebras to any category of algebras over a Koszul operad $\Po$. We first state a  Rosetta stone, which gives four equivalent definitions of a $\Po_\infty$-algebra. The Homotopy Transfer Theorem can be proved with the third equivalent definition and the notion of an $\infty$-morphism is defined with the fourth one. This provides us with all the necessary tools to study the homotopy theory of $\Po_\infty$-algebras.

\subsection{The Rosetta stone for homotopy algebras}
In this section, we give three other equivalent definitions of a  $\Po_\infty$-algebra.\\

Let $\CCC$ be a dg cooperad and let $\Qo$ be a dg operad. Their partial compositions and decompositions endow the space of $\Sy$-equivariant maps 
$$\Hom_\Sy(\CCC, \Qo):=\prod_{n\in \NN} \Hom_{\Sy_n}(\CCC(n), \Qo(n))$$ 
with a dg Lie algebra structure $\big({\Hom}_\Sy(\CCC, \Qo), [\; ,\,], \partial\big)$, called the \emph{convolution algebra}. 

\begin{defi}[Twisting morphism]
A \emph{twisting morphism} $\alpha$ is a solution $\alpha: \CCC \to  \Qo$ of degree $-1$ to the 
 the Maurer-Cartan equation 
$$\partial (\alpha) +\frac{1}{2}[\alpha, \alpha] =0 $$
in the convolution  dg Lie algebra.
The associated set is denoted by $\Tw(\CCC, \Qo)$.
\end{defi}

This notion is sometimes called ``twisting cochain'' in algebraic topology.
This  bifunctor can be represented both on the left-hand side and on the right-hand side by the following \emph{bar-cobar adjunction}. 
$$\Omega \ : \ \textsf{(augmented) dg cooperads} \rightleftharpoons
\textsf{(conilpotent) dg operads} \ : \ \B\ .$$ 

\begin{defi}[Bar and cobar constructions]
The \emph{bar construction} $\B \Qo$ of an (augmented) dg operad $(\Qo, d_\Qo)$ is the dg cooperad 
$$\B \Qo := (\TTT^c(s\overline{\Qo}), d_1+d_2)\ , $$
where $d_1$ is the unique coderivation which extends $d_\Qo$ and where $d_2$ is the unique coderivation 
which extends the partial compositions of the operad $\Qo$. (An augmented operad is an operad such that $\Qo \cong \I \oplus \overline{\Qo}$ is a morphism of operads).

 The \emph{cobar construction} $\Omega \CCC$ of a (coaugmented) dg cooperad $(\CCC, d_\CCC)$ is the dg operad 
$$\Omega \, \CCC := (\TTT(s^{-1}\overline{\CCC}), d_1+d_2)\ , $$
where $d_1$ is the unique derivation which extends $d_\CCC$ and where $d_2$ is the unique derivation 
which extends the partial decompositions of the cooperad $\CCC$. 
\end{defi}

Notice that the Koszul resolution of Section~\ref{subsec:KoszulDuality} is given by the cobar construction of the Koszul dual cooperad: 
$$\Po_\infty=\Omega \, \Po^{\ac}\ . $$

We apply the preceding definitions to the dg cooperad $\CCC:=\Po^{\ac}$ and to the dg operad $\Qo:=\End_A$. Let us begin with 
$$\Hom_\Sy(\Po^{\ac}, \End_A)=\prod_{n \in \NN} \Hom_{\Sy_n}(\Po^{\ac}(n), \Hom(A^{\otimes n}, A))
\cong \Hom\left(\bigoplus_{n\in \NN}\Po^{\ac}(n)\otimes_{\Sy_n} A^{\otimes n}, A\right)\ .$$ 
We denote by $\Po^{\ac}(A):=\bigoplus_{n\in \NN}\Po^{\ac}(n)\otimes_{\Sy_n} A^{\otimes n}$; this is the cofree $\Po^{\ac}$-coalgebra structure on $A$. Since coderivations on cofree $\Po^{\ac}$-coalgebras are completely characterized by their projection $\Po^{\ac}(A) \to A$ onto the space of generators, we get 
$\Hom_\Sy(\Po^{\ac}, \End_A)\cong \mathrm{Coder}(\Po^{\ac}(A))$. Under this isomorphism, twisting morphisms correspond to \emph{square-zero} coderivations, that we call \emph{codifferentials} and the set of which we  denote $\mathrm{Codiff}(\Po^{\ac}(A))$.

\begin{theo}[Rosetta Stone]\label{thm:RosettaStone}
The set of $\Po_\infty$-algebra structures on a dg module $A$ is equivalently given by 
$$\boxed{\underbrace{\Hom_{\mathsf{dg\ Op}}(\Omega \, \Po^{\ac},\, {\End}_A)}_{\rm Definition} \cong  
\underbrace{\Tw(\Po^{\ac},\, {\End}_A)}_{\rm Deformation \ theory} \cong 
\underbrace{\Hom_\mathsf{dg\ Coop}(\Po^{\ac},\, \B \, {\End}_A)}_{\rm HTT}
\cong 
\underbrace{\mathop{\rm Codiff}(\Po^{\ac}(A))}_{\rm \infty-morphism}.}$$ 
\end{theo}

We indicated which equivalent definition suits which question best. The first one was used in Section~\ref{sec:OpSyzygies} to define $\Po_\infty$-algebras. The second one is used to study the deformation theory of 
$\Po$-algebras and $\Po_\infty$-algebras, see \cite[Section~$12.2$]{LodayVallette10}. The third one will be used in Section~\ref{subsec:HTT} to state the Homotopy Transfer Theorem (HTT). We are going to use the fourth one in the next section to define the notion of an $\infty$-morphism between $\Po_\infty$-algebras.

%\espace

%Recall that the notion of a coassociative coalgebra is the dual notion of an associative algebra: it is  a vector space $C$ equipped with a coassociative decomposition map $C \to C \otimes C$. The tensor module $\bar{T}^c(V):=\bigoplus_{k \ge 1} V^{\otimes k}$ equipped with the deconcatenation map 
%$$v_1  \cdots v_{k} \mapsto \sum_{i=1}^{k-1} v_1 \cdots v_i \otimes v_{i+1} \cdots v_{k} \ .$$
%forms a coassociative coalgebra, which 

%\begin{ex}
%Show that the data $\lbrace \mu_n : A^{\otimes n} \to n \rbrace_{n\ge 2}$ of an $A_\infty$-algebra structure on a chain complex $(A, d)$ is equivalent to a square-zero coderivation on the cofree (non-unital) coassociative coalgebra $\bar{T}^c(sA)$. 
%\end{ex}

\begin{ex}
Describe the aforementioned four equivalent definitions in the case of $A_\infty$-algebras. 

\noindent 
\texttt{Hint.} The Koszul resolution $\Omega \, As^{\ac}$ was described in \ref{subsec:Syzygies}. The convolution dg Lie algebra $\Tw(As^{\ac},\, {\End}_A)$ is the Hochschild cochain complex together with the Gerstenhaber bracket \cite{Gerstenhaber63}. Finally, the data of an $A_\infty$-algebra structure on $A$ is equivalently given by  a square-zero coderivation on the noncounital cofree coassocative coalgebra $\overline{T}^c(sA)$. 
\end{ex}

\begin{ex}
Describe the aforementioned four equivalent definitions in the case of $L_\infty$-algebras. 

\noindent 
\texttt{Hint.} The Koszul resolution $\Omega \, Lie^{\ac}$ was described in \ref{subsec:KoszulDuality}. The convolution dg Lie algebra $\Tw(Lie^{\ac},\, {\End}_A)$ is the Chevalley-Eilenberg cochain complex together with the Nijenhuis-Richardson bracket  \cite{NijenhuisRichardson66, NijenhuisRichardson67}.
Finally, the data of an $L_\infty$-algebra structure on $A$ is equivalently given by a square-zero coderivation on the noncounital cofree cocommutative coalgebra $\overline{S}^c(sA)$. 
\end{ex}

\subsection{$\infty$-morphism} 
We use the fourth definition given in the Rosetta stone to define a notion of morphisms between $\Po_\infty$-algebras with nice properties. 

\begin{defi}[$\infty$-morphism]
Let $A$ and $B$ be two $\Po_\infty$-algebras. An \emph{$\infty$-morphism} $A \rightsquigarrow B$ between $A$ and $B$ is a morphism $\Po^{\ac}(A) \to \Po^{\ac}(B)$ of dg $\Po^{\ac}$-coalgebras.
\end{defi}

Two such morphisms are obviously composable and the associated category is denoted 
$\infty\textsf{-}\Po_\infty\textsf{-alg}$.

\begin{ex}
Show that, in the case of the ns operad $As$, ones recovers the definition of an $\Ai$-morphism given in Section~\ref{subsec:AiMorph}. 
\end{ex}

\begin{ex}
Make explicit the notion of $\infty$-morphism between $L_\infty$-algebras. This notion is called \emph{$L_\infty$-morphism} in the litterature. 
\end{ex}

As a morphism to a cofree $\Po^{\ac}$-coalgebra, an $\infty$-morphism is completely charactrized by its projection  $\Po^{\ac}(A) \to B$ onto the space of generators. Such a data is equivalent to giving a morphism of $\Sy$-modules 
$\Po^{\ac} \to \End^A_B$, where $\End^A_B:=\lbrace \Hom(A^{\otimes n}, B)\rbrace_{n\in \NN}$. The commutativity of the respective codifferentials translates into some relation satisfied by this latter map $\Po^{\ac} \to \End^A_B$.

\begin{ex}
Make this relation explicit in terms of the decomposition maps and the partial decompositions of the cooperad $\Po^{\ac}$.
\end{ex}

For any $\infty$-morphism $f : A \rightsquigarrow B$, the image of $\I \in \Pac$ produces a chain map $f_1  : A\to B$. 

\begin{defi}[$\infty$-isomorphism and $\infty$-quasi-isomorphism]
When the map $f_1$ is an isomorphism (resp. a quasi-isomorphism), the $\infty$-morphism  $f$ is called an \emph{$\infty$-isomorphism} (resp. an \emph{$\infty$-quasi-isomorphism}). 
\end{defi}

\begin{prop}
The $\infty$-isomorphisms are the invertible morphisms of the category $\infty\textsf{-}\Po_\infty\mathsf{-alg}$. 
\end{prop}

\begin{proof} We use the weight grading on the Koszul dual cooperad $\Pac \subset \TTT^c(sE)$ and we notice then that the  proof is similar to the proof that power series $a_1 x + a_2 x^2 + \cdots$ with invertible first term   are invertible. 
\end{proof}

\subsection{Homotopy Transfer Theorem}\label{subsec:HTT}
We now  have all the tools to prove the following result, with explicit constructions. 

\begin{theo}[Homotopy Transfer Theorem \cite{GCTV09}]\label{TransferThm}
Let $\Po$ be a Koszul operad and let $(H, d_H)$ be a homotopy retract of $(A, d_A)$:
\begin{eqnarray*}
&\xymatrix{     *{ \quad \ \  \quad (A, d_A)\ } \ar@(dl,ul)[]^{h}\ \ar@<0.5ex>[r]^{p} & *{\
(H,d_H)\quad \ \  \ \quad }  \ar@<0.5ex>[l]^{i}}&\\
& \Id_A-i p =d_A  h+ h  d_A\ \text{and} \ i \  \text{quasi-isomorphism}.
\end{eqnarray*}
Any $\Po_\infty$-algebra structure on $A$ can be transferred into a $\Po_\infty$-algebra structure on $H$ such that $i$ extends to an $\infty$-quasi-isomorphism.
\end{theo}

\begin{rema}
The existence part of the theorem can also be proved by model category arguments, see Clemens Berger and Ieke Moerdijk \cite{BergerMoerdijk03} and Benoit Fresse \cite{Fresse09ter}. 
\end{rema}

\begin{proof}
The proof is based on the third definition $\Hom_\mathsf{dg\ Coop}(\Po^{\ac},\, \B \, {\End}_A)$ of the Rosetta stone together with a morphism of dg cooperads $\Psi : \B\, \End_A \to \B\, \End_H$, that we describe below: 
$$\begin{array}{ccc}
\Hom_\mathsf{dg\ Coop}(\Po^{\ac},\, \B \, {\End}_A) &\xrightarrow{\Psi_*} & \Hom_\mathsf{dg\ Coop}(\Po^{\ac},\, \B \, {\End}_H)     \\
&&\\
{ \rm Initial \  structure}\  \nu & \mapsto &{\rm  Transferred\  structure}\ \mu:=\Psi  \circ \nu .
\end{array}        $$
\end{proof}

\begin{lemm}\cite[Theorem~$5.2$]{VanDerLaan03}\label{lemm:HomoMorphPsi}
Let $(H, d_H)$ be a homotopy retract of $(A, d_A)$. The unique morphism of cooperads which extends 
$\psi : \TTT^c(s \End_A) \to \End_H$:
\begin{center}
\includegraphics[scale=0.23]{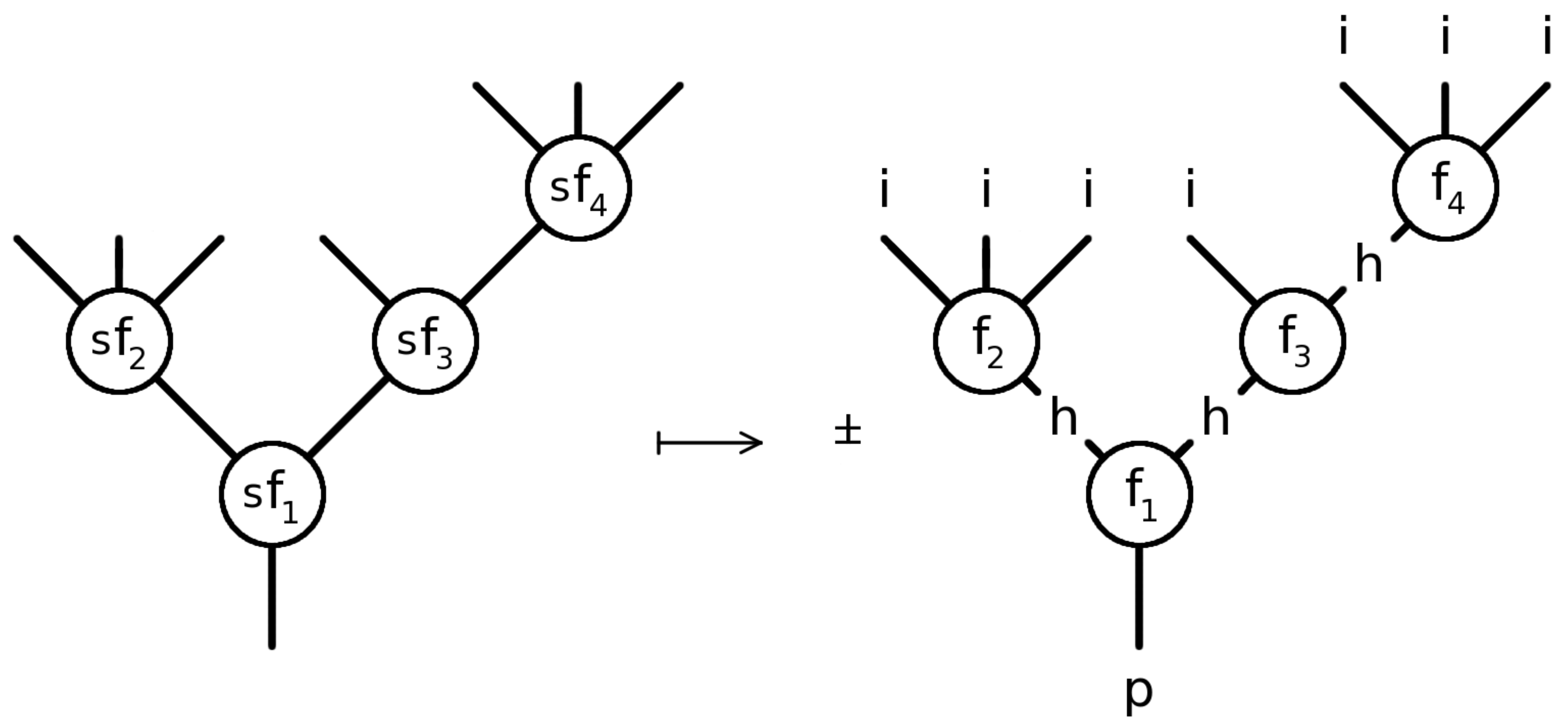} 
\end{center}
is a  morphism of dg cooperads  $\Psi : \B\, \End_A \to \B\, \End_H$ between the cobar constructions. 
\end{lemm}

Explicitly, under the identification $\Tw(\Po^{\ac},\, {\End}_A) \cong 
\Hom_\mathsf{dg\ Coop}(\Po^{\ac},\, \B \, {\End}_A)$, 
if $\nu : \Po^{\ac} \to \End_A$ denotes the initial $\Po_\infty$-algebra on $A$, then the transferred $\Po_\infty$-algebra $\mu : \Po^{\ac} \to \End_H$ on $H$ is equal to the following composite 
$$\boxed{\begin{array}{ccccccc}
\Po^{\ac} & \xrightarrow{\Delta_{\Po^{\ac}}}& \TTT^c(\Po^{\ac}) &\xrightarrow{\TTT^c(s\nu)}& \TTT^c(s{\End}_A) &\xrightarrow{\psi} & {\End}_H \ ,\\
&{ \rm (type \ of \ algebra)} && {\rm (initial\   structure)} && {\rm (homotopy \ data)}&
\end{array}}$$
where the first map $\Delta_{\Po^{\ac}} :  \Po^{\ac} \to  \TTT^c(\Po^{\ac})$ is given by all the possible iterations of the partial decomposition maps of the cooperad $\Po^{\ac}$. This composite is made up of $3$ \underline{independent} terms corresponding respectively to the type of algebraic structure considered, to the initial algebra structure, and to the homotopy data. 

\begin{exams}$ \ $

\begin{itemize}

\item[$\diamond$]  In the example of  the ns operad $As$, its Koszul dual cooperad is the linear dual $As^*$, up to suspension. So, the full decomposition map $\Delta_{As^*} : As^*\to \TTT^c(As^*)$ produces all the planar trees out of any corolla:
\begin{center}
\includegraphics[scale=0.2]{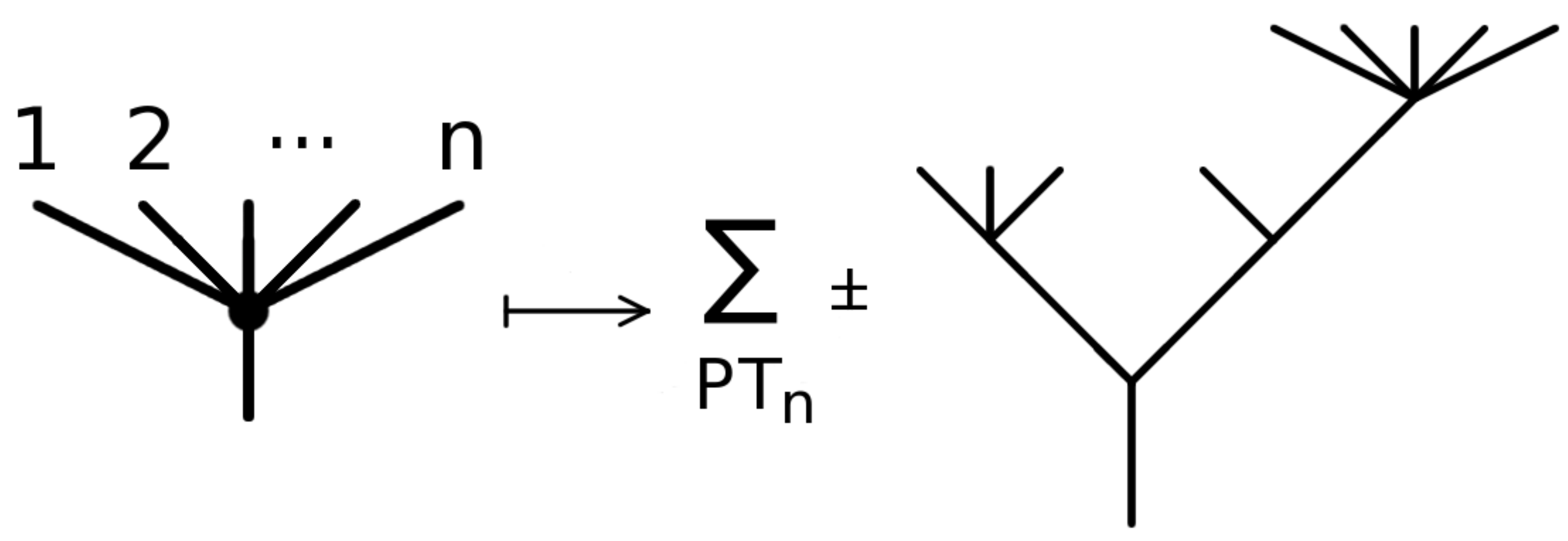}
\end{center}
The second map labels their vertices with the initial operations $\lbrace \nu_n\rbrace_{n\ge 2}$ according to the arity. The third map labels the inputs by $i$, the internal edges by $h$, and the root by $p$.  Finally, we get the formula given in Theorem~\ref{HTT2}.

\smallskip

\item[$\diamond$] The Koszul dual cooperad $Lie^{\ac}=Com^*$ of the operad $Lie$ is the linear dual of $Com$, up to suspension. So, the full decomposition map $\Delta_{Com^*} :Com^*\to \TTT^c(Com^*)$ splits any corolla into the sum  of all rooted trees. If $\{ \ell_n\,:  \, A^{\otimes n} \to A\}_{n\ge 2}$ stands for the initial  $L_\infty$-algebra structure on $A$, then the transferred $L_\infty$-algebra structure $\{ l_n\,:  \, H^{\otimes n} \to H\}_{n\ge 2}$  is equal to
$$l_n=\sum_{t\in RT_n} \pm \,  p \, t(\ell, h) \, i^{\otimes n}\ ,$$
where the sum runs over rooted trees $t$ with $n$ leaves and where the notation $t(\ell, h)$ stands for the $n$-multilinear operation on $A$ defined by the composition scheme $t$ with vertices labeled by the operations $\ell_k$ and the internal edges labeled by $h$.

\smallskip

\item[$\diamond$] Recall that a $C_\infty$-algebra is an $A_\infty$-algebra whose structure maps $\mu_n : A^{\otimes n} \to A$ vanish on the  sum of all $(p,q)$-shuffles for $p+q=n$. Using the morphism of cooperads 
$Ass^{\ac} \to Com^{\ac}$ and the aforementioned boxed formula, 
we leave it to the reader to prove that the planar tree formula for the transfer of $A_\infty$-algebra structures applies to $C_\infty$-algebras as well. See \cite{ChengGetzler08} for a proof by direct computations.

\smallskip

\item[$\diamond$] The data of a dg $D$-module is equivalent to the data of a bicomplex $(A, d, \delta)$, that is a module equipped with two anti-commuting square-zero unary operators. Considering the homotopy retract $(H, d_H):=(H_\bullet(A, d), 0)$ of $(A,d)$, the transferred $D_\infty$-module structure $\lbrace d_n : H \to H \rbrace_{n\ge 1}$ is a lifted version of the spectral sequence associated to the bicomplex $(A, d, \delta)$. Indeed, one can easily see that the formula for the transferred structure
$$d_n:=  \pm\, p  \,  \underbrace{\delta \, h\,  \delta\,  h\,  \,\cdots\, \delta\, h\, \delta}_{n\  \text{times}\  \delta}\, i $$
corresponds to the diagram-chasing formula  \cite[\S~$14$]{BottTu82} giving the higher differentials of the spectral sequence. Such a formula goes back it least to \cite{Shih62}.

\smallskip

\item[$\diamond$] The great advantage with this conceptual presentation is that one can prove the HTT for Koszul properads as well, using the very same arguments. The bar and cobar constructions were extended to the level of properads in \cite{Vallette07}. The only missing piece is a Van der Laan type morphism for graphs with genus instead of trees. Such a morphism is given in \cite{Vallette11bis}. This settles the HTT for homotopy (involutive) Lie bialgebras and for homotopy (involutive) Frobenius bialgebras, for instance.

\end{itemize}
\end{exams}

The quasi-ismorphism $i : H \qi A$ extends to an $\infty$-quasi-isomorphism $\tilde{\imath} : H \stackrel{\sim}{\rightsquigarrow} A$ defined by the same formula as the one giving the transferred structure but replacing the map $p$ labeling the root by the homotopy $h$.

\begin{rema}
The Homotopy Transfer Theorem should not be confused with the \emph{Homological Perturbation Lemma}. One can prove the HTT with it, see \cite{Berglund09}. The other way round, the HTT applied to the algebra of dual numbers gives the perturbation lemma. 
\end{rema}

\subsection{Homotopy theory of $\Po_\infty$-algebras}

\begin{theo}[Fundamental theorem of $\infty$-quasi-isomorphisms \cite{LodayVallette10}]\label{theo:InverseInftyQI}
If there exists an $\infty$-quasi-iso\-mor\-phi\-sm $ A \stackrel{\sim}{\rightsquigarrow} B$ between  two $\Po_\infty$-algebras, then there exists an $\infty$-quasi-isomorphism in the opposite direction $B\stackrel{\sim}{\rightsquigarrow} A$, which is the inverse of $H(A)\xrightarrow{\cong}H(B)$ on the level on holomogy. 
\end{theo}

This is the main property of $\infty$-quasi-isomorphisms, which does not hold for 
 quasi-isomorphisms of dg  $\Po$-algebras. These two notions of quasi-isomorphisms are related by the following property:  there exists a zig-zag of quasi-isomorphisms of dg $\Po$-algebras 
if and only there exists a direct $\infty$-quasi-isomorphism
 $$\exists \ A \stackrel{\sim}{\leftarrow} \bullet  \stackrel{\sim}{\rightarrow} 
\bullet  \stackrel{\sim}{\leftarrow} \bullet \cdots \bullet  \stackrel{\sim}{\rightarrow}  B 
\quad \Longleftrightarrow \quad \exists \ A \ \stackrel{\sim}{\rightsquigarrow} B \ . $$ 

\begin{defi}[Homotopy equivalence]
Two dg $\Po$-algebras (resp. two $\Po_\infty$-algebras) are \emph{homotopy equivalent} if they are related by a zig-zag of quasi-isomorphisms of dg $\Po$-algebras (resp. by an $\infty$-quasi-isomorphism). 
\end{defi}

The underlying homology groups $H(A)$ of any dg $\Po$-algebra $A$ carry a natural $\Po$-algebra structure. 
Moreover, the Homotopy Transfer Theorem~\ref{TransferThm} allows us to endow $H(A)$ with a $\Po_\infty$-algebra structure, with trivial differential, and which  extends this induced $\Po$-algebra structure. 

\begin{defi}[Operadic Massey products]
We call \emph{operadic Massey products} the operations making up this transferred $\Po_\infty$-algebra structure on the homology groups $H(A)$. 
\end{defi}

\begin{defi}[Formality]
A dg $\Po$-algebra $(A, d)$ is called \emph{formal} if  it is homotopy equivalent to 
the $\Po$-algebra $H(A)$ equipped with the induced structure: 
$$(A, d) \stackrel{\sim}{\leftarrow} \bullet  \stackrel{\sim}{\rightarrow} 
\bullet  \stackrel{\sim}{\leftarrow} \bullet \cdots \bullet  \stackrel{\sim}{\rightarrow}  (H(A), 0)  \ .$$
\end{defi}

\begin{prop}$Ê\ $
If the higher operadic Massey products vanish, then the dg $\Po$-algebra is formal.
\end{prop}

\begin{proof}
The proof  is a corollary of the HTT, Theorem~\ref{TransferThm}.
\end{proof}

In other words, when the higher operadic Massey products vanish, the homology $\Po$-algebra $H(A)$ has the same homotopy type as the initial dg $\Po$-algebra $A$. To study the general case, we will need the following result.

\begin{prop}[Rectification property \cite{DolgushevTamarkinTsygan07}]
Any $\Po_\infty$-algebra $A$ admits a dg  $\Po$-algebra, denoted $\mathrm{Rect}(A)$, which is  $\infty$-quasi-isomorphic to it
$$A \stackrel{\sim}{\rightsquigarrow} \mathrm{Rect}(A) \ . $$
\end{prop}

The rectification property together with the HTT provide us with two $\infty$-quasi-isomorphisms
$$\xymatrix{A  &   \ar@{~>}[l]_(0.6)\sim H(A)  \ar@{~>}[r]^\sim &  \mathrm{Rect}(A)} \ , $$
where $H(A)$ is considered as a $\Po_\infty$-algebra equipped with the operadic Massey products. Therefore, the data of the operadic Massey products on the homology groups allows us to recover the homotopy type of the initial dg $\Po$-algebra. So the Massey products faithfully encode the homotopy type of $A$. 

In the homotopy class of any $\Po_\infty$-algebra $(A, d)$, there is a $\Po_\infty$-algebra $H(A)$ with trivial differential, by the HTT, and there is a ``strict'' dg $\Po$-algebra $\mathrm{Rect}(A)$, by the Rectification property. In the first case, the underlying space is rather small but the algebraic structure is more complex. In the latter case, the underlying space is pretty big (one has to fatten $A$ to get $\mathrm{Rect}(A)$) but the algebraic structure is simpler. We informally call this phenomenon the \emph{Heisenberg uncertainty principle of homotopical algebra}: in general, homotopy classes of $\Po_\infty$-algebras cannot be represented by two ``small'' data, space and structure, at the same time. 
\\

One can define a suitable notion of \emph{$\infty$-homotopy}, denoted $\sim_h$, between $\infty$-morphisms, see \cite{Vallette11}, which allows us to state the following result. 

\begin{theo}[Homotopy category \cite{Vallette11}]
The following categories are equivalent
$$\boxed{\mathsf{Ho}(\mathsf{dg}\  \Po\textsf{-}\mathsf{alg})\  \cong \  \infty\mathsf{-}\Po_\infty\mathsf{-alg}/\sim_h    
\ 
 \cong \  \infty\textsf{-}\mathsf{dg}\  \Po\textsf{-}\mathsf{alg}/\sim_h     \   . }$$
\end{theo}

\begin{proof}
The first equivalence is proved by endowing the category of  dg $\Pac$-coalgebras with a model category structure.
The last equivalence with the category of dg  $\Po$-algebras equipped with the $\infty$-morphisms  is given by the Rectification property.
\end{proof}

\subsection{Applications}
In this section, we conclude with a non-exhaustive, but not so short,  list of fields where the aforementioned operadic homotopical algebra is used. 

\subsubsection*{$\quad$ \sc Algebra}$ \ $

\bigskip

\paragraph{$\diamond$ \it Chevalley-Eilenberg cohomology of Lie algebras}The cohomology groups of a Lie algebra 
$\g$, with coefficients in $S(\g)$
and the adjoint representation,  carry a homotopy $BV$-algebra structure  \cite{DrummondColeVallette10}. 
The cohomology of the  Lie algebra $L_{1}$ of polynomial vector fields over the line $\KK^1$ is  
 a $C_{\infty}$-algebra generated by $ H^1_{CE}(L_{1})$ \cite{Millionshchikov10}.
 
\medskip 

\paragraph*{$\diamond$ \it Bar construction} The bar construction for $A_\infty$-algebras carries a natural commutative $BV_\infty$-algebra structure  \cite{TTW10}. 

\medskip 

\paragraph*{$\diamond$ \it Vertex  algebras} The study of vertex algebras  \cite{Borcherds86} yields to BV-algebras. Lian and Zuckerman \cite{LianZuckerman93} conjectured that their structure of BV-algebra on the homology groups of a Topological Vertex Operator Algebra lifts to a homotopy BV-algebra structure on the TVOA itself. This was proved, with explicit formulae, in \cite{GCTV09}.

\medskip 

\paragraph*{$\diamond$  \it  (Cyclic) Deligne conjecture} The (cyclic) Deligne conjecture states that the (BV) Gersenhaber algebra 
structure on the Hochschild cohomology groups of an (cyclic unital) associative algebra can be lifted to a homotopy (BV) Gersenhaber algebra 
structure on the cochain level. This conjecture received several proofs in \cite{Tamarkin98, Voronov00, KontsevichSoibelman00, McClureSmith02, BergerFresse04,  Kaufmann04, Kaufmann07, TradlerZeinalian06, Costello07,  KontsevichSoibelman09}.

\subsubsection*{$\quad$  \sc Deformation theory, Quantum algebra, Noncommutative geometry} $ \ $

\bigskip

\paragraph*{$\diamond$ \it  Deformation-Quantization of Poisson manifolds} The first proof of the deformation-quantization of Poisson manifolds, given by Maxim Kontsevich in \cite{Kontsevich03, Weinstein95}, relies on an explicit $L_\infty$-quasi-isomorphism to prove a formality result. 

\medskip 

\paragraph*{$\diamond$ \it  Drinfeld associators, Grothendieck-Teichm\"uller group} The second proof of the deformation-quantization of Poisson manifolds, given by Dimitry Tamarkin in \cite{Tamarkin98}, relies on the homotopy properties  of $G_\infty$-algebras. It relies on the obstruction theory; so it is less explicit than Kontsevich's proof. It however unveils the links with the Drinfeld associators and the GT groups. 

\medskip

\paragraph*{$\diamond$ \it  Cyclic Homology} The cyclic homology of (unital) associative algebras can be computed using various bicomplexes, see \cite{Loday98}. This result can be made more precise using the HTT for mixed complexes, see \cite{Kassel90} where all these bicomplexes are shown to be $\infty$-quasi-isomorphic.

\subsubsection*{$\quad$ \sc Algebraic Topology.}$ \ $ 

\bigskip

\paragraph*{$\diamond$ \it  Massey products} The Massey products on the cohomology groups of a topological space are related to the operad $A_\infty$, see Section~\ref{subsec:Massey}. 

\medskip 

\paragraph*{$\diamond$ \it   Homotopy groups of spheres} The second page of the Adams spectral sequence which computes the homotopy groups of spheres is equal to the homology of the $\Lambda$ algebra, which is the Koszul dual dg algebra of the Steenrod algebra, see \cite{BCKQRS66}.  

\medskip 

\paragraph*{$\diamond$ \it  Spectral Sequences} We have seen in Section~\ref{subsec:HTT} that spectral sequences actually come from the HTT applied to the algebra of dual numbers.
The advantage of this point of view on bicomplexes versus spectral sequences, is that the HTT can be applied to bicomplexes equipped with a deformation retract whose boundary map is not necessarily trivial.
For instance, the cyclic bicomplex of a unital associative algebra, which involves the boundary maps $
b, b'$ and the cyclic operator, cf.\ \cite{LodayQuillen84, Loday98}, admits a deformation retract made up of the 
the columns involving only $b$. Applying the HTT to it gives  automatically Connes' boundary map $B$. So we recover the fact that, in cyclic homology theory, the $(b,B)$-bicomplex is quasi-isomorphic to the cyclic bicomplex.

\medskip 

\paragraph*{$\diamond$ \it  Iteration of the bar construction} The method proposed by Benoit Fresse in \cite{Fresse10bis} to iterate the bar construction for $E_\infty$-algebras relies on the HTT for a cofibrant $E_\infty$-operad. (An $E_\infty$-algebra is a homotopy version of a commuative algebra, where both the associativity relation and the symmetry property are relaxed up to higher homotopies). 

\medskip 

\paragraph*{$\diamond$ \it  Rational Homotopy Theory} Minimal models in Dennis Sullivan's approach to Rational Homotopy Theory \cite{Sullivan77} are quasi-free commutative algebras on the rational homotopy groups of  topological spaces endowed with an $L_\infty$-(co)algebra structure. In Dan Quillen's approach to RHT \cite{Quillen69}, the minimal models are quasi-free Lie algebras generated by the cohomology groups of the space endowed with a $C_\infty$-(co)algebra structure. 

\medskip 

\paragraph*{$\diamond$ \it  String Topology} String Topology is full of higher structures, to name  but a few: 
Batalin-Vilkovisky algebra in loop homology \cite{ChasSullivan99, CHV06} and (homotopy) involutive Lie bialgebra in (equivariant) string homology \cite{ChasSullivan04}. 

\subsubsection*{$\quad$ \sc Differential Geometry}$ \ $

\bigskip

\paragraph*{$\diamond$ \it  Lie algebroids}  The structure of Courant algebroids is shown to induce a  structure of homotopy Lie algebra by Dmitry Roytenberg and Alan Weinstein in \cite{RoytenbergWeinstein98}.

\medskip

\paragraph*{$\diamond$ \it  K\"ahler manifolds}  The dg commutative algebra of differential forms of a K\"ahler manifold is shown to be formal in \cite{DGMS75}. It was also proved in loc. cit. that this is equivalent to the uniform vanishing of the higher Massey products. 

\medskip 

\paragraph*{$\diamond$ \it  Lagrangian submanifolds, Floer Homology, Symplectic Field Theory} Structures of $A_\infty$-algebras \cite{Fukaya02}, $uA_\infty$-algebras \cite{FOOO09I, FOOO09II} and ${BiLie_\diamond}_\infty$-algebras \cite{EGH00, CFL11} play a crucial role in these fields. 

\medskip 

\paragraph*{$\diamond$ \it  F-manifold, Nijenhuis structure and Poisson manifold}   Sergei Merkulov developed in \cite{Merkulov04, Merkulov05, Merkulov06bis} a programme called ``operadic profiles'', which establishes a new link between differential geometry and higher algebra. He thereby describes 
several differential geometric structures (Hertling-Manin, Nijenhuis, and Poisson) in terms of homotopy algebras, i.e. algebras over operadic Koszul resolutions. 

\medskip 

\paragraph*{$\diamond$ \it  Poincar\'e duality} The study of the Poincar\'e duality of oriented closed manifolds is related to the homotopy (co)unital Frobenius bialgebra structure \cite{Wilson07, HirschMilles10, Miller11} on the differential forms of the manifold. 

\medskip 

\paragraph*{$\diamond$ \it  Fluid mechanics}  Dennis Sullivan proposed in \cite{Sullivan10} a programme to solve the Euler and Navier-Stockes equations, which describe fluid motion, using the HTT for $BV_\infty$-algebras. 

\subsubsection*{$\quad$ \sc Algebraic Geometry} $ \ $

\bigskip

\paragraph*{$\diamond$ \it  Moduli spaces of algebraic curves $\mathcal{M}_{g,n}$ and $\overline{\mathcal{M}}_{g,n}$,  Gromov-Witten invariants}  The two homology operads $H_\bullet(\overline{\mathcal{M}}_{0, n+1})$ and $H_\bullet(\mathcal{M}_{0, n+1})$ are Koszul dual to each other, see \cite{Getzler95}. 

\medskip 

\paragraph*{$\diamond$ \it  Frobenius manifolds, Quantum cohomology}  
The space $H^\bullet(\mathcal{M}_{0, n+1})$ provides the generators of the minimal model (quasi-free resolution without internal differential) of the operad $BV$, see \cite{DrummondColeVallette10}. This allows us to define the notion of a homotopy Frobenius manifold. 
Using the HHT for Lie algebras, S. Barannikov and M. Kontsevich  \cite{BarannikovKontsevich98}, and Y.I. Manin \cite{Manin99} proved that the underlying homology groups of some dg BV-algebras carry a Frobenius manifold structure, i.e. an algebra over $H_\bullet(\overline{\mathcal{M}}_{0, n+1})$.  Using the aforementioned minimal model for $BV$ and the HTT for homotopy BV algebras, we  endowed in \cite{DrummondColeVallette10} the homology groups with a homotopy Frobenius manifold structure, i.e. an algebra parametrized by $H^\bullet(\mathcal{M}_{0, n+1})$, extending the Barannikov-Kontsevich-Manin structure, cf. loc. cit. This latter structure keeps faithfully track of the homotopy type of the initial dg $BV$-algebra.
  
\medskip 

\paragraph*{$\diamond$ \it  Mirror symmetry conjecture} 
The Mirror Symmetry conjecture \cite{Kontsevich95} claims that the Fukaya ``category'' of Lagrangian submanifolds of a Calabi-Yau manifold $\mathcal M$ (A-side) should be equivalent to the bounded derived category of coherent sheaves on a dual Calabi-Yau manifold $\widetilde{\mathcal M}$ (B-side). The Fukaya ``category'' is actually an $A_\infty$-category. The relevant algebraic structure on the B-side is the homotopy Frobenius manifold stated in the previous point, where one starts for the dg BV-algebra of the Dolbeault cochain complex of the Calabi-Yau manifold $\widetilde{\mathcal M}$.

\subsubsection*{$\quad$ \sc Mathematical Physics}$ \ $

\bigskip 

\paragraph*{$\diamond$ \it  Feynman diagrams, Batalin-Vilkovisky formalism, Renormalization theory} 
One of the highlights of this operadic homotopy theory lies in the following result by Sergey Merkulov \cite{Merkulov10}. 
He proved that the Batalin-Vilkovisky formalism is equivalent to the HTT for unimodular Lie bialgebras. In other words, the classical Feynman diagrams are exactly the graphs appearing in the HTT formula for the wheeled properad encoding unimodular Lie bialgebras.

\medskip 

\paragraph*{$\diamond$ \it BRST cohomology} The BRST cohomology carries a $BV_\infty$-algebra structure \cite{GCTV09}.

\medskip 

\paragraph*{$\diamond$ \it  Field theories} The various types of fields theories (Topological Quantum Field Theory \cite{Atiyah88}, Conformal Field Theory \cite{Segal04}, String Field Theories \cite{WittenZwiebach92, Zwiebach93}, etc.) yield to various types of homotopy structures (
$L_\infty$ 
\cite{FulpLadaStasheff02}, 
$uFrob_\infty$ \cite{HirschMilles10}, $OCHA$ \cite{KajiuraStasheff06}, ${BiLie_\diamond}_\infty$ \cite{MuensterSachs11}, etc.). 

\subsubsection*{$\quad$ \sc Computer Science}$ \ $

\bigskip 

\paragraph*{$\diamond$ \it  Rewriting system} The rewriting method of Section~\ref{subsec:Rewriting} is strongly related to the notion of Rewriting system, see \cite{GuiraudMalbos09, GuiraudMalbos11} for instance.

\section*{Conclusion}

\subsection*{Open questions}Here are a few open and interesting questions in this field of research. 

\begin{enumerate}
\item Are the Koszul dual properads $Frob$ and $BiLie_\diamond$, encoding respectively Frobenius bialgebras and involutive Lie bialgebras, Koszul ? 

\item  There are $4$ functors making up the following commutative diagram 
$$\xymatrix{\textsf{modular} \ \textsf{operads} \ar[r]  \ar[d] & \textsf{properads}  \ar[d] \\ 
\textsf{cyclic} \ \textsf{operads} \ar[r]   & \textsf{operads} \ . \\   } $$
It was shown in \cite{Getzler95} that the two graded operads $H_\bullet(\overline{\mathcal{M}}_{0, n})$ and $H_\bullet({\mathcal{M}}_{0, n})$ are Koszul dual operads and Koszul operads. We conjecture that the two full genera properads associated to the modular operads 
$H_\bullet(\overline{\mathcal{M}}_{g, n+1})$ and $H_\bullet({\mathcal{M}}_{g, n+1})$
are Koszul dual properads and Koszul properads.

\item Develop the Koszul duality for operads in characteristic $p$. 
\end{enumerate}

\subsection*{Further reading} At this point, the interested reader is invited to pursue  its journey in   operad theory and  homotopical algebra, with the  book \cite{LodayVallette10}.
%, freely downloadable at: 
%\begin{center}
%\texttt{http://math.unice.fr/$\sim$brunov/Operads.pdf}. 
%\end{center}
The various answers to the exercises proposed here can be found in this reference.

\subsection*{Acknowledgements} This survey follows several talks given at the colloquia or seminars of the universities of Lyon, Dublin, Stony Brooke, G\"ottingen, Glasgow, Paris $6$, at the conference ``Higher Structures 2009'' in Z\"urich, and  at the Hayashibara forum on Symplectric Geometry, Non-commutative geometry and Mathematical Physics at MSRI.  The author would like to thank the audience, and especially Dennis Sullivan, for their questions and  comments, which helped to improve the present text. Many thanks to Olivia Bellier and Jim Stasheff for their corrections and  interesting  remarks on the first draft of the present paper.   
Last but not least, I would like to express my deep gratitude to the Max-Planck Institute f\"ur Mathematik in Bonn for the long term invitation and for the excellent working conditions.

%%%%%%%%%%%%%%%%%%%%%%%%%%%%%%%%%%%%%%%%%%%%%%%%%%%%%%%%%%%%%%%

\bibliographystyle{amsalpha}
\bibliography{bib}

\medskip

%{\small \textsc{Bruno Vallette, Laboratoire J.A. Dieudonn\'e,
%Universit\'e de Nice Sophia-Antipolis,} \\
%textsc{Parc Valrose, 06108 Nice
%Cedex 02, France}\\
%\& {\small \textsc{Max-Planck-Institut f\" ur Mathematik, Vivatsgasse 7,
%53111 Bonn, Deutschland.}

\smallskip

%E-mail address : \texttt{brunov@unice.fr}\\
%URL : \texttt{http://math.unice.fr/$\sim$brunov}}

\end{document}